\documentclass[a4paper]{article}

\usepackage[repositionpullbacks,midshaft,nohug,scriptlabels,size=2em]{diagrams}

\usepackage{graphicx}
\usepackage{color}
\usepackage[mathcal,mathscr]{euscript}
\usepackage{latexsym}
\usepackage{amssymb,amsmath}

\newcommand{\dashbk}{-}
\newcommand{\diagspace}{\mbox{\hspace{2em}}}
\newcommand{\bref}[1]{(\ref{#1})}
\newcommand{\ucontents}[2]{\addcontentsline{toc}{#1}{\numberline{}{#2}}}
\newcommand{\cat}[1]{\mathcal{#1}}
\newcommand{\fcat}[1]{\mathbf{#1}}
\newcommand{\ovln}[1]{\overline{#1}}
\newcommand{\twid}[1]{\widetilde{#1}}
\newcommand{\slsh}{/\linebreak[0]}
\newcommand{\dt}{.\linebreak[0]}
\newcommand{\such}{\:|\:}
\newcommand{\without}{\setminus}
\newcommand{\blob}{{\raisebox{.3ex}{\ensuremath{\scriptscriptstyle{\bullet}}}}}
\newcommand{\epsln}{\varepsilon}
\newcommand{\op}{\mathrm{op}}
\newcommand{\id}{\mathrm{id}}
\newcommand{\Cat}{\fcat{Cat}}
\newcommand{\One}{\fcat{1}}
\newcommand{\Set}{\fcat{Set}}
\newcommand{\Top}{\fcat{Top}}
\newcommand{\ftrcat}[2]{[#1,#2]}
\newcommand{\diso}{\sim}
\newarrow{Get}....>
\newarrow{Incl}C--->
\newarrow{Mod}--+->
\newarrow{Qt}---->
\newarrow{Modget}..+.>
\newcommand{\go}{\rTo\linebreak[0]}
\newcommand{\goby}[1]{\rTo^{#1}\linebreak[0]}
\newcommand{\og}{\lTo}
\newcommand{\ogby}[1]{\lTo^{#1}}
\newcommand{\goesto}{\,\longmapsto\,}
\newcommand{\goiso}{\goby{\raisebox{-0.5ex}[0ex][0ex]{$\diso$}}}
\newcommand{\oppair}[2]{\pile{\rTo^{\scriptstyle #1}\\ 
\lTo_{\scriptstyle #2}}} 
\newcommand{\parpair}[2]{\pile{\rTo^{\scriptstyle #1}\\ 
\rTo_{\scriptstyle #2}}}
\newcommand{\parpairu}{\pile{\rTo\\ \rTo}}
\newcommand{\oppairu}{\pile{\rTo\\ \lTo}}
\newcommand{\upr}[1]{[#1]}		
\newcommand{\reals}{\mathbb{R}}
\newcommand{\minihead}[1]{\subsubsection*{#1}}	
\newcommand{\nent}{\rotatebox{45}{$\Rightarrow$}}
\newcommand{\demph}[1]{\textbf{\textup{#1}}}
\newcommand{\done}{\hfill\ensuremath{\Box}}
\newenvironment{prooflike}[1]{\begin{trivlist}\item\textbf{#1}\ }
{\end{trivlist}}
\newenvironment{proof}{\begin{prooflike}{Proof}}{\end{prooflike}}
\newcommand{\scat}[1]{\mathbb{#1}}
\newcommand{\iso}{\cong}
\newcommand{\nat}{\mathbb{N}}	
\newcommand{\url}[1]{#1}
\newcommand{\pshf}[1]{\ftrcat{#1^\op}{\Set}}
\newcommand{\elt}[1]{\scat{E}\left(#1\right)}
\newcommand{\gomod}{\rMod\linebreak[0]}
\newcommand{\gobymod}[1]{\rMod^{#1}\linebreak[0]}
\newcommand{\ndSet}[1]{\langle #1, \Set \rangle}
\newcommand{\ndTop}[1]{\langle #1, \Top \rangle}
\newcommand{\littleproduct}{(\blob\ \ \ \blob)}
\newcommand{\littleequalizer}{(\blob \parpairu \blob)}
\newcommand{\littlepullback}{(\blob \go \blob \og \blob)}
\newcommand{\cstyle}[1]{\textbf{\upshape{#1}}}
\newcommand{\So}{\cstyle{S}}
\newenvironment{condition}{\begin{description}\item[\ \ \ \ \ \,]}{\end{description}}
\newcommand{\Coalg}[2]{\fcat{Coalg}(#1, #2)}
\newcommand{\frc}[2]{\frac{#1}{#2}}
\newcommand{\half}{\frc{1}{2}}
\newcommand{\dfrc}[2]{\textstyle{\frac{#1}{#2}}}
\newcommand{\dhalf}{\dfrc{1}{2}}
\newcommand{\compt}[1]{{[\, #1\, ]}}
\newcommand{\bigcompt}[1]{{[\, #1 \,]}}
\newcommand{\ob}{\mathrm{ob}}		
\newcommand{\pr}{\mathrm{pr}}
\newcommand{\diam}{\mathrm{diam}}
\newcommand{\Dface}{\Delta_{\mathrm{inj}}}
\newcommand{\propersub}{\subset}
\newcommand{\nepower}{\mathcal{P}_{\neq\emptyset}}
\newcommand{\vtr}[1]{#1}
\newcommand{\oba}{\mathrm{ob}\,}
\newcommand{\blb}{{\scriptscriptstyle{\bullet}}}
\newcommand{\sblb}{\cdot}
\newcommand{\ldlabel}[1]{\makebox[0em]{\hspace*{4.2em}\ensuremath{\scriptstyle#1}}}
\newcommand{\rdlabel}[1]{\makebox[0em]{\hspace*{-4.5em}\ensuremath{\scriptstyle#1}}}
\newcommand{\of}{\,\raisebox{0.08ex}{\ensuremath{\scriptstyle\circ}}\,}
\newcommand{\Cx}{\mathbb{C}}
\newcommand{\sub}{\subseteq}
\newcommand{\catI}{\mathcal{I}}
\newcommand{\ladj}{\dashv}
\newenvironment{diagdiag}
	{\begin{diagram}[size=1.5em]}
	{\end{diagram}}
\newcommand{\cell}[4]{\put(#1,#2){\makebox(0,0)[#3]{\ensuremath{#4}}}}
\newcommand{\zmark}{\scriptstyle{\bullet}}
\definecolor{grey}{gray}{0.5}
\newcommand{\avdots}{\rotatebox{90}{$\cdots$}}
\newcommand{\Colt}[1]{{\displaystyle\lim_{\rightarrow #1}}\,}
\newcommand{\Lt}[1]{{\displaystyle\lim_{\leftarrow #1}}\,}
\newcommand{\colt}[1]{\displaystyle\lim_{\rightarrow #1}}
\newcommand{\posint}{{\mathbb{N}^+}}
\newcommand{\Reso}{\fcat{Reso}}
\newcommand{\Flat}{\fcat{Flat}}
\newcommand{\cov}[1]{\mathcal{#1}}
\newcommand{\natset}[1]{\mathcal{N}_{#1}}

\newtheorem{thm}{Theorem}[section]
\newtheorem{propn}[thm]{Proposition}
\newtheorem{lemma}[thm]{Lemma}
\newtheorem{cor}[thm]{Corollary}
\newtheorem{conj}[thm]{Conjecture}
\newtheorem{predefn}[thm]{Definition}
\newenvironment{defn}{\begin{predefn}\upshape}{\end{predefn}}
\newtheorem{preexample}[thm]{Example}
\newenvironment{example}{\begin{preexample}\upshape}{\end{preexample}}
\newenvironment{example*}[1]{\begin{preexample}[#1]\upshape}{\end{preexample}}
\newtheorem{prewarning}[thm]{Warning}
\newenvironment{warning}{\begin{prewarning}\upshape}{\end{prewarning}}

\begin{document}

\sloppy


\title{%
A general theory of self-similarity}
\author{Tom Leinster%
\thanks{School of Mathematics and Statistics, University of Glasgow, Glasgow
G12 8QW, UK; Tom.Leinster@glasgow.ac.uk.  Supported by a Nuffield Foundation
Award NUF-NAL 04 and an EPSRC Advanced Research Fellowship.}}  \date{}

\maketitle
\thispagestyle{empty}

\begin{center}
\textbf{Abstract}
\end{center}
A little-known and highly economical characterization of the real interval
$[0, 1]$, essentially due to Freyd, states that the interval is homeomorphic
to two copies of itself glued end to end, and, in a precise sense, is
universal as such.  Other familiar spaces have similar universal properties;
for example, the topological simplices $\Delta^n$ may be defined as the
universal family of spaces admitting barycentric subdivision.  We develop a
general theory of such universal characterizations.

This can also be regarded as a categorification of the theory of simultaneous
linear equations.  We study systems of equations in which the variables
represent spaces and each space is equated to a gluing-together of the
others.  One seeks the universal family of spaces satisfying the equations.
We answer all the basic questions about such systems, giving an explicit
condition equivalent to the existence of a universal solution, and an explicit
construction of it whenever it does exist.

\bigskip


\tableofcontents


\section*{Introduction}
\ucontents{section}{Introduction}

Ask a mathematician for a definition of the topological space $[0, 1]$, and
they will probably define it as a subspace of the real line $\reals$.  Pushed
for a definition of $\reals$ itself, they might, with some reluctance, mention
its construction by Dedekind cuts or Cauchy sequences, or its characterization
as a complete ordered field.  The reluctance stems from the fact that in
everyday practice, most mathematicians do not think of real numbers as
Dedekind cuts or equivalence classes of Cauchy sequences; and while the
characterization as a complete ordered field is better from this point of
view, it involves far more structure than is relevant to the mere topology of
the interval.

There is, however, a simple characterization of the topological space $I = [0,
1]$ that reflects rather accurately how it is used in topology.  Roughly, it
says the following: if we define a \demph{path} in a space $S$ to be a
(continuous) map $I \go S$, then $I$ has exactly the structure needed in order
that paths can be composed, and it is universal as such.

Let us make this precise.  To speak of composition of paths, we first need to
know that every path has a starting point and a finishing point.  Whenever we
have a pair of paths, the first finishing where the second starts, we wish to
be able to compose them to form a new path.  These requirements correspond to
$I$ coming equipped with two basepoints, $0$ and $1$, and an
endpoint-preserving map to its `doubling'---the space obtained by taking two
copies of $I$ and gluing the second basepoint of the first to the first
basepoint of the second.  Moreover, the two basepoints are distinct and, as
singleton subsets, closed.

Let $\cat{D}$ be the category in which an object is a space equipped with two
distinct, closed basepoints and an endpoint-preserving map to its doubling;
then we have just observed that $I$, with some extra structure, is an object
of $\cat{D}$.  The characterization of $I$ is that it is, in fact, the
terminal object.  This is a topological version of a theorem of
Freyd~(\ref{thm:topologicalFreyd}).  It characterizes (an interval of) the
real numbers using only the extremely primitive concepts of continuity and
gluing.  

Other important spaces have similar characterizations.  For example, the
embodiment of the concept of convergent sequence is the space $\nat \cup
\{\infty\}$ (the one-point compactification of the discrete space $\nat$), in
the sense that a convergent sequence in an arbitrary space $S$ amounts to a
continuous map $\nat \cup \{\infty\} \go S$.  There is a precise sense in
which the pair $(X_1, X_2) = (\{\star\}, \nat \cup \{\infty\})$ is the
universal solution to the system of `equations'
\begin{eqnarray}
X_1     &\iso   &X_1      
\label{eq:univ-cgt-pt}  \\
X_2     &\iso   &X_1 + X_2.
\label{eq:univ-cgt-it}
\end{eqnarray}
(Here $\{\star\}$ is the one-point space and $+$ denotes coproduct or disjoint
union of spaces.)  Indeed, let $\cat{D}$ be the category in which an object
is a pair $(X_1, X_2)$ of topological spaces together with a pair of maps
$(X_1 \go X_1, X_2 \go X_1 + X_2)$; then the terminal object of $\cat{D}$ is
$(\{\star\}, \nat \cup \{\infty\})$.

Another example characterizes the standard topological simplices $\Delta^n$.
Let $\Dface$ be the category of totally ordered sets $\upr{n} = \{0, \ldots, n
- 1\}$ ($n \geq 0$) and order-preserving injections.  There is a functor $I:
\Dface \go \Top$ assigning to $\upr{n}$ the topological $n$-simplex
$\Delta^n$.  This functor $I$ is fundamental: by a stock categorical
construction it induces the adjunction $\Top \oppairu \Set^{\Dface^\op}$ on
which much of algebraic topology is built.  (The first functor here is the
singular semisimplicial set of a space, and the second is geometric
realization.)  And $I$ has a universal property similar in character to the
two already mentioned: it is the universal functor admitting the combinatorial
process of barycentric subdivision (Example~\ref{eg:rec-bary}).

The spaces mentioned so far are standard objects of classical algebraic
topology, but the same kind of universal characterization also captures some
non-classical spaces.  For example, there are similar characterizations of
certain fractals---spaces that seem to be the epitome of complexity, but turn
out to have simple universal properties.  Conjecturally~(\ref{conj:julia}),
this includes the Julia set of any complex rational function.

We use the term \emph{self-similarity} in a `global' sense.  The interval $[0,
1]$, for example, is called self-similar because it is homeomorphic to a
gluing-together of two copies of itself.  `Local' statements of
self-similarity say something like `almost any small pattern observed in one
part of the object can be observed throughout the object, at all scales'.
(See for instance Chapter~4 of Milnor~\cite{Mil}, where such statements are
made about Julia sets.)  Global statements say something like `the whole
object consists of several smaller copies of itself glued together'; more
generally, there may be a whole family of objects, each of which can be
described as several objects in the family glued together.  The purpose of
this paper is to develop a theory of global self-similarity.

Viewed from another angle, this is a theory of recursive decomposition.  Our
first example concerned a recursive decomposition of the real interval.  In
the second, the isomorphisms~(\ref{eq:univ-cgt-pt}) and~(\ref{eq:univ-cgt-it})
can be interpreted as a pair of mutually recursive type definitions, in the
sense of computer science.  Here we only study recursive characterizations of
sets and topological spaces.  It may be possible to extend the theory to
encompass other types of space, hence other types of recursive decomposition
or self-similarity: conformal, statistical, type-theoretic, and so on.

Another possibility is to develop the algebraic topology of self-similar
spaces, for which the usual homotopical and homological invariants are often
useless: in the case of a connected fractal subset of the plane, for example,
they only give us $\pi_1$, which is typically either infinite-dimensional or
trivial.  However, a characterization by a recursive system of equations is a
discrete description, and so might lead to useful invariants.

We set up the basic language in~\S\ref{sec:dsss} and~\S\ref{sec:sss}.  The
main aim there is to motivate the definitions of \emph{equational system} and
of \emph{universal solution} of an equational system.  Informally, an
equational system is a system of equations in which each variable,
representing a space, is equated to a colimit or gluing-together of the
others.  A universal solution of such a system is a solution with a particular
universal property.  For example, the result above states that the real
interval (equipped with some extra structure) is the universal solution of a
certain simple equational system.

With the language set up, the principal results of the rest of the paper can be
summarized~(\S\ref{sec:summary}).  These results completely answer all the
basic questions about existence, construction and recognition of the
universal solutions of equational systems.  

Category theory is essential here, for two reasons.  First, our spaces are to
be characterized by universal properties.  Second, the appropriate general
notion of `gluing' is the categorical notion of colimit.  Further categorical
concepts become needed, and are explained, as the theory develops.

\paragraph*{Related work}
Various other theories are related to this one.  Symbolic dynamics~\cite{LM}
seems most closely related to the case of \emph{discrete} equational
systems~(\S\ref{sec:dsss}).  Iterated function systems
\cite{Fal,Hut} are related, but differ crucially in that they take
place inside a fixed ambient space, whereas we are concerned with spaces in
the abstract; see Examples~\ref{eg:rec-Sierpinski} and~\ref{eg:rec-ifs}.

The motivating example for this work was the theorem of Freyd on the real
interval~\cite{Fre,FreMA}.  This in turn was inspired by a theorem of
Pavlovi\'c and Pratt~\cite{PP}.  Their results are part of a long line of work
on terminal coalgebras in computer science.  (See~\cite{FreACC}, for instance,
and~\cite{Ad} for a survey.) In that context, (co)recursively defined
data types occur as terminal coalgebras; they are a non-topological analogue
of our recursively decomposable spaces.  Freyd's Theorem stimulated other
related work, in particular that of Escard\'o and Simpson~\cite{ES}.
Escard\'o also obtained a topological version of Freyd's Theorem~\cite{Esc},
different from ours.

A paper of Barr has some obvious similarities to the present work~\cite{Barr}.
He discusses terminal coalgebras for an endofunctor and the metrics
associated with them.  However, the class of endofunctors that he considers
has little overlap with the class considered here.  The categories on which his
endofunctors act always have a terminal object, and his terminal coalgebras
can be constructed as limits; contrast~\ref{warning:fin} below.

Recent work of Karazeris, Matzaris and Velebil~\cite{KMV} builds on the work
here, giving new theorems, and new proofs of old theorems, in the general
theory of categorical coalgebras.

A short survey of the work contained in this paper is available~\cite{GSSO}.

\paragraph*{Notation and terminology}  
The sum (coproduct) of a family $(X_i)_{i \in I}$ of objects of a category
is written $\sum_i X_i$.  If $X_i = X$ for all $i$ then the sum is written
$I \times X$.  The sum of a finite family $X_1, \ldots, X_n$ of objects is
written as $X_1 + \cdots + X_n$, or as $0$ if $n = 0$.

Given categories $\cat{A}$ and $\cat{B}$, the category whose objects are
functors from $\cat{A}$ to $\cat{B}$ and whose morphisms are natural
transformations is written $\ftrcat{\cat{A}}{\cat{B}}$.  

$\Top$ is the category of \emph{all} topological spaces and continuous
maps.

A \demph{discrete} category is one in which the only maps are the
identities.  Small discrete categories are therefore just sets.  A
\demph{finite} category is one with only finitely many maps (hence only
finitely many objects).  A category is \demph{connected} if it is nonempty and
cannot be written as the coproduct of two nonempty categories.

The set $\nat$ of natural numbers is taken to include $0$.

The cardinality of a finite set $S$ is denoted $|S|$.

We will use extensively the language of modules (in the sense of category
theory), also called bimodules, profunctors or distributors \cite{Ben,Law}.
An introduction to modules can be found in Appendix~\ref{app:modules}; here we
just state the basic conventions.

Given categories $\scat{A}$ and $\scat{B}$, a \demph{module}
\[
M: \scat{B} \gomod \scat{A}
\]
is a functor $M: \scat{B}^\op \times \scat{A} \go \Set$.  (When $\scat{A}$ and
$\scat{B}$ are monoids construed as one-object categories, such a module $M$
is a set with a compatible left $\scat{A}$-action and right
$\scat{B}$-action.)  For objects $a \in \scat{A}$ and $b \in \scat{B}$, we
write
\[
m: b \gomod a
\]
to mean $m \in M(b, a)$.  Thus, a module $M: \scat{B} \gomod \scat{A}$ is an
indexed family $(M(b, a))_{b \in \scat{B}, a \in \scat{A}}$ of sets together
with actions:
\begin{eqnarray*}
b \gobymod{m} a \goby{f} a'	&
\ \textrm{gives}	&\ 
b \gobymod{fm} a',	\\
b' \goby{g} b \gobymod{m} a	&
\ \textrm{gives}	&\ 
b' \gobymod{mg} a.	
\end{eqnarray*}
These are required to satisfy axioms: $(f'f)m = f'(fm)$, $1m = m$, and dually;
and $(fm)g = f(mg)$. 

A functor $X: \scat{A} \go \Set$ can be viewed as a module $\One \gomod
\scat{A}$, where $\One$ denotes the category with one object and only the
identity arrow.  In this special case, the `$fm$' notation above becomes the
following: given an arrow $f: a \go a'$ in $\scat{A}$ and an element $x \in
X(a)$, we write $fx$ for the element $(Xf)(x) \in X(a')$.  Similar notation
(`$yg$') is used for contravariant functors $Y: \scat{B}^\op \go \Set$.  

We will also use commutative diagrams involving crossed arrows $\gomod$,
as explained in Appendix~\ref{app:modules}.

\paragraph*{Acknowledgements}

I am very grateful for the encouragement and help of Clemens Berger, Araceli
Bonifant, Edward Crane, Marcelo Fiore, Panagis Karazeris, Steve Lack,
Apostolos Matzaris, John Milnor, Mary Rees, Justin Sawon, Carlos Simpson, Ivan
Smith, Ji\v{r}\'\i\ Velebil, and most especially Andr\'e Joyal, who made many
insightful and detailed comments at several stages of this work, and very
kindly suggested some new results
(\ref{lemma:flat-rep-old}--\ref{eg:Freyd-fam-rep} and
\ref{thm:can-rep}--\ref{propn:cofilt-complexes}).

This research has relied crucially on the categories mailing list
(see~\cite{Fre}); I thank Bob Rosebrugh, who runs it.  I am also very grateful
to Jon Nimmo for creating Figure~\ref{fig:Julia}.  The commutative diagrams
were made using Paul Taylor's macros.

This work was supported by a Nuffield Foundation award NUF-NAL 04 and an EPSRC
Advanced Research Fellowship.


\section{Discrete equational systems}
\label{sec:dsss}

We work our way up to the concept of equational system by first considering an
important special case, discrete equational systems.  It illustrates many
aspects of the general case, but in a simpler setting.  

A discrete equational system can be thought of as a system of linear equations
such as
\begin{eqnarray}
x_1     &=      &2x_1 + 5x_2 + x_3      
\label{eq:des-arb-1}    \\
x_2     &=      &x_2                    
\label{eq:des-arb-2}    \\
x_3     &=      &4x_1 + x_2.            
\label{eq:des-arb-3}
\end{eqnarray}
Better, it can be thought of as a \emph{categorification} of such a system:
the variables $x_i$ represent spaces, addition is coproduct, and the
equalities are really isomorphisms.  General equational systems can also be
thought of as a categorification of such systems of equations---but a more
subtle one.

We introduce discrete equational systems using two examples.

\minihead{The Cantor set}

The \demph{Cantor set} is the topological space $2^\posint$, that is, the
product $2 \times 2 \times \cdots$ of countably infinitely many copies of the
discrete two-point space $2 = \{0, 1\}$.  (Here $\posint$ is the set $\{1, 2,
\ldots\}$ of positive integers.)  The Cantor set is often regarded as a subset
of the real interval $[0, 1]$ via the embedding
\[
(m_n)_{n \geq 1}
\goesto
\sum_{n \geq 1}
2m_n \cdot 3^{-n}
\]
($m_n \in \{0, 1\}$), but here we will only consider it as an abstract
topological space.

The Cantor set satisfies an `equation': $2^\posint = 2^\posint + 2^\posint$.
More precisely, there is a canonical isomorphism
\[
\iota: 2^\posint \goiso 2^\posint + 2^\posint,
\]
where $\iota(0, m_2, m_3, \ldots)$ is the element $(m_2, m_3, \ldots)$ of the
first copy of $2^\posint$, and $\iota(1, m_2, m_3, \ldots)$ is the element
$(m_2, m_3, \ldots)$ of the second copy of $2^\posint$.  The pair $(2^\posint,
\iota)$ has, moreover, a universal property: it is terminal among all pairs
$(X, \xi)$ where $X$ is a topological space and $\xi: X \go X + X$ is any
(continuous) map.  In other words, for any such pair $(X, \xi)$ there is a
unique map $\ovln{\xi}: X \go 2^\posint$ such that the square
\[
\begin{diagram}
X                       &\rTo^\xi       &X + X                          \\
\dTo<{\ovln{\xi}}       &               &\dTo>{\ovln{\xi} + \ovln{\xi}} \\
2^\posint               &\rTo_\iota     &2^\posint + 2^\posint          \\
\end{diagram}
\]
commutes.  This can easily be verified directly; it is also a very special
case (Example~\ref{eg:rec-terminal}) of the theory developed in this paper.  

Some terminology will allow us to express this universal property more
succinctly.  

\begin{defn}
Let $\cat{C}$ be a category and $G$ an endofunctor of $\cat{C}$ (that is, a
functor $\cat{C} \go \cat{C}$).  A \demph{$G$-coalgebra} is a pair $(X, \xi)$
where $X \in \cat{C}$ and $\xi: X \go G(X)$.  A \demph{map} $(X, \xi) \go (X',
\xi')$ of $G$-coalgebras is a map $X \go X'$ in $\cat{C}$ such that the
evident square commutes.
\end{defn}

\begin{example}
Let $\cat{C}$ be the category of modules over some commutative ring, and let
$G$ be the endofunctor defined by $G(X) = X \otimes X$.  Then a $G$-coalgebra
is a (not necessarily coassociative) coalgebra in the algebraists' sense. 
\end{example}

Now let $\cat{C}$ be the category $\Top$ of topological spaces, and let $G$ be
the endofunctor defined by $G(X) = X + X$.  A $G$-coalgebra is a space $X$
together with a map $\xi: X \go X + X$.  The universal property of the Cantor
set is that $(2^\posint, \iota)$ is the \demph{terminal coalgebra}, that is,
the terminal object in the category of coalgebras.

In our example, the structure map $\iota$ of the terminal coalgebra is an
isomorphism.  This is not coincidence, as the following elementary result
reveals.  

\begin{lemma}[Lambek~\cite{Lam}]
\label{lemma:Lambek}
Let $\cat{C}$ be a category and $G$ an endofunctor of $\cat{C}$.  If $(I,
\iota)$ is terminal in the category of $G$-coalgebras then $\iota: I \go
G(I)$ is an isomorphism.
\done
\end{lemma}

A $G$-coalgebra $(X, \xi)$ in which $\xi: X \go G(X)$ is an isomorphism is
called a \demph{fixed point} of $G$.

\minihead{Spaces of walks}

We turn now to a different topological object with a different universal
property.  Consider walks on the natural numbers, of the following type:
\begin{itemize}
\item start at some position $n$
\item with each tick of the clock, take one step left or one step
right---unless at position $0$, in which case stay there
\item continue forever.
\end{itemize}
(One might consider imposing a different rule at $0$; see
Example~\ref{eg:walks-newrule}.)

Let $W_n$ be the set of all walks starting at position $n$.  Formally, $W_n$
is the set of elements $(a_0, a_1, \ldots) \in \nat^\nat$ such that $a_0 = n$
and for all $r \in \nat$, \emph{either} $a_r > 0$ and $a_{r + 1} \in \{a_r -
1, a_r + 1\}$, \emph{or} $a_r = a_{r + 1} = 0$.  There is a (profinite)
topology on $W_n$ generated by taking, for each $n, a_0, \ldots, a_n \in
\nat$, the set of all walks beginning $(a_0, \ldots, a_n)$ to be closed.  So
we have a family $(W_n)_{n \in \nat}$ of spaces, and this is the `topological
object' that we will characterize by a universal property.

First note that the spaces $W_n$ satisfy some `equations', or rather,
isomorphisms.  A walk starting at position $n > 0$ consists of either a step
left followed by a walk starting at $n - 1$, or a step right followed by a
walk starting at $n + 1$.  Thus, there is a canonical isomorphism
\[
\iota_n: W_n \goiso W_{n - 1} + W_{n + 1}
\]
for each $n > 0$.  Similarly, a walk starting at position $0$ consists of a
null step followed by another walk starting at $0$, so there is a canonical
isomorphism
\[
\iota_0: W_0 \goiso W_0.
\]
(In fact, $W_0$ is the one-point space, so $\iota_0$ is the identity.)

These isomorphisms can be expressed as follows.  The family $W = (W_n)_{n \in
\nat}$ is an object of the category $\cat{C} = \Top^\nat$ of sequences of
spaces.  There is an endofunctor $G$ of $\cat{C}$ defined by
\begin{equation}
\label{eq:walk-endofunctor}
(G(X))_n
=
\left\{
\begin{array}{ll}
X_{n - 1} + X_{n + 1}   &
\textrm{if } n > 0      \\
X_0                     &
\textrm{if } n = 0
\end{array}
\right.
\end{equation}
($X \in \cat{C}$, $n \in \nat$).  We have just observed that there is a
canonical isomorphism $\iota: W \goiso G(W)$; that is, $(W, \iota)$ is a fixed
point of $G$.  The universal property is that $(W, \iota)$ is the terminal
$G$-coalgebra.  Again, this can be proved directly and follows from later
theory.

(Of the many types of walk that could be considered, this one is of special
interest: in a certain sense, the sequence $(W_n)_{n \geq 1}$ has period $6$.
See~\cite{PSW} and compare~\cite{Blass} and~\cite{OCCN}.)

\minihead{Abstractions}

In both of our examples, we characterized a topological object as the terminal
coalgebra for an endofunctor.  But our two examples have further features in
common.  We now record those features and abstract, arriving at notions of
`discrete equational system' and `universal solution' of such a system.

In the Cantor set example, $\cat{C} = \Top$, and in the walks example,
$\cat{C} = \Top^\nat$.  In both, then, $\cat{C} = \Top^A$ for some set $A$.
We write objects of $\Top^A$ as indexed families $(X_a)_{a \in A}$.

In the Cantor set example, the functor $G: \cat{C} \go \cat{C}$ is defined by
$G(X) = X + X$, and in the walks example, $G$ is defined
by~(\ref{eq:walk-endofunctor}).  In both, $G$ has the following property: for
each $a \in A$, the space $(G(X))_a$ is a finite sum of spaces $X_b$ ($b \in
A$).  More precisely, there is a family $(M_{b, a})_{b, a \in A}$ of natural
numbers such that for all $X \in \Top^A$ and $a \in A$,
\[
(G(X))_a 
=
\sum_{b \in A}
M_{b, a} \times X_b.
\]
These are \emph{finite} sums, that is, $\sum_{b \in A} M_{b, a} < \infty$ for
all $a \in A$.  It makes no difference for now if we take $M_{b, a}$ to be a
finite set rather than a natural number, and for reasons of functoriality that
emerge later, it will be better if we do so.

Thus, in both examples the category $\cat{C}$ and the endofunctor $G$ are
determined by a set $A$ and a matrix of sets $M = (M_{b, a})_{b, a \in A}$.
This suggests the following definition.

\begin{defn}
\label{defn:des}
A \demph{discrete equational system} is a pair $(A, M)$ where $A$ is a set and
$M$ is a family $(M_{b, a})_{b, a \in A}$ of sets such that for each $a \in
A$, the disjoint union $\sum_{b \in A} M_{b, a}$ is finite.  
\end{defn}

Let $(A, M)$ be a discrete equational system and let $\cat{E}$ be a category
with finite sums.  Then there is an endofunctor $M \otimes \dashbk$ of
$\cat{E}^A$ defined by
\begin{equation}
\label{eq:des-tensor}
(M \otimes X)_a
=
\sum_{b \in A} M_{b, a} \times X_b
\in 
\cat{E}
\end{equation}
($X \in \cat{E}^A$, $a \in A$).  So far we have taken $\cat{E} \in \Top$; the
only other case with which we will be concerned is $\cat{E} = \Set$.

\begin{example*}{One-variable systems}
\label{eg:des-des-one}
A discrete equational system $(A, M)$ in which $A$ is a one-element set
amounts to just a finite set $M$.  If $M$ has $n$ elements then the induced
endofunctor $M \otimes \dashbk$ of $\Top$ is $X \goesto n \times X$.  In the
Cantor set example, $n = 2$.
\end{example*}

\begin{example*}{Walks}
\label{eg:des-des-walks}
The walks example corresponds to the discrete equational system $(A, M)$ in
which $A = \nat$ and 
\[
|M_{b, a}|
=
\left\{
\begin{array}{ll}
1       &\textrm{if } a > 0 \textrm{ and } b = a \pm 1  \\
1       &\textrm{if } a = b = 0 \\
0       &\textrm{otherwise}
\end{array}
\right.
\]
($b, a \in \nat$).  The induced endofunctor $M \otimes \dashbk$ is exactly the
functor $G$ defined earlier.
\end{example*}

In general, a discrete equational system can be viewed as a system of
simultaneous equations using only addition, such as
\begin{eqnarray*}
x_0     &=      &x_0                    \\
x_n     &=      &x_{n - 1} + x_{n + 1} 
\hspace*{4em} (n \in \posint)   
\end{eqnarray*}
(the walks example), or equations (\ref{eq:des-arb-1})--(\ref{eq:des-arb-3})
above.  Formally, equations (\ref{eq:des-arb-1})--(\ref{eq:des-arb-3})
correspond to the discrete equational system $(A, M)$ in which $A = \{1, 2,
3\}$ and $M$ is the transpose of the matrix of coefficients on the right-hand
side: $M_{1, 1} = 2$, $M_{2, 1} = 5$, and so on.

A discrete equational system $(A, M)$ can also be viewed as a graph.  Call an
element $m \in M_{b, a}$ a \demph{sector} of type $b$ in $a$, and write $m: b
\gomod a$.  Then there is one sector of type $b$ in $a$ for each copy of $X_b$
appearing in the expression~(\ref{eq:des-tensor}) for $(M \otimes X)_a$.  The
(directed) graph corresponding to $(A, M)$ has the elements of $A$ as its
vertices and the sectors as its edges (Figure~\ref{fig:graphs}).
\begin{figure}
\hfill
\includegraphics[width=4.7em]{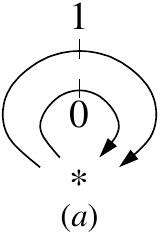}
\hfill
\includegraphics[width=22.3em]{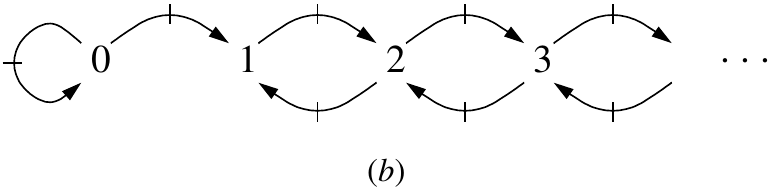}
\hfill
\caption{Graphs corresponding to the discrete equational systems for~(a) the
Cantor set (Example~\ref{eg:des-des-one}) and~(b) the spaces of walks
(Example~\ref{eg:des-des-walks}).  In~(b), one walks \emph{backwards} along
the arrows}  
\label{fig:graphs}
\end{figure}
The finiteness condition on $M$ is that each $a \in A$ contains only finitely
many sectors, or equivalently that each vertex is at the head of only finitely
many edges.

The universal properties of the Cantor set and the spaces of walks will be
expressed in the following terms.

\begin{defn}
Let $(A, M)$ be a discrete equational system and $\cat{E}$ a category with
finite sums.  An \demph{$M$-coalgebra} (in $\cat{E}$) is a coalgebra for the
endofunctor $M \otimes \dashbk$ of $\cat{E}^A$.  A \demph{universal solution}
of $(A, M)$ (in $\cat{E}$) is a terminal $M$-coalgebra. 
\end{defn}

\begin{example*}{Cantor set}
Take the discrete equational system $(A, M)$ of Example~\ref{eg:des-des-one},
with $n = 2$.  The universal solution is the Cantor set $2^\posint \in \Top$
together with the canonical isomorphism 
\[
\iota: 2^\posint \goiso 2 \times 2^\posint = M \otimes 2^\posint.
\]
\end{example*}

\begin{example*}{Walks}
Take the discrete equational system $(A, M)$ of
Example~\ref{eg:des-des-walks}.  The universal solution is $W = (W_n)_{n \in
\nat} \in \Top^\nat$ together with the canonical isomorphism $\iota: W \goiso
G(W) = M \otimes W$.
\end{example*}

Universal solutions are evidently unique (up to canonical isomorphism) when
they exist.  The word `solution' is justified by Lambek's
Lemma~(\ref{lemma:Lambek}): if $(I, \iota)$ is a universal solution then $I
\iso M \otimes I$.  The converse, however, fails: for any discrete equational
system $(A, M)$, the empty family $0 = (\emptyset)_{a \in A} \in \Top^A$
satisfies $0 \iso M \otimes 0$, but it is not usually the universal solution.  

When $\cat{E}$ is $\Set$ or $\Top$, or more generally if $\cat{E}$ has enough
limits, every discrete equational system has a universal solution.  This can
be constructed as follows.  

Let $(A, M)$ be a discrete equational system.  For each $a \in A$, let $I_a$
be the set of all infinite sequences
\[
\cdots \gobymod{m_3} a_2 \gobymod{m_2} a_1 \gobymod{m_1} a_0 = a
\]
of sectors.  Then $I \in \Set^A$.  For each $a \in A$ we have 
\begin{eqnarray*}
(M \otimes I)_a         &=      &
\sum_{b \in A} M_{b, a} \times I_b      \\
        &\iso   &
\{ 
\textrm{diagrams }
\cdots \gobymod{p_2} b_1 \gobymod{p_1} b_0 = b \gobymod{m} a
\},
\end{eqnarray*}
so there is a canonical isomorphism $\iota_a: I_a \goiso (M \otimes I)_a$.
This defines an $M$-coalgebra $(I, \iota)$.  It can be verified directly, and
follows from a more general result (Theorem~\ref{thm:universalsolutioninSet}),
that $(I, \iota)$ is the universal solution of $(A, M)$ in $\Set$.

Moreover, each set $I_a$ carries a natural topology, generated by declaring
that for each finite diagram
\begin{equation}
\label{eq:fin-complex}
a_n \gobymod{m_n} \ \cdots \ \gobymod{m_1} a_0 = a,
\end{equation}
the set of all elements of $I_a$ ending in~(\ref{eq:fin-complex}) is closed.
(Denoting by $(I_n)_a$ the set of all diagrams~(\ref{eq:fin-complex}), we have
$I_a = \Lt{n} (I_n)_a$, and this is the corresponding profinite topology.)
Thus, $(I, \iota)$ becomes a coalgebra in $\Top$.  Again, it can be shown
directly and follows from a later result
(Theorem~\ref{thm:universalsolutioninTop}) that this is the universal solution
in $\Top$.  

So, any discrete equational system specifies a family of spaces, its universal
solution.  But as a tool for specifying spaces, this has severe limitations:
for as the construction of the universal solution $(I, \iota)$ reveals, the
spaces $I_a$ are always totally disconnected.  This is a consequence of the
fact that $I_a$ is isomorphic to a \emph{disjoint} union of the other spaces
$I_b$:
\begin{equation}
\label{eq:discrete-I-eqn}
I_a 
\iso 
(M \otimes I)_a
=
\sum_{b \in A,\ m: b \gomod a}
I_b.
\end{equation}
To specify spaces that are not totally disconnected, we will need to use
non-disjoint unions, that is, glue the spaces $I_b$ together in some
nontrivial way.  This is the step up from discrete equational systems to
general equational systems.


\section{Equational systems}
\label{sec:sss}

A discrete equational system is a system of equations specifying each member
of a family of spaces as a disjoint union of some of the others.  A (general)
equational system is the same, except that we are no longer confined to
\emph{disjoint} unions: any kind of union, or gluing, is permitted.  Formally,
this is the generalization from coproducts to colimits.

However, the process of generalization is not totally straightforward.  In the
general setting there are subtleties that were invisible in the discrete case,
as we shall see.

The definitions are introduced by way of two examples.

\minihead{The real interval}

In 1999, Peter Freyd~\cite{Fre} found a new characterization of the real
interval $[0, 1]$.  The interval is isomorphic to two copies of itself joined
end to end, and Freyd's theorem says that it is universal as such.

The result of joining two copies of $[0, 1]$ end to end is naturally described
as the interval $[0, 2]$, and then multiplication by $2$ gives a bijection
$[0, 1] \go [0, 2]$, which may be written as
\begin{equation}
\label{eq:interval-iso-it}
\iota_1: 
\begin{array}{c}
\setlength{\unitlength}{1em}
\begin{picture}(3,0.5)(-1.5,-0.25)
\put(-1.3,0){\line(1,0){2.6}}
\cell{-1.3}{0}{c}{\zmark}
\cell{1.3}{0}{c}{\zmark}
\end{picture}
\end{array}
\goiso
\begin{array}{c}
\setlength{\unitlength}{1em}
\begin{picture}(5.6,0.5)(-2.8,-0.25)
\put(0,0){\line(1,0){2.6}}
\put(0,0){\line(-1,0){2.6}}
\cell{-2.6}{0}{c}{\zmark}
\cell{0}{0}{c}{\zmark}
\cell{2.6}{0}{c}{\zmark}
\end{picture}
\end{array}.
\end{equation}
The one-point space plays a role here, since that is what we are gluing
along.  For reasons that will become apparent, let us write
\begin{equation}
\label{eq:interval-iso-pt}
\iota_0: \ \ \zmark\  \goiso \ \zmark\ 
\end{equation}
for the identity on the one-point space.

Now let $\cat{C}$ be the category whose objects are diagrams $X_0
\parpair{u}{v} X_1$ where $X_0$ and $X_1$ are sets and $u$ and $v$ are
injections with disjoint images.  (For now we consider only sets; we consider
spaces later.)  An object $X = (X_0, X_1, u, v)$ of $\cat{C}$ can be drawn as
\[
\includegraphics[width=9.0em]{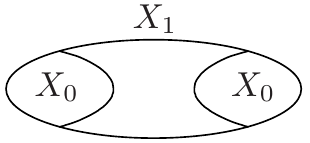}
\]
where the copies of $X_0$ on the left and the right are the
images of $u$ and $v$ respectively.  A map $X \go X'$ in
$\cat{C}$ consists of functions $X_0 \go X'_0$ and $X_1 \go X'_1$
making the evident two squares commute.  

Given $X \in \cat{C}$, we can form a new object $G(X)$ of $\cat{C}$ by gluing
two copies of $X$ end to end:
\begin{equation}
\label{eq:Freyd-gluing}
\begin{array}{c}
\includegraphics[width=14.3em]{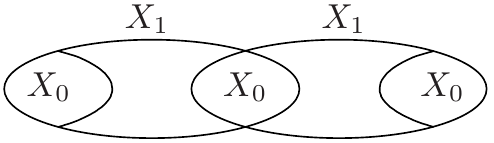}
\end{array}
.
\end{equation}
Formally, the endofunctor $G$ of $\cat{C}$ is defined by pushout:
\begin{equation}
\label{eq:Freyd-pushout}
\begin{diagdiag}
	&	&	&	&(G(X))_1&	&	&	&	\\
	&	&	&\ruTo	&	&\luTo	&	&	&	\\
	&	&X_1	&	&\textrm{pushout} &
						&X_1	&	&	\\
	&\ruTo<u&	&\luTo>v&	&\ruTo<u&	&\luTo>v&	\\
(G(X))_0 = X_0&	&	&	&X_0	&	&	&	&X_0.	\\
\end{diagdiag}
\end{equation}
For example, the unit interval with its endpoints distinguished forms an
object 
\[
I = \left( \{\star\} \parpair{0}{1} [0,1] \right)
\]
of $\cat{C}$, and
\[
G(I) 
= 
\left( \{\star\} \parpair{0}{2} [0,2] \right).
\]
So there is a coalgebra structure $\iota: I \goiso G(I)$ on $I$ given
by~(\ref{eq:interval-iso-it}) and~(\ref{eq:interval-iso-pt}).  

\begin{thm}[Freyd]	\label{thm:Freyd}
$(I, \iota)$ is the terminal $G$-coalgebra.
\end{thm}

This follows from a later result (Example~\ref{eg:rec-Freyd}).  A direct proof
is not hard either, and runs roughly as follows.  Take a $G$-coalgebra $(X,
\xi)$ and an element $x_0 \in X_1$.  Then $\xi(x_0) \in (G(X))_1$ is in either
the left-hand or the right-hand copy of $X_1$, so gives rise to a binary digit
$m_1 \in \{0, 1\}$ and a new element $x_1 \in X_1$.  (If $\xi(x_0)$ is in the
intersection of the two copies of $X_1$, choose left or right arbitrarily.)
Iterating gives a binary representation $0.m_1 m_2 \ldots$ of an element of
$[0, 1]$, and this is the image of $x_0$ under the unique coalgebra map $(X,
\xi) \go (I, \iota)$.

In the definition of $\cat{C}$, the condition that the maps $u, v: X_0 \go
X_1$ are injective with disjoint images is essential.  Without it, the theorem
would degenerate entirely: the terminal coalgebra would be $(\{\star\}
\parpairu \{\star\})$.  As we shall see, this condition is a form of flatness.
It is the source of most of the new subtleties in the non-discrete case.

There is also a topological version of Freyd's theorem.  Let $\cat{C}'$ be the
category whose objects are diagrams $X_0 \parpair{u}{v} X_1$ of topological
spaces and continuous closed injections with disjoint images, and whose maps
are pairs of continuous maps making the evident squares commute.  (A map of
topological spaces is \demph{closed} if the direct image of every closed
subset is closed.)  Define an endofunctor $G'$ of $\cat{C}'$ by the same
pushout diagram~(\ref{eq:Freyd-pushout}) as before.  Define a $G'$-coalgebra
$(I, \iota)$ as before, with the Euclidean topology on $[0, 1]$.

\begin{thm}[Topological Freyd]
\label{thm:topologicalFreyd}
$(I, \iota)$ is terminal in the category of $G'$-coalgebras.
\end{thm}
This is proved in Example~\ref{eg:rec-Freyd}.  The importance of the condition
that $u$ and $v$ are closed is that without it, the terminal coalgebra would
be given by the indiscrete topology on $[0, 1]$.

\minihead{A Julia set}

The second example concerns the Julia set of a certain rational function
(Figure~\ref{fig:Julia}(a)).
\begin{figure}
\centering
\includegraphics[width=34.3em]{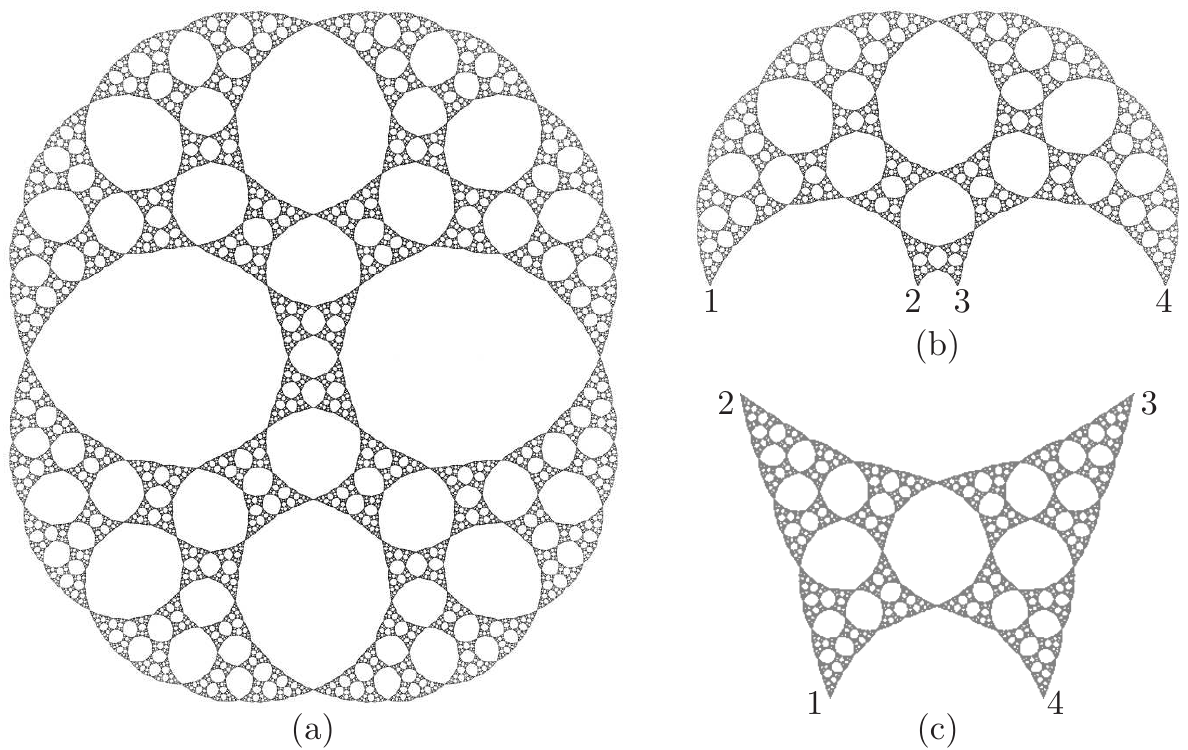}
\caption{(a) The Julia set of $(2z/(1+z^2))^2$; (b),~(c) two subsets,
  rescaled.  (Image by Jon Nimmo; see also~\cite[Fig.~2]{Mil}
  and~\cite[Fig.~53]{PR})}
\label{fig:Julia}
\end{figure}
Since the sole purpose of the example is to motivate the definitions, we will
not need the definition of Julia set, and we will proceed informally.  The
background is that every holomorphic map $f: S \go S$ on a Riemann surface $S$
has a Julia set $J(f) \sub S$; it is the part of $S$ on which $f$ behaves
unstably under iteration.  The best-explored case is where $S$ is the Riemann
sphere $\Cx \cup \{\infty\}$ and $f$ is a rational function with complex
coefficients.  In this case, $J(f)$ is a closed subset of $\Cx \cup
\{\infty\}$, and is almost always fractal in nature.

Figure~\ref{fig:Julia}(a) shows the Julia set of the function $z \goesto (2z /
(1 + z^2))^2$.  Write $I_1$ for this Julia set, regarded as an abstract
topological space.  Evidently $I_1$ has reflectional symmetry in a horizontal
axis, so may be written
\begin{equation}
\label{eq:Julia1}
I_1 \iso
\begin{array}{c}
\includegraphics[width=5.2em]{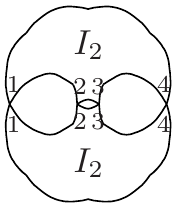}
\end{array}
\end{equation}
where $I_2$ is a certain space with 4 distinguished points, shown in
Figure~\ref{fig:Julia}(b).  In turn, $I_2$ may be regarded as a
gluing-together of subspaces:
\begin{equation}
\label{eq:Julia2}
I_2 \iso
\begin{array}{c}
\includegraphics[width=14.8em]{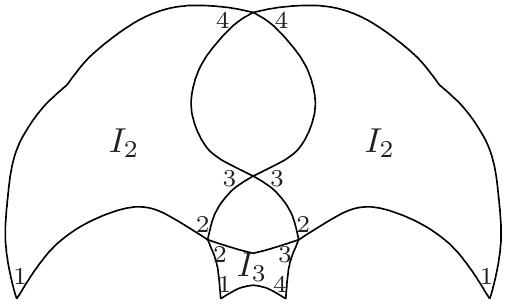}
\end{array}
\end{equation}
where $I_3$ is another space with 4 distinguished points
(Figure~\ref{fig:Julia}(c)).  Finally, $I_3$ is homeomorphic to two copies of
itself glued together:
\begin{equation}
\label{eq:Julia3}
I_3 \iso
\begin{array}{c}
\includegraphics[width=9.2em]{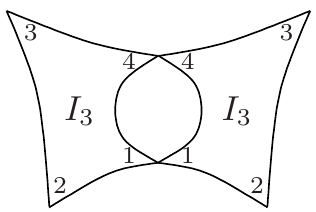}
\end{array}
\end{equation}
No new spaces appear at this stage, so the process ends.  However, the
one-point space has played a role (since we are gluing at single points),
so let us write $I_0$ for the one-point space and record the trivial
isomorphism
\begin{equation}
\label{eq:Julia0}
I_0 \iso I_0.
\end{equation}
Conjecturally, the spaces $I_n$ together with the isomorphisms
(\ref{eq:Julia1})--(\ref{eq:Julia0}) have the following universal property.

Let $\cat{C}$ be the category whose objects are diagrams
\[
\setlength{\unitlength}{1em}
\begin{picture}(7.5,7)(-1.5,-3.3)
\cell{-.5}{0}{r}{X_0}
\cell{4.5}{3}{l}{X_1}
\cell{4.5}{0}{l}{X_2}
\cell{4.5}{-3}{l}{X_3}
\put(0,.8){\vector(1,0){4}}
\put(0,.4){\vector(1,0){4}}
\put(0,0){\vector(1,0){4}}
\put(0,-.4){\vector(1,0){4}}
\put(0,.4){\vector(3,-2){4}}
\put(0,0){\vector(3,-2){4}}
\put(0,-.4){\vector(3,-2){4}}
\put(0,-.8){\vector(3,-2){4}}
\cell{2.5}{.8}{b}{\scriptstyle u_1}
\cell{2.5}{.4}{c}{\scriptstyle u_2}
\cell{2.5}{-.1}{c}{\scriptstyle u_3}
\cell{2.5}{-.5}{t}{\scriptstyle u_4}
\cell{3}{-1.5}{b}{\scriptstyle v_1}
\cell{3}{-2.0}{c}{\scriptstyle v_2}
\cell{3}{-2.5}{c}{\scriptstyle v_3}
\cell{3}{-3.0}{t}{\scriptstyle v_4}
\end{picture}
\]
of topological spaces and continuous closed injections such that $u_1$,
$u_2$, $u_3$ and $u_4$ have disjoint images, and similarly $v_1$, $v_2$, $v_3$
and $v_4$.  Let $G$ be the endofunctor of $\cat{C}$ corresponding to the
right-hand sides of (\ref{eq:Julia1})--(\ref{eq:Julia0}); for instance, 
\[
(G(X))_1	
=
\begin{array}{c}
\includegraphics[width=5.2em]{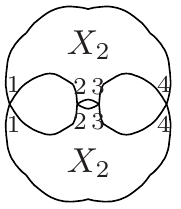}
\end{array}
=
(X_2 + X_2)/\sim
\]
for a certain equivalence relation $\sim$.  (The picture of $(G(X))_1$ is
drawn as if $X_0$ were a single point.)  Then, conjecturally,
(\ref{eq:Julia1})--(\ref{eq:Julia0}) give an isomorphism $\iota: I \goiso
G(I)$ and $(I, \iota)$ is the terminal $G$-coalgebra.
If true, this means that the simple diagrams
(\ref{eq:Julia1})--(\ref{eq:Julia0}) contain as much topological information
as the apparently very complex spaces in Figure~\ref{fig:Julia}: given the
system of equations, we recover these spaces as the universal solution.
(Caveat: we consider only the intrinsic, topological aspects of the spaces,
not how they are embedded into an ambient space or any metric or conformal
structure.)

\minihead{Abstractions}

We now set out the common features of these two examples, eventually arriving
at the general notion of equational system.

For both the real interval and the Julia set, the category $\cat{C}$ is not
$\Set^A$ or $\Top^A$ for any set $A$ (as it was for discrete systems); rather,
it is a full subcategory of $\ftrcat{\scat{A}}{\Set}$ or
$\ftrcat{\scat{A}}{\Top}$ for some small category $\scat{A}$.  In the case of
the interval,
\begin{equation}
\label{eq:interval-category}
\scat{A}
=
\Bigl(
0 \parpair{\sigma}{\tau} 1
\Bigr),
\end{equation}
and in the case of the Julia set,
\begin{equation}
\label{eq:Julia-category}
\scat{A}
=
\left(
\begin{array}{c}
\setlength{\unitlength}{1em}
\begin{picture}(7,7)(-1,-3.3)
\cell{-.5}{0}{r}{0}
\cell{4.5}{3}{l}{1}
\cell{4.5}{0}{l}{2}
\cell{4.5}{-3}{l}{3}
\put(0,.8){\vector(1,0){4}}
\put(0,.4){\vector(1,0){4}}
\put(0,0){\vector(1,0){4}}
\put(0,-.4){\vector(1,0){4}}
\put(0,.4){\vector(3,-2){4}}
\put(0,0){\vector(3,-2){4}}
\put(0,-.4){\vector(3,-2){4}}
\put(0,-.8){\vector(3,-2){4}}
\end{picture}
\end{array}
\right).
\end{equation}
In neither case is $\cat{C}$ the \emph{whole} functor category
$\ftrcat{\scat{A}}{\Set}$ or $\ftrcat{\scat{A}}{\Top}$, because of the
nondegeneracy conditions on the maps $u$ and $v$.  A fruitful generalization
of these conditions is as follows.  (For the tensor product notation, see
Appendix~\ref{app:modules}.)

\begin{defn}
\label{defn:nd-ftr}
Let $\scat{A}$ be a small category.  A functor $X: \scat{A} \go \Set$ is
\demph{nondegenerate} (or \demph{componentwise flat}) if the functor
\[
\dashbk \otimes X: \pshf{\scat{A}} \go \Set
\]
preserves finite connected limits.  The full subcategory of
$\ftrcat{\scat{A}}{\Set}$ formed by the nondegenerate functors is written
$\ndSet{\scat{A}}$.

Write $U: \Top \go \Set$ for the underlying set functor.  A functor $X:
\scat{A} \go \Top$ is \demph{nondegenerate} if $U \of X$ is nondegenerate and
for each map $f$ in $\scat{A}$, the map $Xf$ is closed.  The full subcategory
of $\ftrcat{\scat{A}}{\Top}$ formed by the nondegenerate functors is written
$\ndTop{\scat{A}}$.
\end{defn}

(In fact, it makes no difference to Definition~\ref{defn:nd-ftr} if we change
`finite connected limits' to `pullbacks', by Lemma~2.1 of~\cite{CJ}.  However,
the class of finite connected limits is in some sense better-behaved than the
class of pullbacks: see~\cite{ABLR}.)

It will be shown in~\S\ref{sec:nondegen} that when $\scat{A}$ is the
category~(\ref{eq:interval-category}), $\ndSet{\scat{A}}$ and
$\ndTop{\scat{A}}$ are the categories $\cat{C}$ and $\cat{C}'$ defined in the
real interval example.  Similar statements hold for~(\ref{eq:Julia-category})
and the Julia set example.

A discrete equational system $(A, M)$ consists of a set $A$ and a (suitably
finite) matrix $M$ of natural numbers, that is, a map $M: A \times A \go
\nat$.  The matrix $M$ encodes the right-hand sides of the `equations' that we
seek to solve, and induces an endofunctor $G = M \otimes \dashbk$ of $\Set^A$.
I claim that in our two non-discrete examples, the right-hand sides are
encoded by a module $M: \scat{A} \gomod \scat{A}$ (that is, a functor $M:
\scat{A}^\op \times \scat{A} \go \Set$), and that the induced endofunctor $M
\otimes \dashbk$ of $\ndSet{\scat{A}}$ or $\ndTop{\scat{A}}$ is the
endofunctor $G$ of our examples.

As in the discrete case, the idea is that
\begin{eqnarray*}
M(b, a)		&=	&
\{\textrm{copies of the }b\textrm{th space used in the gluing formula} \\
		&	&
\ \,\textrm{for the }a\textrm{th space} \}
\end{eqnarray*}
($b, a \in \scat{A}$), and elements $m \in M(b, a)$ are called \demph{sectors}
of type $b$ in $a$, written $m: b \gomod a$.  

\begin{example*}{Interval}
We have $\scat{A} = \Bigl( 0 \parpair{\sigma}{\tau} 1 \Bigr)$.  Since, for
instance, the formula~(\ref{eq:Freyd-pushout}) (or~(\ref{eq:Freyd-gluing}))
for $(G(X))_1$ contains $3$ copies of $X_0$, we should have $|M(0, 1)| = 3$.
Naming the elements of the sets $M(b, a)$ suggestively, 
\begin{align*}
M(0, 0)         &= \{\id\},             &
M(0, 1)         &= \{ 0, \dhalf, 1 \},  \\
M(1, 0)         &= \emptyset,           &
M(1, 1)         &= \{ [0, \dhalf], [\dhalf, 1] \}.
\end{align*}
The whole functor $M: \scat{A}^\op \times \scat{A} \go \Set$, including its
action on morphisms, is defined as follows:
\begin{equation}	\label{eq:Freyd-SSS}
\begin{tabular}[c]{c|l}
        &
\begin{diagram}[width=4em,height=3em,tight]
M(\dashbk, 0)							&
\pile{\rTo^{\sigma \cdot \dashbk}\\ 
	\rTo_{\tau \cdot \dashbk}}				&
M(\dashbk, 1)							\\
\end{diagram}
        \\[2ex]
\hline
\raisebox{-9ex}{%
\begin{diagram}[width=3em,height=3em,tight]
M(0, \dashbk)   \\
\uTo<{\dashbk \cdot \sigma} 
\uTo>{\dashbk \cdot \tau}       \\
M(1, \dashbk)   \\
\end{diagram}}
        &
\hspace*{1em}\raisebox{-9ex}{%
\begin{diagram}[width=4em,height=3em,tight]
\{ \id \}							&
\pile{\rTo^0\\ \rTo_1}                                          &
\{ 0, \dhalf, 1 \}                                              \\
\uTo \uTo                                                       &
								&
\uTo<\inf \uTo>\sup                                             \\
\emptyset							&
\pile{\rTo\\ \rTo}                                              &
\{ [0, \dhalf], [\dhalf, 1] \}.					\\
\end{diagram}}
\end{tabular}
\end{equation}
Now $M$ is a module $\scat{A} \gomod \scat{A}$, so induces an endofunctor $M
\otimes \dashbk$ of $\ftrcat{\scat{A}}{\Set}$.  (See
Appendix~\ref{app:modules} for a primer on categorical modules.)  Then, for
instance,
\begin{eqnarray*}
(M \otimes X)_1 &
=       &
(M(0, 1) \times X_0 + M(1, 1) \times X_1)/\sim  \\
        &=      &
(3 \times X_0 + 2 \times X_1)/\sim
\end{eqnarray*}
for some equivalence relation $\sim$.  (Compare~(\ref{eq:Freyd-gluing})
and~(\ref{eq:Freyd-pushout}).)  It follows from later theory that $M \otimes
\dashbk$ restricts to an endofunctor of $\cat{C} = \ndSet{\scat{A}}$, and this
restricted endofunctor is precisely $G$, the endofunctor defined previously.
Analogous statements hold in the topological case.
\end{example*}

\begin{example*}{Julia set}     \label{eg:Julia}
Here $\scat{A}$ is given by~(\ref{eq:Julia-category}).  In the gluing
formula~(\ref{eq:Julia2}) for $I_2$, the one-point space $I_0$ appears 8 times
(Figure~\ref{fig:module}),
\begin{figure}
\centering
\includegraphics[width=14.8em]{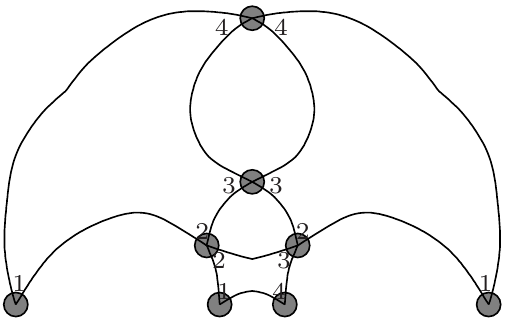}
\caption{The eight elements of $M(0, 2)$} 
\label{fig:module}
\end{figure}
$I_1$ does not appear at all, $I_2$ appears twice, and $I_3$ appears once,
so
\[
|M(0, 2)| = 8,
\diagspace
|M(1, 2)| = 0,
\diagspace
|M(2, 2)| = 2,
\diagspace
|M(3, 2)| = 1.
\]
So, for instance, if $X \in \ftrcat{\scat{A}}{\Top}$ then
\[
(M \otimes X)_2
=
(8 \times X_0 + 2 \times X_2 + X_3)/\sim
\]
where $\sim$ identifies the $8$ copies of $X_0$ with their images in $X_2$ and
$X_3$.  Again it can be shown that $M \otimes \dashbk$ restricts to an
endofunctor of $\cat{C} = \ndTop{\scat{A}}$ and that this is the endofunctor
$G$ described earlier.
\end{example*}

Here is an alternative way of seeing that a system of equations of this type
can be expressed as a module.  The right-hand sides of each of
(\ref{eq:interval-iso-it})--(\ref{eq:interval-iso-pt}) and
(\ref{eq:Julia1})--(\ref{eq:Julia0}) are formal gluings of objects of
$\scat{A}$.  `Gluings' are colimits, so if $\widehat{\scat{A}}$ is the
category obtained by taking $\scat{A}$ and freely adjoining all possible
colimits then the system of equations amounts to a functor from $\scat{A}$ to
$\widehat{\scat{A}}$.  But $\widehat{\scat{A}} = \pshf{\scat{A}}$
(by~\cite[I.5.4]{MM}), so the system is a functor $\scat{A} \go
\pshf{\scat{A}}$, that is, a module $\scat{A} \gomod \scat{A}$.

We confine ourselves to systems of equations in which the right-hand sides are
\emph{finite} gluings.  To formalize this, recall that any presheaf $Y:
\scat{B}^\op \go \Set$
on a small category $\scat{B}$ has a \demph{category of elements}%
\label{p:cat-elts}
$\elt{Y}$, whose objects are pairs $(b, y)$ with $b \in \scat{B}$ and $y
\in Y(b)$; maps $(b, y) \go (b', y')$ are maps $g: b \go b'$ in
$\scat{B}$ such that $y'g = y$.  Similarly, any covariant functor $X:
\scat{A} \go \Set$ has a \demph{category of elements} $\elt{X}$.  In each
case, there is a covariant projection functor from the category of elements
to $\scat{B}$ or $\scat{A}$.

\begin{defn}    \label{defn:finite}
A presheaf $Y: \scat{B}^\op \go \Set$ is \demph{finite} if its category of
elements is finite.  A module $M: \scat{B} \gomod \scat{A}$ is \demph{finite}
if for each $a \in \scat{A}$, the presheaf $M(\dashbk, a)$ is finite.
\end{defn}

Explicitly, $M$ is finite if for each $a \in \scat{A}$ there are only finitely
many diagrams of the form
\[
b' \goby{f} b \gobymod{m} a.
\]
Certainly this holds if, as in the interval example, the category $\scat{A}$
and the sets $M(b, a)$ are finite.

Since our endofunctors $M \otimes \dashbk$ are to act on the subcategory
$\ndSet{\scat{A}}$ of $\ftrcat{\scat{A}}{\Set}$ formed by the nondegenerate
functors, we need $M$ to satisfy a further condition.
Proposition~\ref{propn:preservationofnondegeneracy} shows that the following
condition is sufficient (and, in fact, necessary).
Proposition~\ref{propn:topologicalmoduleaction} shows that also, for such an
$M$, the endofunctor $M \otimes \dashbk$ of $\ftrcat{\scat{A}}{\Top}$
restricts to an endofunctor of $\ndTop{\scat{A}}$.

\begin{defn}    \label{defn:mod-nd}
Let $\scat{A}$ and $\scat{B}$ be small categories.  A module $M: \scat{B}
\gomod \scat{A}$ is \demph{nondegenerate} if $M(b, \dashbk): \scat{A} \go
\Set$ is nondegenerate for each $b \in \scat{B}$. 
\end{defn}

\begin{defn}
An \demph{equational system} is a small category $\scat{A}$ together
with a finite nondegenerate module $M: \scat{A} \gomod \scat{A}$.
\end{defn}

We might more precisely say `finite-colimit equational system'.  The discrete
equational systems are precisely the equational systems $(\scat{A}, M)$ in
which the category $\scat{A}$ is discrete (Example~\ref{eg:discrete-nd}). 

\begin{defn}
Let $(\scat{A}, M)$ be an equational system.  An \demph{$M$-coalgebra in
$\Set$} (respectively, \demph{$\Top$}) is a coalgebra for the endofunctor $M
\otimes \dashbk$ of $\ndSet{\scat{A}}$ (respectively, $\ndTop{\scat{A}}$).  

A \demph{universal solution} of $(\scat{A}, M)$, in $\Set$ or $\Top$, is a
terminal object in the category of $M$-coalgebras.
\end{defn}

Universal solutions are unique (up to isomorphism) when they exist; but just
as not every ordinary system of equations has a solution, not every equational
system has a universal solution.
Theorem~\ref{thm:existenceofuniversalsolution} gives necessary and sufficient
conditions. 

For example, Freyd's theorem~(\ref{thm:Freyd}) characterizes the set $[0, 1]$,
together with its endpoints and the map that multiplies by two, as the
universal solution in $\Set$ of a certain equational system.  The topological
Freyd theorem~(\ref{thm:topologicalFreyd}) characterizes the space $[0, 1]$,
with the same extra structure and the Euclidean topology, as the universal
solution in $\Top$.

Our other example seeks to characterize a certain Julia set as (part of) the
universal solution in $\Top$ of a certain equational system.  Heuristic
arguments and evidence from the theory of laminations~\cite{Thu,Kiwi} suggest
a more general phenomenon.  To discuss it, we need some further definitions.

\begin{defn}    \label{defn:realizable}
A topological space $S$ is \demph{realizable} if there exist an equational
system $(\scat{A}, M)$ with universal solution $(I, \iota)$, and an object $a
\in \scat{A}$, such that $S \iso I(a)$.  It is \demph{discretely realizable}
(respectively, \demph{finitely realizable}) if $\scat{A}$ can be taken to be
discrete (respectively, finite).
\end{defn}
(Instead of `realizable', we might more precisely say `corecursively
realizable by finite colimits'.)

\begin{conj}    \label{conj:julia}
The Julia set $J(f)$ of any complex rational function $f$ is finitely
realizable.  
\end{conj}
This says that in the example, we could have taken any rational function $f$
and seen the same type of behaviour: after a finite number of decompositions,
no more new spaces $I_n$ appear.  Both $J(f)$ and its complement are invariant
under $f$, so $f$ restricts to an endomorphism of $J(f)$, which is, with
finitely many exceptions, a $\deg(f)$-to-one mapping.  This suggests that $f$
itself should provide the recursive structure of $J(f)$, and that if
$(\scat{A}, M)$ is the corresponding equational system then the sizes of
$\scat{A}$ and $M$ should be bounded in terms of $\deg(f)$.

\minihead{Products of equational systems}

We finish with some observations on products that will not be used
until~\S\ref{sec:examples}, and could be omitted on first reading.

Equational systems form a category.  A \demph{map} $(R, \rho): (\scat{A}, M)
\go (\scat{A}', M')$ consists of a functor $R: \scat{A} \go \scat{A}'$
together with a natural transformation
\[
\begin{diagram}
\scat{A}^\op \times \scat{A}    &               &\rTo^{R^\op \times R}  &
                &\scat{A}'^\op \times \scat{A}' \\
                                &\rdTo<M        &\nent\rho              &
\ldTo>{M'}      &                               \\
                                &               &\Set.                  &
                &                               \\
\end{diagram}
\]
This means that $\rho$ assigns to each sector $b \gobymod{m} a$ in $(\scat{A},
M)$ a sector $R(b) \gobymod{\rho(m)} R(a)$ in $(\scat{A}', M')$, in such a way
that the equation $\rho(fmg) = R(f)\rho(m)R(g)$ is satisfied.

\begin{lemma}[Functors on products] 
\label{lemma:functors-on-products}
Let $Z: \scat{B} \go \Set$ and $Z': \scat{B}' \go \Set$ be functors on
categories $\scat{B}, \scat{B}'$, and consider the functor
\[
\begin{array}{cccc}
Z \times Z':    &\scat{B} \times \scat{B}'      &\go            &
\Set            \\
                &(b, b')                        &\goesto        &
Z(b) \times Z'(b').   
\end{array}
\]
Then:
\begin{enumerate}
\item \label{item:prod-elt}
$\elt{Z \times Z'} \iso \elt{Z} \times \elt{Z'}$
\item \label{item:prod-fin}
if $Z$ and $Z'$ are finite then so is $Z \times Z'$
\item \label{item:prod-nd}
if $Z$ and $Z'$ are nondegenerate then so is $Z \times Z'$.
\end{enumerate}
\end{lemma}

\begin{proof}
Part~(\ref{item:prod-elt}) is straightforward, and~(\ref{item:prod-fin})
follows immediately.  Part~(\ref{item:prod-nd}) will follow
from~(\ref{item:prod-elt}) once we have
Theorem~\ref{thm:componentwiseflatness}; it can also be proved by a direct
calculation.  
\done
\end{proof}

Now let $(\scat{A}, M)$ and $(\scat{A}', M')$ be equational systems.  There is
a module
\[
M \times M': \scat{A} \times \scat{A}' \gomod \scat{A} \times \scat{A}'
\]
defined by
\[
(M \times M')((b, b'), (a, a'))
=
M(b, a) \times M'(b', a'),
\]
which by Lemma~\ref{lemma:functors-on-products} is finite and nondegenerate.
So $(\scat{A} \times \scat{A}', M \times M')$ is an equational system, and it
is straightforward to check that it is the product $(\scat{A}, M) \times
(\scat{A}', M')$ in the category of equational systems. 

Later we will use this construction to show that the product of two realizable
spaces is realizable.


\section{Summary of results}
\label{sec:summary}

Now that the language of equational systems has been explained, it is possible
to describe the main results of the rest of this paper.  These results will
give us three fundamental abilities: given an equational system
$(\scat{A}, M)$, we will be able to:
\begin{itemize}
\item determine whether there is a universal solution
\item construct the universal solution whenever it does exist
\item check easily whether a given coalgebra is the universal solution.
\end{itemize}

We begin~(\S\ref{sec:nondegen}) by examining the nondegeneracy condition.  We
give an equivalent formulation of nondegeneracy that is easy to verify in
examples, unlike the original definition
(\ref{defn:nd-ftr},~\ref{defn:mod-nd}).   

There is a well-developed general theory of coalgebras for endofunctors, but
for endofunctors $M \otimes \dashbk$ arising from equational systems, the
theory has a special flavour~(\S\ref{sec:coalgs}).  In a loose way it
resembles homological algebra; we use terms such as \emph{complex},
\emph{double complex} and \emph{resolution}.  We develop this theory and prove
that the endofunctor of $\ftrcat{\scat{A}}{\Set}$ restricts to an endofunctor
of $\ndSet{\scat{A}}$, and similarly for $\Top$, as was assumed in the
introductory sections.

The universal solution of an equational system is quite easily described, in
the case that it exists.  In~\S\ref{sec:construction} we give explicit
sufficient conditions for its existence, and construct it.  In
Appendix~\ref{app:solv} we prove that these conditions are also necessary.
Existence of a universal solution turns out to be unaffected by whether we
work over $\Set$ or $\Top$.

The proof that this really \emph{is} the universal solution is substantial:
\S\ref{sec:set} and \S\ref{sec:Top-proofs} contain the proofs over $\Set$ and
$\Top$, respectively.  The main tools are K\"onig's Lemma~(\ref{lemma:Koenig})
and the homological-like algebra of coalgebras for endofunctors $M \otimes
\dashbk$. 

The third ability is to recognize a universal solution when we see one.  We
prove theorems that allow us to take a coalgebra for some equational system
and decide whether it is the universal solution~(\S\ref{sec:recognition}).
This is much easier than checking directly whether it matches the explicit
construction. 

Using these theorems we can give many examples of equational systems and their
universal solutions~(\S\ref{sec:examples}).  They also let us settle the
question of which topological spaces are realizable, or discretely
realizable---that is, occur as one of the spaces $I(a)$ in the universal
solution of some (discrete) equational system (Appendix~\ref{app:admitting}).

The results of this paper completely answer the most basic questions about
equational systems and their universal solutions.  But an important unanswered
question is this: which topological spaces are \emph{finitely} realizable?
Arguably, the finite equational systems are the most interesting ones, and
come closer to intuitive notions of self-similarity.  But in this paper we do
not attempt a serious development of the more precise theory of finite
equational systems, making only the few remarks at the end
of~\S\ref{sec:nondegen} and the beginning of Appendix~\ref{app:admitting}.


\section{Nondegeneracy}
\label{sec:nondegen}

The main result of this section (Theorem~\ref{thm:componentwiseflatness}) is
that a functor $X: \scat{A} \go \Set$ is nondegenerate if and only if it
satisfies the following explicit conditions:
\begin{description}
\item[\cstyle{ND1}]
given
\[
\begin{diagdiag}
a	&		&	&		&a'	\\
	&\rdTo<f	&	&\ldTo>{f'}	&	\\
	&		&b	&		&	\\
\end{diagdiag}
\]
in $\scat{A}$ and $x \in X(a)$, $x' \in X(a')$ such that $fx = f'x'$, there
exist a commutative square
\[
\begin{diagdiag}
	&		&c	&		&	\\
	&\ldTo<g	&	&\rdTo>{g'}	&	\\
a	&		&	&		&a'	\\
	&\rdTo<f	&	&\ldTo>{f'}	&	\\
	&		&b	&		&	\\
\end{diagdiag}
\]
and $z \in X(c)$ such that $x = gz$ and $x' = g'z$
\item{\cstyle{ND2}} given $a \parpair{f}{f'} b$ in $\scat{A}$ and $x \in
X(a)$ such that $fx = f'x$, there exist a fork 
\begin{equation}	\label{eq:ND2-fork}
c \goby{g} a \parpair{f}{f'} b
\end{equation}
and $z \in X(c)$ such that $x = gz$.  
(A diagram~\bref{eq:ND2-fork} is a \demph{fork} if $fg = f'g$.)
\end{description}

Before developing the theory that leads up to this result, we give some
examples of nondegenerate functors.  They illustrate that nondegeneracy means
`no unforced equalities', in a sense to be explained.  After the main result,
we give explicit conditions for a \emph{module} to be nondegenerate, and we
look more closely at the case where $\scat{A}$ is finite.

\minihead{Examples of nondegenerate functors}

Let us work out what nondegeneracy says for various specific categories
$\scat{A}$, assuming for now that nondegeneracy is equivalent to conditions
\cstyle{ND1} and \cstyle{ND2}.

Note that \cstyle{ND1} holds automatically if either $f$ or $f'$ is an
isomorphism, and that \cstyle{ND2} holds automatically if $f = f'$.  Moreover,
if $f$ is monic then \cstyle{ND1} in the case $f = f'$ just says that $Xf$ is
injective.

\begin{example} \label{eg:nondegen-arrowcat}
Let $\scat{A} = \left( 0 \goby{\sigma} 1 \right)$.  Then $X: \scat{A} \go
\Set$ is nondegenerate if and only if the function $X\sigma: X(0) \go X(1)$ is
injective.

Intuitively, nondegeneracy of a functor $X$ says that no equation between
elements of $X$ holds unless it must.  In this example, nondegeneracy of $X$
says that the equation $\sigma x_0 = \sigma x'_0$ holds only when it must,
that is, only when $x_0 = x'_0$.  
\end{example}

\begin{example} \label{eg:nondegen-Freyd}
Let $\scat{A} = \left( 0 \parpair{\sigma}{\tau} 1 \right)$, so that a
functor $X: \scat{A} \go \Set$ is a pair $\Bigl( X_0
\parpair{X\sigma}{X\tau} X_1 \Bigr)$ of functions.  Then \cstyle{ND1} in
the case $f = f'$ says that $X\sigma$ and $X\tau$ are injective.  The only
other nontrivial case of \cstyle{ND1} is $f = \sigma$, $f' = \tau$, and
since the diagram
\[
\begin{diagdiag}
0	&		&	&		&0	\\
	&\rdTo<\sigma	&	&\ldTo>\tau	&	\\
	&		&1	&		&	\\
\end{diagdiag}
\]
cannot be completed to a commutative square, \cstyle{ND1} says that
$X\sigma$ and $X\tau$ have disjoint images.  The only nontrivial case of
\cstyle{ND2} is $f = \sigma$, $f' = \tau$, and since the diagram $\left( 0
\parpair{\sigma}{\tau} 1 \right)$ cannot be completed to a fork, this says
that $\sigma x_0 \neq \tau x_0$ for all $x_0 \in X_0$, which we
already know.  So a nondegenerate functor on $\scat{A}$ is a parallel pair
of injections with disjoint images, as claimed in~\S\ref{sec:sss}.
\end{example}

\begin{example} \label{eg:nondegen-cofork}
Let $\scat{A}$ be the category generated by objects and
arrows
\[
0 \parpair{\sigma}{\tau} 1 \goby{\rho} 2
\]
subject to $\rho\sigma = \rho\tau$, and consider a
functor $X: \scat{A} \go \Set$.  From \cstyle{ND1} and \cstyle{ND2} it follows
that $X$ is nondegenerate just when:
\begin{itemize}
\item $X\sigma$, $X\tau$, and $X(\rho\sigma)$ are injective
\item $X\sigma$ and $X\tau$ have disjoint images
\item if $\rho x_1 = \rho x'_1$ then $x_1 = x'_1$ or there exists $x_0$
  such that $\{ x_1, x'_1 \} = \{ \sigma x_0, \tau x_0 \}$.
\end{itemize}
The last clause corresponds again to the intuition: the equation $\rho x_1 =
\rho x'_1$ holds only when it must.

An example of a nondegenerate functor on $\scat{A}$ is the diagram
\[
\{ \star \} \parpair{0}{1} [0, 1] \go S^1
\]
exhibiting the circle as an interval with its endpoints identified. 
\end{example}

\begin{example} \label{eg:nondegen-coglobular}
Let $\scat{A}$ be the category generated by objects and arrows
\[
0 \parpair{\sigma_1}{\tau_1} 1 \parpair{\sigma_2}{\tau_2} \ \cdots
\]
subject to $\sigma_{k+1} \sigma_k = \tau_{k+1} \sigma_k$ and $\sigma_{k+1}
\tau_k = \tau_{k+1} \tau_k$ for all $k\geq 1$.  A functor $\scat{A}^\op \go
\Set$ is usually called a globular set or an $\omega$-graph.  It can be shown
that a coglobular set $X: \scat{A} \go \Set$ is nondegenerate precisely when:
\begin{itemize}
\item for all $k \geq 1$, $X\sigma_k$ and $X\tau_k$ are injective
\item for all $k\geq 1$ and $x, x' \in X_k$ satisfying $\sigma_{k+1} x = 
  \tau_{k+1} x'$, we have $x = x' \in \mathrm{image}(X\sigma_k) \cup
  \mathrm{image}(X\tau_k)$ 
\item the images of $X\sigma_1$ and $X\tau_1$ are disjoint.
\end{itemize}
For instance, the underlying coglobular set of any disk in the sense of
Joyal \cite{Joy,SDN} is nondegenerate.
\end{example}

\begin{example*}{Discrete systems} \label{eg:discrete-nd}
It is immediate from \cstyle{ND1} and \cstyle{ND2} that every $\Set$- or
$\Top$-valued functor on a discrete category is nondegenerate.  It follows
that a discrete equational system is the same thing as an equational system
$(\scat{A}, M)$ in which the category $\scat{A}$ is discrete.  The categories
of $M$-coalgebras defined in~\S\ref{sec:dsss} and~\S\ref{sec:sss} then
match up (working over either $\Set$ or $\Top$); hence, so do the notions of
universal solution.
\end{example*}

\minihead{Theory of nondegenerate functors}

The proof of the main theorem on nondegenerate
functors~(\ref{thm:componentwiseflatness}) uses some more sophisticated
category theory than the rest of the paper.  Readers who prefer to take it on
trust can jump straight to the statement of the theorem.
 
None of this theory is new: it goes back to Grothendieck and
Verdier~\cite{GV} and Gabriel and Ulmer~\cite{GU}, and was later developed
by Weberpals~\cite{Web}, Lair~\cite{Lai}, Ageron~\cite{Age}, and Ad\'amek,
Borceux, Lack, and Rosick\'y~\cite{ABLR}.  More general statements of much
of what follows can be found in~\cite{ABLR}.

Let us begin with ordinary flat functors.  A functor $X: \scat{A} \go \Set$
on a small category $\scat{A}$ is \demph{flat} if $\dashbk \otimes X:
\pshf{\scat{A}} \go \Set$ preserves finite limits.  For example, representable
functors are flat: if $X = \scat{A}(a, \dashbk)$ then $\dashbk \otimes X$ is
evaluation at $a$, which preserves all limits.  

\begin{thm}[Flatness]
\label{thm:flatness}
Let $\scat{A}$ be a small category.  The following conditions on a functor
$X: \scat{A} \go \Set$ are equivalent:
\begin{enumerate}
\item \label{item:flat-flat}
$X$ is flat
\item \label{item:flat-fin}
every finite diagram in $\elt{X}$ admits a cone
\item \label{item:flat-fin-bits}
each of the following holds:
\begin{itemize}
\item there exists $a \in \scat{A}$ for which $X(a) \neq \emptyset$
\item given $a, a' \in \scat{A}$, $x \in X(a)$, and $x' \in
  X(a')$, there exist a diagram $a \ogby{g} c \goby{g'} a'$ in $\scat{A}$ and
  $z \in X(c)$ such that $gz = x$ and $g'z = x'$
\item \cstyle{ND2}.
\end{itemize}
\end{enumerate}
\end{thm}
\begin{proof}
See \cite[\S 6.3]{Bor1} or \cite[VII.6]{MM}, for instance.  
\done
\end{proof}

The following lemmas are often used to prove this theorem, and will also be
needed later.

\begin{lemma}[Existence of cones]
\label{lemma:existenceofcones}
Let $\scat{I}$ and $\scat{A}$ be small categories and let $X: \scat{A} \go
\Set$.  Suppose that $\dashbk \otimes X: \pshf{\scat{A}} \go \Set$ preserves
limits of shape $\scat{I}$. 
Then every diagram of shape $\scat{I}$ in $\elt{X}$ admits a cone.
\end{lemma}

\begin{proof}
Let $D: \scat{I} \go \elt{X}$ be a diagram of shape $\scat{I}$, writing $D(i)
= (a_i, x_i)$ for each $i \in \scat{I}$.  Then there is a diagram 
$\scat{I} \go \pshf{\scat{A}}$ given by $i \goesto \scat{A}(\dashbk, a_i)$, so
by hypothesis the canonical map
\[
\left( \Lt{i} \scat{A}(\dashbk, a_i) \right) \otimes X
\go
\Lt{i} (\scat{A}(\dashbk, a_i) \otimes X)
\iso
\Lt{i} X(a_i)
\]
is a bijection, and in particular a surjection.  With the usual explicit
formula for limits in $\Set$, this map is 
\[
(a \goby{p_i} a_i)_{i \in \scat{I}} \otimes x
\goesto
(p_i x)_{i \in \scat{I}}
\]
where $a \in \scat{A}$, $x \in X(a)$ and
\[
(a \goby{p_i} a_i)_{i \in \scat{I}} 
\in
\{ \textrm{cones from } a \textrm{ to } (a_i)_{i \in \scat{I}} \}
=
\Lt{i} \scat{A}(a, a_i).
\]
Since $(x_i)_{i \in \scat{I}} \in \Lt{i} X(a_i)$, there exist $a \in \scat{A}$
and
\[
((p_i)_{i \in \scat{I}}, x) 
\in 
\left( \Lt{i} \scat{A}(a, a_i) \right) \times X(a)
\]
such that $p_i x = x_i$ for all $i$.  Then 
$
\left( 
(a, x) \goby{p_i} (a_i, x_i)
\right)_{i \in \scat{I}}
$
is a cone on $D$.
\done
\end{proof}

Let us say that a category $\cat{C}$ has the \demph{square-completion
property} if there exists a cone on every diagram of shape $\littlepullback$
in $\cat{C}$.

\begin{lemma}[Connectedness by spans]
\label{lemma:connectednessbyspans}
Two objects $c, c'$ of a category with the square-completion property are
in the same connected-component if and only if there exists a span
$
c \og c'' \go c'
$
connecting them.
\done
\end{lemma}

\begin{lemma}[Equality in a tensor product]
\label{lemma:equalityinatensorproduct}
Let $\scat{A}$ be a small category and 
\[
X: \scat{A} \go \Set,
\diagspace
Y: \scat{A}^\op \go \Set.
\]
Suppose that $\elt{X}$ has the square-completion property.  Let 
\[
a, a' \in \scat{A},
\diagspace
(y, x) \in Y(a) \times X(a),
\diagspace
(y', x') \in Y(a') \times X(a').
\]
Then $y \otimes x = y' \otimes x' \in Y \otimes X$ if and only if there exist
a span $ a \ogby{f} b \goby{f'} a' $ and an element $z \in X(b)$ such that $x
= fz$, $x' = f'z$, and $yf = y'f'$.
\end{lemma}
\begin{proof}
See the remarks after the statement of Theorem VII.6.3 in~\cite{MM}.
\done
\end{proof}

We need a fact about connectedness.

\begin{lemma}[Components of a functor]
\label{lemma:componentsofafunctor}
Any functor $X: \scat{A} \go \Set$ on a small category $\scat{A}$ can be
written as a sum $X \iso \sum_{j \in J} X_j$, where $J$ is some set and
$\elt{X_j}$ is connected for each $j\in J$.
\end{lemma}

\begin{proof}
We use the equivalence between $\Set$-valued functors and discrete
opfibrations.  Write $\elt{X}$ as a sum $\sum_{j \in J} \scat{E}_j$ of
connected categories.  For each $j$, the restriction to $\scat{E}_j$ of the
projection $\elt{X} \go \scat{A}$ is still a discrete opfibration, so
corresponds to a functor $X_j: \scat{A} \go \Set$.  Then
\[
\elt{\sum X_j}
\iso 
\sum \elt{X_j}
\iso
\sum \scat{E}_j
\iso
\elt{X}
\]
compatibly with the projections, so $\sum X_j \iso X$.  \done
\end{proof}

Here is the main result.
\begin{thm}[Nondegenerate functors]
\label{thm:componentwiseflatness}
Let $\scat{A}$ be a small category.  The following conditions on a functor
$X: \scat{A} \go \Set$ are equivalent:
\begin{enumerate}
\item \label{item:cwflat-cwflat}
$X$ is nondegenerate
\item \label{item:cwflat-finconn}
every finite connected diagram in $\elt{X}$ admits a cone
\item \label{item:cwflat-finconn-bits}
$X$ satisfies \cstyle{ND1} and~\cstyle{ND2}
\item \label{item:cwflat-sum}
$X$ is a sum of flat functors.
\end{enumerate}
\end{thm}
\paragraph*{Remark} In Lemma~\ref{lemma:componentsofafunctor}, the functors
$X_j$ may be regarded as the connected-components of $X$.  A further
equivalent condition is that every connected-component of $X$ is flat:
hence the name `componentwise flat'.

\begin{proof}

\paragraph*{\bref{item:cwflat-cwflat}$\implies$\bref{item:cwflat-finconn}}
Follows from Lemma~\ref{lemma:existenceofcones}.

\paragraph*{\bref{item:cwflat-finconn}$\implies$\bref{item:cwflat-finconn-bits}}
\cstyle{ND1} says that every diagram of shape $\littlepullback$ in
$\elt{X}$ admits a cone, and similarly~\cstyle{ND2} for $\littleequalizer$.

\paragraph*{\bref{item:cwflat-finconn-bits}$\implies$\bref{item:cwflat-sum}}
Write $X \iso \sum_{j\in J} X_j$ as in
Lemma~\ref{lemma:componentsofafunctor}.  Then in each $\elt{X_j}$, there
exists a cone on every diagram of shape
\[
\littlepullback
\diagspace
\textrm{or}
\diagspace
\littleequalizer
\]
(since $\elt{X} \iso \sum_j \elt{X_j}$), of shape $\emptyset$ (since
$\elt{X_j}$ is connected and therefore nonempty), and of shape
$\littleproduct$ (since $\elt{X_j}$ is connected and has the
square-completion property).  So by
\bref{item:flat-fin-bits}$\implies$\bref{item:flat-flat} of
Theorem~\ref{thm:flatness}, each $X_j$ is flat.

\paragraph*{\bref{item:cwflat-sum}$\implies$\bref{item:cwflat-cwflat}}
Sums commute with connected limits in $\Set$, so any sum of nondegenerate
functors is nondegenerate. 
\done
\end{proof}

\begin{cor}[Componentwise filtered categories]
\label{cor:componentwisefilteredcategories}
The following conditions on a small category $\scat{B}$ are equivalent:
\begin{enumerate}
\item \label{item:cwfilt-comm}
finite connected limits commute with colimits of shape $\scat{B}$ in $\Set$
\item \label{item:cwfilt-finconn}
every finite connected diagram in $\scat{B}$ admits a cocone
\item \label{item:cwfilt-finconn-bits} 
every diagram $b_1 \og b_3 \go b_2$ in $\scat{B}$ can be completed to a
commutative square, and every parallel pair $b_1 \parpair{f}{f'} b_2$
of arrows in $\scat{B}$ can be extended to a cofork.
\end{enumerate}
\end{cor}
\begin{proof}
In Theorem~\ref{thm:componentwiseflatness}, take $\scat{A} = \scat{B}^\op$
and $X$ to be the functor with constant value $1$.  Then $\elt{X} \iso
\scat{B}^\op$ and $\dashbk \otimes X \iso \colt{\scat{B}}$.
The result follows. 
\done
\end{proof}

A small category $\scat{B}$ satisfying the equivalent conditions of
Corollary~\ref{cor:componentwisefilteredcategories} is called
\demph{componentwise filtered}, since a further equivalent condition is that
every connected-component is filtered.  (Grothendieck and Verdier call such
categories `pseudo-filtrantes' \cite{GV}.)  So $X: \scat{A} \go \Set$ is
nondegenerate just when $\elt{X}$ is componentwise \emph{co}filtered.

\minihead{Nondegenerate modules}

We now give a diagrammatic formulation of nondegeneracy of
a \emph{module}.  This will be invaluable later.  By
Theorem~\ref{thm:componentwiseflatness}, a module $M: \scat{A} \gomod
\scat{A}$ is nondegenerate if and only if:
\begin{description}
\item[\cstyle{ND1}]
any commutative square of solid arrows
\[
\begin{diagram}[size=2em]
	&		&b	&		&	\\
	&\ldMod(2,4)<m	&\dModget>{\!\!\!\!\!p}
				&\rdMod(2,4)>{m'}&	\\
	&		&d	&		&	\\
	&\ldGet>g	&	&\rdGet<{g'}	&	\\
a	&		&	&		&a'	\\
	&\rdTo<f	&	&\ldTo>{f'}	&	\\
	&		&c,	&		&	\\
\end{diagram}
\]
can be filled in by dotted arrows to a commutative diagram as shown, and
\item[\cstyle{ND2}] any diagram $b \gobymod{m} a \parpair{f}{f'} c$
with $fm = f'm$ can be extended to a diagram
\[
\begin{diagram}
	&		&d			\\
	&\ruModget<p	&\dGet>e		\\
b	&\rMod_m	&a			\\
	&		&\dTo<f\dTo>{f'}	\\
	&		&c
\end{diagram}
\]
in which the triangle commutes and the right-hand column is a fork.
\end{description}

\minihead{Finite equational systems}

When the category $\scat{A}$ is finite, as is often the case in examples of
equational systems $(\scat{A}, M)$, some more precise results can be proved.
They will not be used in the main development of the theory, but nevertheless
shed light on the concept of nondegeneracy.  I thank Andr\'e Joyal for
bringing them to my attention.

We will use the categorical notion of Cauchy-completeness.  An
\demph{idempotent} in a category $\scat{B}$ is an endomorphism $e: b \go b$ in
$\scat{B}$ such that $e^2 = e$.  It \demph{splits} if there exist maps $a
\oppair{i}{p} b$ such that $pi = 1_a$ and $ip = e$.  A category $\scat{B}$ is
\demph{Cauchy-complete} (or \demph{Karoubi closed}) if every idempotent in
$\scat{B}$ splits.  The importance of this condition is explained
in~\cite{Law} and~\cite{Bor1}.  Every example of a category $\scat{A}$ in this
paper is Cauchy-complete.

\begin{lemma}
\label{lemma:flat-rep-old}
Let $\scat{B}$ be a Cauchy-complete category and $X: \scat{B} \go \Set$ a
finite functor (that is, a functor whose category of elements is finite).
Then $X$ is flat $\iff$ $X$ is representable. 
\end{lemma}

\begin{proof}
See Lemma~5.2 of~\cite{ECC}, for instance.
\done
\end{proof}

\begin{lemma}
\label{lemma:fin-flat-fin}
Let $\scat{B}$ be a finite category and $X: \scat{B} \go \Set$ a flat
functor.  Then $X$ is finite.
\end{lemma}

\begin{proof}
Write $N$ for the number of arrows in $\scat{B}$: then every object in
$\elt{X}$ is the domain of at most $N$ arrows.  Let $S$ be any finite set of
objects of $\elt{X}$.  Since $X$ is flat, $\elt{X}$ is cofiltered, so there is
a cone on $S$ in $\elt{X}$.  Its vertex is the domain of at least $|S|$
arrows, so $|S| \leq N$.  Hence $\elt{X}$ has at most $N$ objects.  Finally,
the hom-sets of $\elt{X}$ are finite, since the same is true in $\scat{B}$; so
$\elt{X}$ is finite.  \done
\end{proof}

\begin{propn}[Flat functors on finite categories]
\label{propn:flat-on-fin}
Let $\scat{B}$ be a finite Cauchy-complete category and $X: \scat{B} \go \Set$
a functor.  Then:
\begin{enumerate}
\item \label{part:ff-flat}
$X$ is flat $\iff$ $X$ is representable
\item \label{part:ff-nd}
$X$ is nondegenerate $\iff$ $X$ is a sum of representables.
\end{enumerate}
\end{propn}

\begin{proof}
For~(\ref{part:ff-flat}), combine Lemmas~\ref{lemma:flat-rep-old}
and~\ref{lemma:fin-flat-fin}.  Part~(\ref{part:ff-nd}) follows, using
Theorem~\ref{thm:componentwiseflatness}.  
\done
\end{proof}

\begin{example} \label{eg:Freyd-fam-rep}
Consider the Freyd equational system $(\scat{A}, M)$~(\S\ref{sec:sss}).  The
category $\scat{A}$ is finite and Cauchy-complete, so
Proposition~\ref{propn:flat-on-fin} applies.  Hence the functors
\[
M(0, \dashbk),\  
M(1, \dashbk),\ 
I:
\scat{A} \go \Set
\]
are all sums of representables.  Indeed,
\begin{align*}
M(0, \dashbk)   &=       
(\{\id\} \parpairu \{0, \dhalf, 1\})    &&\iso    
\scat{A}(0, \dashbk) + \scat{A}(1, \dashbk),    \\
M(1, \dashbk)   &=       
(\emptyset \parpairu \{ [0, \dhalf], [\dhalf, 1] \} )   &&\iso
2 \scat{A}(1, \dashbk), \\
I               &=      
(\{\star\} \parpairu [0, 1] )   &&\iso
\scat{A}(0, \dashbk) + (0, 1) \times \scat{A}(1, \dashbk)
\end{align*}
where $(0, 1)$ is the open real interval.  
\end{example}


\section{Coalgebras}
\label{sec:coalgs}

The general theory of coalgebras for endofunctors has been studied
extensively: see~\cite{Ad}, for instance.  But it turns out that coalgebras
can be understood particularly well when the endofunctor is presented as
\[
M \otimes \dashbk: 
\ftrcat{\scat{A}}{\Set} \go \ftrcat{\scat{A}}{\Set}
\]
for some small category $\scat{A}$ and module $M: \scat{A} \gomod \scat{A}$.
(Every colimit-preserving endofunctor of $\ftrcat{\scat{A}}{\Set}$ has a
unique such presentation.)  We begin this section with some results about
coalgebras for such endofunctors.  These results will be used later, and have
also been used in the pure theory of coalgebras~\cite{KMV}.

We then restrict to the situation where $(\scat{A}, M)$ is an equational
system, and discharge our obligation to prove that $M \otimes \dashbk$ defines
an endofunctor on the categories $\ndSet{\scat{A}}$ and $\ndTop{\scat{A}}$ of
nondegenerate functors.

The key concept throughout is \emph{resolution}.

\minihead{Resolutions}

A coalgebra can be thought of as a kind of iterative system~\cite{Ad}.  To see
this in our context, let $\scat{A}$ be any small category, $M: \scat{A} \gomod
\scat{A}$ any module, and $(X, \xi)$ a coalgebra for the endofunctor $M
\otimes \dashbk$ of $\ftrcat{\scat{A}}{\Set}$.  Let $a_0 \in \scat{A}$ and
$x_0 \in X(a_0)$.  The map
\[
\xi_{a_0}:
X(a_0)
\go
(M \otimes X)(a_0) 
= 
\biggl( \sum_{a_1} M(a_1, a_0) \times X(a_1) \biggr) / \sim
\]
sends $x_0$ to 
\[
\xi_{a_0}(x_0) = (a_1 \gobymod{m_1} a_0) \otimes x_1
\]
for some $a_1 \in \scat{A}$, $m_1 \in M(a_1, a_0)$ and $x_1 \in X(a_1)$.  (To
represent $\xi_{a_0}(x_0)$ as $m_1 \otimes x_1$ requires a choice; there are
in general many such representations.)  Similarly, we may write 
\[
\xi_{a_1}(x_1) = (a_2 \gobymod{m_2} a_1) \otimes x_2.
\]
Continuing in this way, we obtain a diagram
\begin{equation}
\label{eq:complex}
\cdots \ \gobymod{m_{n + 1}} a_n \gobymod{m_n} 
\ \cdots \ \gobymod{m_2} a_1 \gobymod{m_1} a_0
\end{equation}
and a sequence $x_\blb = (x_n)_{n \in \nat}$ with $x_n \in X(a_n)$ and 
\[
\xi_{a_n} (x_n) = m_{n + 1} \otimes x_{n + 1}
\]
for all $n \in \nat$.  The diagram~(\ref{eq:complex}) together with the
sequence $x_\blb$ will be called a \demph{resolution} $(a_\blb, m_\blb,
x_\blb)$ of $x_0$.  I will also call $x_\blb$ a resolution of $x_0$
\demph{along} the diagram~(\ref{eq:complex}).

Clearly every element $x$ of a coalgebra has at least one resolution.  But to
what extent are resolutions unique?  We cannot expect there to be, literally,
a unique resolution of $x$, since at each step there is some choice in how to
represent $\xi_{a_n}(x_n)$.  However, we might hope that the various
resolutions of $x$ are related in some way.  This is indeed the case, as we
shall see, when the functor $X$ is nondegenerate.

We begin by describing how much choice is involved in each individual step.

\begin{lemma}[Equality in $M\otimes X$]
\label{lemma:equalityinMtensorX}
Let $\scat{A}$ be a small category, let $M: \scat{A} \gomod \scat{A}$, and
let $X \in \ndSet{\scat{A}}$.  Take module elements
\[
\begin{diagdiag}
b	&		&	&		&b'	\\
	&\rdMod<m	&	&\ldMod>{m'}	&	\\
	&		&a	&		&	\\
\end{diagdiag}
\]
and $x \in X(b)$, $x' \in X(b')$.  Then $m \otimes x = m' \otimes x' \in (M
\otimes X)(a)$ if and only if there exist a commutative square
\[
\begin{diagdiag}
	&		&c	&		&	\\
	&\ldTo<f	&	&\rdTo>{f'}	&	\\
b	&		&	&		&b'	\\
	&\rdMod<m	&	&\ldMod>{m'}	&	\\
	&		&a	&		&	\\
\end{diagdiag}
\]
and an element $z \in X(c)$ such that $fz = x$ and $f'z = x'$.
\end{lemma}

\begin{proof}
By Theorem~\ref{thm:componentwiseflatness}, $\elt{X}$ has the
square-completion property.  Now apply
Lemma~\ref{lemma:equalityinatensorproduct} with $Y = M(\dashbk, a)$.  
\done
\end{proof}

We will need some terminology.  Let $\scat{A}$ be a small category and $M:
\scat{A} \gomod \scat{A}$ a module.  

A \demph{complex} in $(\scat{A}, M)$ is a diagram~(\ref{eq:complex}),
abbreviated as $(a_\blb, m_\blb)$.  A \demph{map} $(a_\blb, m_\blb)
\go (a'_\blb, m'_\blb)$ of complexes is a sequence $f_\blb = (f_n)_{n \in
\nat}$ of maps in $\scat{A}$ such that the diagram
\[
\begin{diagram}
\cdots  &\      &\rMod^{m_{n+1}}        &a_n    &\rMod^{m_n}    &\      &
\cdots  &\      &\rMod^{m_2}            &a_1    &\rMod^{m_1}    &a_0    \\
        &       &                       &\dTo>{f_n}&            &       &
        &       &                       &\dTo>{f_1}&            &\dTo>{f_0}\\
\cdots  &\      &\rMod_{m'_{n+1}}       &a'_n   &\rMod_{m'_n}   &\      &
\cdots  &\      &\rMod_{m'_2}           &a'_1   &\rMod_{m'_1}   &a'_0   \\
\end{diagram}
\]
commutes.  For each $a \in \scat{A}$ there is a category $\catI(a)$ whose
objects are the complexes $(a_\blb, m_\blb)$ satisfying $a_0 = a$, and whose
maps $f_\blb$ are those satisfying $f_0 = 1_a$.  

Now let $(X, \xi)$ be a coalgebra, $a \in \scat{A}$, and $x \in X(a)$.  There
is a category $\Reso(x)$ whose objects are resolutions $(a_\blb, m_\blb,
x_\blb)$ of $x$, and whose maps 
\[
(a_\blb, m_\blb, x_\blb) 
\go 
(a'_\blb, m'_\blb, x'_\blb)
\]
are the maps $f_\blb: (a_\blb, m_\blb) \go (a'_\blb, m'_\blb)$ in $\catI(a)$
such that $f_n x_n = x'_n$ for all $n\in\nat$.

\begin{propn}[Essential uniqueness of resolutions]
\label{propn:uniqueresos}
Let $\scat{A}$ be a small category, $M: \scat{A} \gomod \scat{A}$ a module,
and $(X, \xi)$ a coalgebra for the endofunctor $M \otimes \dashbk$ of
$\ftrcat{\scat{A}}{\Set}$, with $X$ nondegenerate.  Let $a \in \scat{A}$ and $x
\in X(a)$.  Then the category $\Reso(x)$ is connected.
\end{propn}

\paragraph*{Remark} In fact, $\Reso(x)$ is cofiltered, as can be proved by
an easy extension of the argument below.  We will not need this sharper
result.

\begin{proof}
Certainly $\Reso(x)$ is nonempty.  Now take resolutions $(a_\blb, m_\blb,
x_\blb)$ and $(a'_\blb, m'_\blb, x'_\blb)$ of $x$.  We will construct a span
\[
(a_\blb, m_\blb, x_\blb) 
\og
(b_\blb, p_\blb, y_\blb)
\go
(a'_\blb, m'_\blb, x'_\blb)
\]
in $\Reso(x)$.  Such a span consists of a commutative diagram
\[
\begin{diagram}
\cdots		&\rMod^{m_3}	&a_2		&\rMod^{m_2}	&
a_1		&\rMod^{m_1}	&a_0 = a	\\
		&		&\uTo>{f_2}	&		&
\uTo>{f_1}	&		&\uTo>{f_0 = 1_a}\\
\cdots		&\rMod^{p_3}	&b_2		&\rMod^{p_2}	&
b_1		&\rMod^{p_1}	&b_0 = a	\\
		&		&\dTo>{f'_2}	&		&
\dTo>{f'_1}	&		&\dTo>{f'_0 = 1_a}\\
\cdots		&\rMod_{m'_3}	&a'_2		&\rMod_{m'_2}	&
a'_1		&\rMod_{m'_1}	&a'_0 = a	\\
\end{diagram}
\]
and a sequence $(y_n)_{n \in \nat}$ with $y_n \in X(b_n)$, such that $y_0 =
x$ and 
\[
\xi_{b_n}(y_n) = p_{n + 1} \otimes y_{n + 1},
\quad
f_n y_n = x_n,
\quad
f'_n y_n = x'_n
\]
for each $n \in \nat$.  

Suppose inductively that $n \in \nat$ and $b_r$, $p_r$, $y_r$, $f_r$ and
$f'_r$ have been constructed for all $r \leq n$.  We may write
\[
\xi(y_n) = (c \gobymod{q} b_n) \otimes z
\]
for some $c \in \scat{A}$ and $z \in X(c)$.  Then 
\[
\xi(x_n)
=	
\xi(f_n y_n)
=
f_n \xi(y_n)
=
f_n(q \otimes z)
=
(f_n q) \otimes z,  
\]
but also $\xi(x_n) = m_{n+1} \otimes x_{n+1}$, so by nondegeneracy of $X$ and
Lemma~\ref{lemma:equalityinMtensorX}, there exist a commutative diagram as
labelled~(a) below and an element $w \in X(d)$ such that $gw =
x_{n+1}$ and $hw = z$:
\[
\begin{diagram}[size=1.5em]
	&	&a_{n+1}&	&\rMod^{m_{n+1}}&&a_n	\\
	&	&\uTo<g	&	&	&	&	\\
	&	&d	&	&\textrm{(a)}&	&\uTo>{f_n}\\
	&\ruTo<k&	&\rdTo>h&	&	&	\\
b_{n+1}	&	&\textrm{(b)}&	&c	&\rMod^q&b_n	\\
	&\rdTo<{k'}&	&\ruTo>{h'}&	&	&	\\
	&	&d'	&	&\textrm{(a$'$)}&&\dTo>{f'_n}\\
	&	&\dTo<{g'}&	&	&	&	\\
	&	&a'_{n+1}&	&\rMod_{m'_{n+1}}&&a'_n.\\
\end{diagram}
\]
Similarly, there exist a commutative diagram~(a$'$) and $w' \in X(d')$ such
that $g'w' = x'_{n+1}$ and $h'w' = z$.  So by nondegeneracy of $X$
(condition \cstyle{ND1}), there exist a commutative square~(b) and $y_{n+1}
\in X(b_{n+1})$ such that $k y_{n+1} = w$ and $k' y_{n+1} = w'$.
Put $p_{n + 1} = qhk$, $f_{n+1} = gk$,  and $f'_{n+1} = g'k'$: then
\[
\xi_{b_n}(y_n) 
=
q \otimes z
=
q \otimes hw
=
q \otimes hky_{n + 1}
=
qhk \otimes y_{n + 1}
=
p_{n + 1} \otimes y_n,
\]
and the inductive construction is complete.
\done
\end{proof}

\begin{cor}[Resolving complex]  \label{cor:resolving-complex}
Take $(\scat{A}, M)$, $(X, \xi)$, $a \in \scat{A}$ and $x \in X(a)$ as in
Proposition~\ref{propn:uniqueresos}.  Then any two complexes along which $x$
can be resolved lie in the same connected-component of $\catI(a)$.
\end{cor}

\begin{proof}
The complexes along which $x$ can be resolved are the objects of $\catI(a)$ in
the image of the forgetful functor $\Reso(x) \go \catI(a)$.
\done
\end{proof}

Hence, assuming that the functor $X$ is nondegenerate (and with no assumptions
on $\scat{A}$ and $M$), each element $x \in X(a)$ gives rise canonically to a
connected-component of complexes ending at $a$.

From the perspective of computer science, a complex along which $x$ can be
resolved may be thought of as the observed behaviour of $x$ under iterated
application of $\xi$.  The corollary states that any two observed behaviours
are equivalent.

\minihead{Coalgebras for nondegenerate modules}

We still have to prove that for any equational system $(\scat{A}, M)$, the
endofunctor $M \otimes \dashbk$ of $\ftrcat{\scat{A}}{\Set}$ restricts to an
endofunctor of $\ndSet{\scat{A}}$, and similarly with $\Top$ in place of
$\Set$.  The set-theoretic case is straightforward.  

\begin{propn}[Set-theoretic endofunctor]
\label{propn:preservationofnondegeneracy}
Let $\scat{A}$ be a small category and $M: \scat{A} \gomod \scat{A}$ a
nondegenerate module.  Then the endofunctor $M \otimes \dashbk$ of
$\ftrcat{\scat{A}}{\Set}$ restricts to an endofunctor of $\ndSet{\scat{A}}$.
\end{propn}
Nondegeneracy of $M$ is also a \emph{necessary} condition, since for each $b
\in \scat{A}$ the representable $\scat{A}(b, \dashbk)$ is nondegenerate, and
$M \otimes \scat{A}(b, \dashbk) = M(b, \dashbk)$.

\begin{proof}
Let $X: \scat{A} \go \Set$ be nondegenerate.  Then for any finite connected
limit $\Lt{i} Y_i$ in $\pshf{\scat{A}}$,
\[
\Lt{i} (Y_i \otimes M \otimes X)
\iso
\left(
\Lt{i} (Y_i \otimes M)
\right)
\otimes X
\iso
\left(
\Lt{i} Y_i
\right)
\otimes M \otimes X,
\]
the first isomorphism by nondegeneracy of $X$ and the second by
nondegeneracy of $M$.  So $M \otimes X$ is nondegenerate.
\done
\end{proof}

We now begin the topological case.  

\begin{lemma}[Closed quotient map]
\label{lemma:closedquotientmap}
Let $\scat{A}$ be a small category, $X: \scat{A} \go \Top$ a nondegenerate
functor, and $Y: \scat{A}^\op \go \Set$ a finite functor.  Then the quotient
map 
\[
q: 
\sum_a Y(a) \times X(a)
\go
Y \otimes X
\]
is closed. 
\end{lemma}

\begin{proof}
A subset of $Y \otimes X$ is closed just when its inverse image under $q$ is
closed, so we must show that if $V$ is a closed subset of $\sum Y(a) \times
X(a)$ then its saturation $q^{-1}qV$ is also closed.  Given $a \in \scat{A}$
and $y \in Y(a)$, write $V_{a, y} \sub X(a)$ for the intersection of $V$ with
the $(a, y)$-summand $X(a)$ of
\[
\sum_{(a, y) \in \elt{Y}} X(a) 
\iso
\sum_{a \in \scat{A}} Y(a) \times X(a).
\]
Then $q^{-1}qV = \bigcup_{(a, y) \in \elt{Y}} q^{-1}q V_{a, y}$, so by
finiteness of $Y$ it suffices to show that each set $q^{-1} q V_{a, y}$
is closed.

Fix $(a, y) \in \elt{Y}$.  By definition,
\[
q^{-1}q V_{a, y}
=
\{
(a', y', x') 
\in
\sum_{a' \in \scat{A}}
Y(a') \times X(a')
\such
y' \otimes x' = y \otimes x \textrm{ for some } x \in V_{a, y}
\}.
\]
So by nondegeneracy of $X$ and Lemma~\ref{lemma:equalityinMtensorX}, $(a',
y', x') \in q^{-1}q V_{a, y}$ if and only if:
\begin{condition}
there exist a span
\[
\begin{diagdiag}
	&	&b	&	&	\\
	&\ldTo<f&	&\rdTo>{f'}&	\\
a	&	&	&	&a'	\\
\end{diagdiag}
\]
in $\scat{A}$ and $z \in X(b)$ such that $fz \in V_{a, y}$, $f'z = x'$, and
$yf = y'f'$
\end{condition}
or equivalently:
\begin{condition}
there exist a span
\begin{equation}	\label{eq:spanforclosedness}
\begin{diagdiag}
	&	&(b, w)	&	&	\\
	&\ldTo<f&	&\rdTo>{f'}&	\\
(a, y)	&	&	&	&(a', y')\\
\end{diagdiag}
\end{equation}
in $\elt{Y}$ and $z \in X(b)$ such that $fz \in V_{a, y}$ and $f'z = x'$.
\end{condition}
So
\[
q^{-1}q V_{a, y}
=
\bigcup_{\mathrm{spans\ }\bref{eq:spanforclosedness}}
\{
(a', y', x') 
\such
x' \in (Xf')(Xf)^{-1} V_{a, y}
\}.
\]
But $V_{a, y}$ is closed in $X(a)$, each $Xf$ is continuous, and each $Xf'$ is
closed, so each of the sets $\{ \ldots \}$ in this union is a closed subset of
the $(a', y')$-summand $X(a')$.  Moreover, finiteness of $Y$ guarantees
that the union is finite.  Hence $q^{-1}q V_{a, y}$ is closed, as required. 
\done
\end{proof}

Part of the definition of nondegeneracy of a functor $X: \scat{A} \go \Top$ is
that for each map $f$ in $\scat{A}$, the map $Xf$ is closed.  In later
theory we will never use this condition directly; we will only use the
property described in the lemma.

\begin{cor}[Coprojections closed]
\label{cor:coprojections-closed}
Let $\scat{A}$ be a small category, $M: \scat{A} \gomod \scat{A}$ a finite
module, and $X: \scat{A} \go \Top$ a nondegenerate functor.  Then for each $m:
b \gomod a$ in $M$, the coprojection
\[
m \otimes \dashbk:
X(b) \go (M \otimes X)(a)
\]
is closed.
\end{cor}

\begin{proof}
The map $m \otimes \dashbk$ is the composite
\[
X(b)
\goby{(m, \dashbk)}
\sum_{b'} M(b', a) \times X(b')
\rQt^{\mathrm{quotient\ map}}
(M \otimes X)(a).
\]
The first map is closed since it is a coproduct-coprojection, and the
second is closed by Lemma~\ref{lemma:closedquotientmap}.  
\done
\end{proof}

\begin{lemma}[Change of category]
\label{lemma:changeofcategory}
Let $\scat{A}$ be a small category and $M: \scat{A} \gomod \scat{A}$ a finite
module.  Let $\cat{E}$ and $\cat{E}'$ be categories with finite colimits, and
$F: \cat{E} \go \cat{E}'$ a functor preserving finite colimits.  Then the
square
\[
\begin{diagram}
\ftrcat{\scat{A}}{\cat{E}}	&\rTo^{M \otimes \dashbk}	&
\ftrcat{\scat{A}}{\cat{E}}	\\
\dTo<{F \of \dashbk}		&				&
\dTo>{F \of \dashbk}		\\
\ftrcat{\scat{A}}{\cat{E}'}	&\rTo_{M \otimes \dashbk}	&
\ftrcat{\scat{A}}{\cat{E}'},	\\
\end{diagram}
\]
commutes up to canonical isomorphism. 
\end{lemma}

\begin{proof}
Straightforward.
\done
\end{proof}

\begin{propn}[Topological endofunctor]
\label{propn:topologicalmoduleaction}
Let $(\scat{A}, M)$ be an equational system.  Then the endofunctor $M \otimes
\dashbk$ of $\ftrcat{\scat{A}}{\Top}$ restricts to an endofunctor of
$\ndTop{\scat{A}}$.  
\end{propn}
\begin{proof}
Let $X \in \ndTop{\scat{A}}$.  The functor $U: \Top \go \Set$ preserves
colimits (being left adjoint to the indiscrete space functor), so $M
\otimes (U \of X) \iso U \of (M \otimes X)$ by
Lemma~\ref{lemma:changeofcategory}.  But $M \otimes (U \of X)$ is
nondegenerate by Proposition~\ref{propn:preservationofnondegeneracy}, so $U
\of (M \otimes X)$ is nondegenerate.

Now let $a \goby{f} a'$ be a map in $\scat{A}$, and consider the
commutative square
\[
\begin{diagram}[height=6ex]
\sum_b M(b, a) \times X(b)		&
\rTo^{\sum f_* \times 1}		&
\sum_b M(b, a') \times X(b)		\\
\dQt<{q_a}				&
					&
\dQt>{q_{a'}}				\\
(M \otimes X)(a)			&
\rTo_{(M \otimes X) f}			&
(M \otimes X)(a').			\\
\end{diagram}
\]
The map $\sum f_* \times 1$ is closed because each set $M(b, a)$ is finite.
The map $q_{a'}$ is closed by Lemma~\ref{lemma:closedquotientmap} and
finiteness of $M$.  So $((M \otimes X) f) \of q_a$ is closed; but $q_a$ is
a continuous surjection, so $(M \otimes X) f$ is closed.  \done
\end{proof}

We have now shown that for an equational system $(\scat{A}, M)$, there are
induced endofunctors $M \otimes \dashbk$ of both $\ndSet{\scat{A}}$ and
$\ndTop{\scat{A}}$.  We will study the categories of coalgebras of these
endofunctors, denoted $\Coalg{M}{\Set}$ and $\Coalg{M}{\Top}$.

The forgetful functor $U: \Top \go \Set$ induces a functor
\[
U_*: 
\Coalg{M}{\Top}
\go
\Coalg{M}{\Set}.
\]
Indeed, if $(X, \xi)$ is an $M$-coalgebra in $\Top$ then $U \of X: \scat{A}
\go \Set$ is nondegenerate, and by Lemma~\ref{lemma:changeofcategory} there is
a natural transformation
\[
U \xi: 
U \of X 
\go
U \of (M \otimes X)
\iso
M \otimes (U \of X).
\]

\begin{propn}[$\Top$ vs $\Set$]
\label{propn:coalgebrasinTopandSet}
Let $(\scat{A}, M)$ be an equational system.  The forgetful functor
\[
U_*: 
\Coalg{M}{\Top}
\go
\Coalg{M}{\Set}
\]
has a left adjoint, and if $(I, \iota)$ is a universal solution in $\Top$
then $U_*(I, \iota)$ is a universal solution in $\Set$. 
\end{propn}
Conversely, we will see later that any universal solution in $\Set$
carries a natural topology, and is then the universal solution in $\Top$.

\begin{proof}
Let $D$ be the left adjoint to $U: \Top \go \Set$, assigning to each set
the corresponding discrete space.  Then composition with $D$ induces a functor
$\ndSet{\scat{A}} \go \ndTop{\scat{A}}$.  Moreover, $D$ preserves colimits, so
commutes with $M \otimes \dashbk$ (Lemma~\ref{lemma:changeofcategory}); hence
$D$ also induces a functor 
\[
D_*: 
\Coalg{M}{\Set} 
\go 
\Coalg{M}{\Top}.
\]
For purely formal reasons, the adjunction $D \ladj U$ induces an adjunction
$D_* \ladj U_*$.  The statement on universal solutions follows from the
fact that right adjoints preserve terminal objects.
\done 
\end{proof}

\begin{example*}{Discrete systems}
When $\scat{A}$ is discrete, most of the results of this section become
trivial.  Every $\Set$- or $\Top$-valued functor on a discrete category is
nondegenerate, so $\ndSet{\scat{A}} = \ftrcat{\scat{A}}{\Set}$ and
$\ndTop{\scat{A}} = \ftrcat{\scat{A}}{\Top}$.

Let $M: \scat{A} \gomod \scat{A}$ be a module and $(X, \xi)$ an $M$-coalgebra
in $\Set$.  Then every element $x \in X(a)$ ($a \in \scat{A}$) has a unique
resolution, and $\Reso(x)$ is the terminal category $\One$.  As we saw in
\S\ref{sec:dsss}, every discrete equational system $(\scat{A}, M)$ has a
universal solution in both $\Top$ and $\Set$; and in accordance with
Proposition~\ref{propn:coalgebrasinTopandSet}, the universal solution in
$\Top$ is the universal solution in $\Set$, suitably topologized.
\end{example*}


\section{Construction of the universal solution}
\label{sec:construction}

In this section we construct the universal solutions in $\Set$ and in $\Top$
of any given equational system, assuming that the system satisfies a certain
solvability condition \So.  In Appendix~\ref{app:solv}, this sufficient
condition is shown to be necessary: \So\ holds if and only if there is a
universal solution in $\Set$, if and only if there is a universal solution in
$\Top$.  The construction therefore gives the universal solution whenever one
exists.  This is very unusual in the theory of coalgebras: in many contexts,
sufficient conditions are known for the existence of a terminal coalgebra, but
few are known to be necessary.  Compare also~\cite{KMV}.

Condition \So\ on an equational system $(\scat{A}, M)$ is:
\begin{description}
\item[\cstyle{S1}]
for every commutative diagram
\[
\begin{diagram}
\cdots		&\rMod^{m_3}	&a_2		&\rMod^{m_2}	&
a_1		&\rMod^{m_1}	&a_0		\\
		&		&\dTo>{f_2}	&		&
\dTo>{f_1}	&		&\dTo>{f_0}	\\
\cdots		&\rMod^{p_3}	&b_2		&\rMod^{p_2}	&
b_1		&\rMod^{p_1}	&b_0		\\
		&		&\uTo>{f'_2}	&		&
\uTo>{f'_1}	&		&\uTo>{f'_0}	\\
\cdots		&\rMod_{m'_3}	&a'_2		&\rMod_{m'_2}	&
a'_1		&\rMod_{m'_1}	&a'_0,		\\
\end{diagram}
\]
there exists a commutative square
\[
\begin{diagdiag}
	&	&a_0	&		&	\\
	&\ruTo	&	&\rdTo>{f_0}	&	\\
\cdot	&	&	&		&b_0	\\
	&\rdTo	&	&\ruTo>{f'_0}	&	\\
	&	&a'_0	&		&	\\
\end{diagdiag}
\]
in $\scat{A}$, and
\item[\cstyle{S2}] for every serially commutative diagram
\[
\begin{diagram}
\cdots		&\rMod^{m_3}	&a_2			&\rMod^{m_2}	&
a_1			&\rMod^{m_1}	&a_0			\\
		&		&\dTo<{f_2} \dTo>{f'_2}	&		&
\dTo<{f_1} \dTo>{f'_1} 	&		&\dTo<{f_0} \dTo>{f'_0}	\\
\cdots		&\rMod_{p_3}	&b_2			&\rMod_{p_2}	&
b_1			&\rMod_{p_1}	&b_0,		\\
\end{diagram}
\]
there exists a fork $\cdot \go a_0 \parpair{f_0}{f'_0} b_0$ in
$\scat{A}$. 
\end{description}

In \cstyle{S2}, `serially commutative' means that $f_{n - 1} m_n = p_n f_n$
and $f'_{n - 1} m_n = p_n f'_n$ for all $n \geq 1$.

\begin{example}
For any small category $\scat{A}$ there is a module $M: \scat{A} \gomod
\scat{A}$ defined by $M(b, a) = \scat{A}(b, a)$, and $(\scat{A}, M)$ is an
equational system as long as $\sum_b \scat{A}(b, a)$ is finite for
each $a \in \scat{A}$.  Condition \So\ says that $\scat{A}$ is
componentwise cofiltered; so, for instance, the equational system
obtained by taking $\scat{A} = ( 0 \parpairu 1 )$ has no universal
solution.  If $\scat{A}$ \emph{is} componentwise cofiltered then the
universal solution is the functor $\scat{A} \go \Top$ constant at the
one-point space, with its unique coalgebra structure.
\end{example}

We now construct the universal solutions in $\Set$ and in $\Top$ of any
equational system satisfying \So.  The proofs that they are indeed universal
solutions are in~\S\ref{sec:set} and~\S\ref{sec:Top-proofs}, respectively.

\minihead{The universal solution in $\Set$}

Let $(\scat{A}, M)$ be an equational system.  For each $a \in \scat{A}$, we
have the category $\catI(a)$ of complexes ending at $a$ (\S\ref{sec:coalgs}).
Each map $f: a \go a'$ in $\scat{A}$ induces a functor $\catI f: \catI(a) \go
\catI(a')$, sending a complex 
\[
(a_\blb, m_\blb)
=
\left(
\cdots \gobymod{m_3} a_2 \gobymod{m_2} a_1 \gobymod{m_1} a_0 = a
\right)
\]
to the complex
\[
\cdots \gobymod{m_3} a_2 \gobymod{m_2} a_1 \gobymod{f m_1} a'.
\]
This defines a functor $\catI : \scat{A} \go \Cat$.  

Write $\Pi_0: \Cat \go \Set$ for the functor sending a small category to its
set of connected-components, and put $I = \Pi_0 \of \catI : \scat{A} \go
\Set$.  We write $[a_\blb, m_\blb] \in I(a)$ for the connected-component of a
complex $(a_\blb, m_\blb) \in \catI(a)$.  In~\S\ref{sec:set} we will show that
if $(\scat{A}, M)$ satisfies~\So\ then $I$ is nondegenerate.

Later we will analyze in detail the relation of connectedness in $\catI(a)$,
that is, equality in $I(a)$.  For now, let us just note the following: for any
diagram
\[
\cdots 
\gobymod{m'_{n+2}} a'_{n+1} \gobymod{m'_{n+1}} a'_n 
\goby{f}
a_n \gobymod{m_n} \cdots \gobymod{m_1} a_0 = a,
\]
there is a map 
\begin{equation}
\label{eq:finitary}
\begin{diagram}
\cdots                  &\              &
\rMod^{m'_{n+2}}        &a'_{n+1}       &
\rMod^{m'_{n+1}}        &a'_n           &
\rMod^{m_n f}           &a_{n-1}        &
\rMod^{m_{n-1}}         &\              &
\cdots                  &\              &
\rMod^{m_1}             &a_0 = a        \\
                        &               &
                        &\dTo>1         &
                        &\dTo>f         &
                        &\dTo>1         &
                        &               &
                        &               &
                        &\dTo>1         \\
\cdots                  &\              &
\rMod_{m'_{n+2}}        &a'_{n+1}       &
\rMod_{f m'_{n+1}}      &a_n            &
\rMod_{m_n}             &a_{n-1}        &
\rMod_{m_{n-1}}         &\              &
\cdots                  &\              &
\rMod_{m_1}             &a_0 = a        \\
\end{diagram}
\end{equation}
in $\catI(a)$, so the complexes in the top and bottom rows represent the same
element of $I(a)$.

\begin{warning} \label{warning:fin}
The set $I(a)$ is \emph{not} in general the limit of finite approximations.
That is, let $\catI_n(a)$ be the category whose objects are diagrams of the
form
\begin{equation} \label{eq:finitesequence}
a_n \gobymod{m_n} \cdots \gobymod{m_1} a_0 = a
\end{equation}
and whose arrows are commutative diagrams; $\catI_n(a)$ is finite, since $M$
is.  Let $I_n(a) = \Pi_0 \catI_n(a)$.  Then $\catI(a)$ is the limit of the
categories $\catI_n(a)$, but $I(a)$ is typically not the limit of the sets
$I_n(a)$.  (An exception is when $\scat{A}$ is discrete.)  More precisely, the
canonical map $I(a) \go \Lt{n} I_n(a)$ need not be injective, since there may
be two complexes in different connected-components of $\catI(a)$ whose images
in each $\catI_n(a)$ are, nevertheless, always in the same component.  An
example is given at the end of~\ref{eg:construction-Freyd}, which also shows
that the sequential limit of connected categories need not be connected.

The point can be clarified using the notion of `distance' in a category.  For
objects $A$ and $A'$ of a category $\cat{C}$, the \demph{distance}
$d_{\cat{C}}(A, A')$ is the smallest number $n \in \nat$ for which there
exists a diagram
\[
\begin{diagram}[size=1.5em]
        &       &B_1    &       &       &       &B_2    &
        &       &       &       &       &B_n    &       &       \\
        &\ldTo  &       &\rdTo  &       &\ldTo  &       &
\rdTo   &       &       &       &\ldTo  &       &\rdTo  &       \\
A = A_0 &       &       &       &A_1    &       &       &
        &\      &\cdots &\      &       &       &       &A_n = A'\\
\end{diagram}
\]
in $\cat{C}$, or $\infty$ if no such diagram exists.  Thus, $A$ and $A'$ are
in the same connected-component if and only if $d_{\cat{C}}(A, A') < \infty$.
Any functor $F: \cat{C} \go \cat{D}$ induces a distance-decreasing map:
$d_{\cat{D}}(F(A), F(A')) \leq d_{\cat{C}}(A, A')$. 

Now take $a \in \scat{A}$ and two complexes $\alpha, \alpha' \in
\catI(a)$.  Writing $\pr_n: \catI(a) \go \catI_n(a)$ for projection, we have
\[
d_{\catI_n(a)}(\pr_n(\alpha), \pr_n(\alpha'))
\leq
d_{\catI(a)} (\alpha, \alpha')
\]
for all $n$.  So if $\alpha$ and $\alpha'$ are in the same connected-component
of $\catI(a)$ then not only are the distances $d_{\catI_n(a)}(\pr_n(\alpha),
\pr_n(\alpha'))$ finite individually, but also there is an overall bound:
\begin{equation}
\label{eq:distancebound}
\sup_{n \geq 1} 
\Bigl( d_{\catI_n(a)} \bigl( \pr_n(\alpha), \pr_n(\alpha') \bigr) \Bigr)
< 
\infty.
\end{equation}
Hence, the condition that $\pr_n(\alpha)$ and $\pr_n(\alpha')$ always
represent the same element of $I_n(a)$ is not enough to guarantee that
$\alpha$ and $\alpha'$ represent the same element of $I(a)$:
(\ref{eq:distancebound}) must also hold.  (In fact,~(\ref{eq:distancebound})
is also a \emph{sufficient} condition: Proposition~\ref{propn:equalityinIa}).  
\end{warning}

We need to define a coalgebra structure on $I$, that is, a natural
transformation $I \go M \otimes I$.  In order to do so, we first define one on
$\oba\catI$, the composite of $\catI: \scat{A} \go \Cat$ with the objects
functor $\ob: \Cat \go \Set$.  The functor $\oba\catI$ is nondegenerate
(whether or not \So\ holds), since
\begin{equation}        \label{eq:obcatI-nondegen}
\oba\catI  
\iso
\sum_b 
\oba\catI(b) \times M(b, \dashbk)
\end{equation}
and the class of nondegenerate functors is closed under sums
(Theorem~\ref{thm:componentwiseflatness}).  The coalgebra structure $\iota:
\oba\catI \go M \otimes \oba\catI$ is defined by
\[
\iota_a
( \cdots \gobymod{m_3} a_2 \gobymod{m_2} a_1 \gobymod{m_1} a )
\ 
=
\ 
(a_1 \gobymod{m_1} a)
\otimes
(\cdots \gobymod{m_3} a_2 \gobymod{m_2} a_1)
\]
($a \in \scat{A}$), or equivalently by taking $\iota_a$ to be the composite  
\begin{equation}	\label{eq:iota-as-comp}
\oba\catI(a)
\goiso
\sum_b M(b, a) \times \oba\catI(b)
\goby{\mathrm{quotient\ map}}
(M \otimes \oba\catI)(a).
\end{equation}
We also have a quotient map $\pi: \oba\catI \go I$, mapping a complex
$(a_\blb, m_\blb) \in \oba\catI(a)$ to its connected-component $[a_\blb,
m_\blb] \in I(a)$.  It is easy to show that the coalgebra structure on
$\oba\catI$ induces a coalgebra structure on $I$, unique such that $\pi$ is a
map of coalgebras.  We call this coalgebra structure $\iota$, too; it is
characterized by
\[
\iota_a
\Bigl(
\bigcompt{
\cdots \gobymod{m_3} a_2 \gobymod{m_2} a_1 \gobymod{m_1} a
}
\Bigr)
=
(
a_1 \gobymod{m_1} a
)
\otimes
\bigcompt{
\cdots \gobymod{m_3} a_2 \gobymod{m_2} a_1
}.
\]

\minihead{The universal solution in $\Top$}

Next we equip $I$ with a topology.  For each $a \in \scat{A}$, $n \in \nat$,
and truncated complex~\bref{eq:finitesequence}, there is a subset $V_{m_1,
\ldots, m_n}$ of $I(a)$ consisting of all those $t \in I(a)$ such that
\[
t = 
\bigcompt{ 
\cdots 
\gobymod{m_{n+2}} a_{n+1} \gobymod{m_{n+1}} a_n \gobymod{m_n}
\cdots
\gobymod{m_1} a_0 = a
}
\]
for some $m_{n+1}, a_{n+1}, m_{n+2}, \ldots$.  Equivalently, each sector $m: b
\gomod a$ induces a function $\phi_m: I(b) \go I(a)$, with 
\[
\phi_m 
\Bigl(
\compt{ \cdots \gobymod{p_2} b_1 \gobymod{p_1} b }
\Bigr)
=
\compt{ \cdots \gobymod{p_2} b_1 \gobymod{p_1} b \gobymod{m} a },
\]
and then
\[
V_{m_1, \ldots, m_n}
=
\phi_{m_1} \phi_{m_2} \cdots \phi_{m_n}(I(a_n)).
\]
Generate a topology on $I(a)$ by taking each such subset to be closed. 

In order for $(I, \iota)$ to be a coalgebra in $\Top$, the maps $\iota_a$ must
be continuous (for every $a \in \scat{A}$) and the maps $If$ must be
continuous and closed (for every map $f$ in $\scat{A}$).  We will prove
in~\S\ref{sec:Top-proofs} that these statements are true if \So\ holds.  In
fact, it will follow from Lemma~\ref{lemma:fixedpointcomponents} that we have
just given the sets $I(a)$ the coarsest possible topology for which $(I,
\iota)$ is a coalgebra in $\Top$.  

\begin{example*}{Discrete systems}      \label{eg:construction-discrete}
Let $(\scat{A}, M)$ be a discrete equational system.  Condition \So\ holds
trivially.  For each $a \in \scat{A}$, the category $\catI(a)$ is discrete, so
$I(a)$ is simply $\oba\catI(a)$, the set of complexes ending at $a$.  The
topology on $I(a)$ is generated by declaring that for each
diagram~(\ref{eq:finitesequence}), the set of all complexes ending
in~(\ref{eq:finitesequence}) is closed in $I(a)$.  This is the profinite
topology on $I(a)$ defined at the end of~\S\ref{sec:dsss}.
\end{example*}

\begin{example*}{Interval}      \label{eg:construction-Freyd}
We run through the constructions of this section in the case of the Freyd
equational system $(\scat{A}, M)$~(\S\ref{sec:sss}).

Condition~\So\ is easily verified.  Theorem~\ref{thm:Freyd} states---although
we have yet to prove it---that the universal solution is the coalgebra $(I,
\iota)$ defined in~\S\ref{sec:sss}; in particular, $I(1) = [0, 1]$.  So an
element of $[0, 1]$ should be an equivalence class of complexes
\[
\cdots \gobymod{m_3} a_2 \gobymod{m_2} a_1 \gobymod{m_1} 1.
\]
If each $a_n$ is $1$ then each $m_n$ is either $[0, \half]$ or $[\half, 1]$
and the complex is essentially a binary expansion; for instance, the diagram
\[
\cdots \gobymod{[0, \half]} 1 \gobymod{[\half, 1]} 1 
\gobymod{[0, \half]} 1 \gobymod{[\half, 1]} 1 
\gobymod{[0, \half]} 1 \gobymod{[\half, 1]} 1
\]
corresponds to $0.101010\ldots$, representing $2/3 \in [0, 1]$.
Otherwise, the diagram is of the form
\[
\cdots 
\gobymod{\id} 0 \gobymod{\id} 0
\gobymod{m_{n+1}} 1 \gobymod{m_n}
\cdots
\gobymod{m_1} 1
\]
where $m_1, \ldots, m_n \in \{ [0, \half], [\half, 1] \}$ and $m_{n+1} \in
\{ 0, \half, 1 \}$.  Take, for instance,
\[
\cdots 
\gobymod{\id} 0 \gobymod{\id} 0
\gobymod{\half} 1 \gobymod{[\half, 1]} 1
\gobymod{[\half, 1]} 1 \gobymod{[0, \half]} 1.
\]
To see which element $t$ of $[0, 1]$ this represents, we can reason as
follows.  The $[0, \half]$ says that $t \in [0, 1/2]$.  The right-hand
instance of $[\half, 1]$ says that $t$ is in the upper half of $[0, 1/2]$,
that is, in $[1/4, 1/2]$.  The left-hand instance of $[\half, 1]$ says that
$t$ is in the upper half of $[1/4, 1/2]$, that is, in $[3/8, 1/2]$.  The
$\half$ says that $t$ is the midpoint of $[3/8, 1/2]$; that is, $t = 7/16$.

An element of $[0, 1]$ has at most two binary expansions, but may have
infinitely many representations in $\catI(1)$.  For instance, the
representations of $1/2$ are
\begin{eqnarray}
\cdots \gobymod{\id} 0 \gobymod{\id} 0 \gobymod{\half} 1,
\label{eq:half-1}	\\
\cdots \gobymod{[0, \half]} 1 \gobymod{[0, \half]} 1 
\gobymod{[\half, 1]} 1, 
\label{eq:half-2}	\\
\cdots \gobymod{[\half, 1]} 1 \gobymod{[\half, 1]} 1 
\gobymod{[0, \half]} 1,
\label{eq:half-3}  
\end{eqnarray}
and for any $n\in\nat$,
\begin{eqnarray}
\cdots \gobymod{\id} 0 \gobymod{1} 1
\gobymod{[\half, 1]} \cdots \gobymod{[\half, 1]} 1
\gobymod{[0, \half]} 1,
\label{eq:half-4}	\\
\cdots \gobymod{\id} 0 \gobymod{0} 1
\gobymod{[0, \half]} \cdots \gobymod{[0, \half]} 1
\gobymod{[\half, 1]} 1\ 
\label{eq:half-5}	
\end{eqnarray}
with $n$ copies of $[\half, 1]$ and $[0, \half]$ respectively.
(Example~\ref{eg:can-interval} shows that, in a sense, (\ref{eq:half-1}) is
the canonical representation.)

The construction says that two objects of $\catI(1)$ represent the same element
of $[0, 1]$ if and only if they are in the same connected-component.  So, for
instance, each of \bref{eq:half-1}--\bref{eq:half-5} are in the same
component.  The connected diagram
\[
\begin{diagram}
\cdots	&\rMod^{[\half, 1]}	&1		&
\rMod^{[\half, 1]}	&1		&
\rMod^{[0, \half]}	&1	\\
	&			&\uTo>\tau	&
			&\uTo>\tau	&
			&\uTo>1	\\
\cdots	&\rMod^{\id}		&0		&
\rMod^{\id}		&0		&
\rMod^{\half}		&1	\\
	&			&\dTo>\sigma	&
			&\dTo>\sigma	&
			&\dTo>1	\\
\cdots	&\rMod_{[0, \half]}	&1		&
\rMod_{[0, \half]}	&1		&
\rMod_{[\half, 1]}	&1	\\
\end{diagram}
\]
shows that \bref{eq:half-1}--\bref{eq:half-3} are; the others are left
to the reader.  

The constructed topology on $[0, 1]$ is generated by taking as closed
all subsets of the form $[k/2^n, l/2^n]$ where $k, l, n \in \nat$ and $l
\in \{k, k + 1\}$.  This is exactly the Euclidean topology.

Finally, this example shows that $I(a)$ need not be the limit of the finite
approximations $I_n(a)$ (Warning~\ref{warning:fin}).  It is not hard to show
that for each $n$, the category $\catI_n(1)$ is connected, and so each
$I_n(1)$ is a one-element set.  But $I(1) = [0, 1]$, so $I(1) \not\iso \Lt{n}
I_n(1)$.

\end{example*}


\section{Set-theoretic proofs}
\label{sec:set}

The main result of this section is that, for an equational system satisfying
the solvability condition~\So, the construction above really does give the
universal solution in $\Set$.

We do not even know yet that this construction gives a coalgebra.  Given an
equational system $(\scat{A}, M)$, we do have a functor $I: \scat{A} \go \Set$
and a natural transformation $\iota: I \go M \otimes I$, but in order for $(I,
\iota)$ to be called an $M$-coalgebra, it must, by definition, be
nondegenerate.  A large part of this section is devoted to proving that.  (The
proof requires condition~\So.)  It then follows quite quickly that $(I,
\iota)$ is the universal solution.

An element of one of the sets $I(a)$ is an equivalence class of complexes.  We
finish the section by showing that under very mild conditions on $\scat{A}$,
each such element has a canonical complex representing it.

\minihead{Connectedness in $\catI(a)$}

The functor $I: \scat{A} \go \Set$ was constructed by a two-step process:
first we defined $\catI: \scat{A} \go \Cat$, then we took $I(a)$ to be the set
of connected-components of $\catI(a)$.  To understand $I$ we therefore need to
understand the relation of connectedness in the category $\catI(a)$.  We now
begin to analyze this relation.  This analysis is what gives the theory much
of its substance, and we will return to it later too
(\ref{cor:complex-span},~\ref{propn:relationsdetermineequality}).

Notation: if $\Gamma$ is the limit of a diagram
\begin{equation}        \label{eq:sequence-for-limit}
\cdots \go \Gamma_3 \go \Gamma_2 \go \Gamma_1
\end{equation}
in some category, $\pr_n$ will denote both the projection $\Gamma \go
\Gamma_n$ and the given map $\Gamma_m \go \Gamma_n$ for any $m \geq n$.

\begin{lemma}[K\"onig~\cite{Koen}]
\label{lemma:Koenig}
The limit in $\Set$ of a diagram~(\ref{eq:sequence-for-limit}) of finite
nonempty sets is nonempty.  More precisely, for any sequence $(x_n)_{n \geq
1}$ with $x_n \in \Gamma_n$, there exists an element $y \in \Lt{n} \Gamma_n$
such that
\[
\forall r \geq 1, 
\ 
\exists n \geq r:
\ 
\pr_r (x_n) = \pr_r (y).
\]
\end{lemma}
 
\begin{proof}
Take a sequence $(x_n)_{n \geq 1}$ with $x_n \in \Gamma_n$.  We define, for
each $r \geq 1$, an infinite subset $\natset{r}$ of $\posint$ and an element
$y_r \in \Gamma_r$ such that
\begin{itemize}
\item for all $r \geq 1$, $\natset{r} \sub \natset{r-1} \cap \{ r, r+1, \ldots
\}$ (writing $\natset{0} = \posint$)
\item for all $r \geq 1$ and $n \in \natset{r}$, $\pr_r (x_n) = y_r$.
\end{itemize}
Suppose inductively that $r \geq 1$ and $\natset{r-1}$ is defined.  As $n$
runs through the infinite set $\natset{r-1} \cap \{ r, r+1, \ldots \}$,
$\pr_r (x_n)$ takes values in the finite set $\Gamma_r$, so takes some
value $y_r \in \Gamma_r$ infinitely often.  Putting
\[
\natset{r} 
= 
\{ 
n \in \natset{r-1} \cap \{ r, r+1, \ldots \}
\such
\pr_r (x_n) = y_r 
\}
\]
completes the induction.

For each $r \geq 1$ we have $y_r = \pr_r (y_{r+1})$, since we may choose $n
\in \natset{r+1}$, and then
\[
\pr_r (y_{r+1})
=
\pr_r (\pr_{r+1} (x_n))
=
\pr_r (x_n)
=
y_r.
\]
So there is a unique element $y \in \Lt{n} \Gamma_n$ such that $\pr_r (y) =
y_r$ for all $r \geq 1$.  Given $r \geq 1$, we may choose $n \in \natset{r}$,
and then $n \geq r$ and $\pr_r (x_n) = y_r = \pr_r (y)$, as required.  \done
\end{proof}

We now use the notion of distance in a category, introduced in
Warning~\ref{warning:fin}.

\begin{lemma}[Distance in a limit]
\label{lemma:distanceinalimit}
Let
\begin{equation}        \label{eq:cat-seq}
\cdots \go \scat{L}_3 \go \scat{L}_2 \go \scat{L}_1
\end{equation}
be a diagram of finite categories, and let $A, A'$ be objects of $\scat{L} =
\Lt{n} \scat{L}_n$.  Then
\[
d_{\scat{L}} (A, A')
=
\sup_{n \geq 1} 
\Bigl( d_{\scat{L}_n} \bigl( \pr_n(A), \pr_n(A') \bigr) \Bigr).
\]
\end{lemma}

\begin{proof}
Write $s = \sup_{n \geq 1} d_{\scat{L}_n} (\pr_n(A), \pr_n(A'))$.  Certainly
$d_{\scat{L}}(A, A') \geq s$, since functors are distance-decreasing
(Warning~\ref{warning:fin}).  Now let us show that $d_{\scat{L}}(A, A') \leq
s$.  Certainly this is true if $s = \infty$; assume that $s < \infty$.  For
each $n \in \nat$, let $\Gamma_n$ be the set of diagrams
\[
\begin{diagram}[size=1.5em]
        &       &\beta_1&       &       &       &\beta_2&
        &       &       &       &       &\beta_s&       &       \\
        &\ldTo  &       &\rdTo  &       &\ldTo  &       &
\rdTo   &       &       &       &\ldTo  &       &\rdTo  &       \\
\pr_n(A) = \alpha_0     
        &       &       &       &\alpha_1&      &       &
        &\      &\cdots &\      &       &       &       &
\alpha_s = \pr_n(A')                                            \\
\end{diagram}
\]
in $\scat{L}_n$.  Then $\Gamma_n$ is finite since $\scat{L}_n$ is, and
nonempty by hypothesis.  So by K\"onig's Lemma, $\Lt{n} \Gamma_n$ is
nonempty; that is, $d_{\scat{L}}(A, A') \leq s$.
\done
\end{proof}

\begin{propn}[Equality and distance]
\label{propn:equalityinIa}
Let $\scat{A}$ be a small category and $M: \scat{A} \gomod \scat{A}$ a finite
module.  Let $a \in \scat{A}$ and $(a_\blb, m_\blb), (a'_\blb, m'_\blb) \in
\catI(a)$.  Then
\begin{align*}
        &
[a_\blb, m_\blb] = [a'_\blb, m'_\blb]   \\
\iff    &
\sup_{n \geq 1}
\Bigl(
d_{\catI_n(a)}
\bigl( a_n \gobymod{m_n} \cdots \gobymod{m_1} a_0, \ 
a'_n \gobymod{m'_n} \cdots \gobymod{m'_1} a'_0 \bigr)
\Bigr)
< \infty.
\end{align*}
\end{propn}

\begin{proof}
Since $M$ is finite, each category $\catI_n(a)$ is finite.  Now apply
Lemma~\ref{lemma:distanceinalimit} with $\scat{L}_n = \catI_n(a)$. 
\done
\end{proof}

This result gives a criterion for connectedness in the category $\catI(a)$ of
complexes, purely in terms of the categories $\catI_n(a)$ of \emph{truncated}
complexes.

\minihead{$I$ is nondegenerate}

We begin with a standard categorical construction.  Any functor $\cat{X}:
\scat{B} \go \Cat$ has a \demph{category of elements} $\elt{\cat{X}}$.  An
object of $\elt{\cat{X}}$ is a pair $(b, x)$ with $b \in \scat{B}$ and $x \in
\cat{X}(b)$, and an arrow $(b, x) \go (b', x')$ is a pair $(g, \xi)$ with $g:
b \go b'$ in $\scat{B}$ and $\xi: (\cat{X}g)(x) \go x'$ in $\cat{X}(b')$.
This is related to the notion of the category of elements of a $\Set$-valued
functor $X: \scat{B} \go \Set$ (defined before~\ref{defn:finite}) by the
isomorphism $\elt{X} \iso \elt{D \of X}$, where $D: \Set \go \Cat$ is the
functor assigning to each set the corresponding discrete category.

As remarked after Corollary~\ref{cor:componentwisefilteredcategories}, a
$\Set$-valued functor $X$ is nondegenerate if and only if $\elt{X}$ is
componentwise cofiltered.  We now show that the $\Cat$-valued
functor $\catI$ has a kind of nondegeneracy property: $\elt{\catI}$ is
componentwise cofiltered.  

For the rest of this section, let $(\scat{A}, M)$ be an equational system
satisfying the solvability condition \So.  

The category of elements $\elt{\catI}$ of $\catI: \scat{A} \go \Cat$ is the
category of complexes.  For each $n\in\nat$ we have a functor $\catI_n:
\scat{A} \go \Cat$ (as in Warning~\ref{warning:fin}); its category of elements
is the category of complexes of length $n$.  Then $\elt{\catI}$ is the limit
in $\Cat$ of 
\[
\cdots \go \elt{\catI_2} \go \elt{\catI_1}.
\]

\begin{propn}
\label{propn:catInondegenerate}
$\elt{\catI}$ is componentwise cofiltered.
\end{propn}
\begin{proof}
We have to prove that every diagram $\cdot \go \cdot \og \cdot$ in
$\elt{\catI}$ can be completed to a commutative square, and that every
parallel pair $\cdot \parpairu \cdot$ can be completed to a fork.  The two
cases are very similar, so I will just do the first.

Take a diagram
\begin{equation} \label{eq:corner-in-I}
\begin{diagram}
\cdots		&\rMod^{m_3}	&a_2		&\rMod^{m_2}	&
a_1		&\rMod^{m_1}	&a_0		\\
		&		&\dTo>{f_2}	&		&
\dTo>{f_1}	&		&\dTo>{f_0}	\\
\cdots		&\rMod^{p_3}	&b_2		&\rMod^{p_2}	&
b_1		&\rMod^{p_1}	&b_0		\\
		&		&\uTo>{f'_2}	&		&
\uTo>{f'_1}	&		&\uTo>{f'_0}	\\
\cdots		&\rMod_{m'_3}	&a'_2		&\rMod_{m'_2}	&
a'_1		&\rMod_{m'_1}	&a'_0		\\
\end{diagram}
\end{equation}
of shape $\littlepullback$ in $\elt{\catI}$.  For $n \geq 1$, let $\Gamma_n$
be the set of diagrams
\begin{equation} \label{eq:factorization-chunk}
\begin{diagram}
a_n		&\rMod^{m_n}	&a_{n-1}	&\rMod^{m_{n-1}}&
\	&\cdots	&\		&\rMod^{m_2}	&a_1		\\
\uTo<{g_n}	&		&\uTo<{g_{n-1}}	&		&
	&	&		&		&\uTo>{g_1}	\\
c_n		&\rMod^{q_n}	&c_{n-1}	&\rMod^{q_{n-1}}&
\	&\cdots	&\ 		&\rMod^{q_2}	&c_1		\\
\dTo<{g'_n}	&		&\dTo<{g'_{n-1}}&		&
	&	&		&		&\dTo>{g'_1}	\\
a'_n		&\rMod_{m'_n}	&a'_{n-1}	&\rMod_{m'_{n-1}}&
\	&\cdots	&\ 		&\rMod_{m'_2}	&a'_1		\\
\end{diagram}
\end{equation}
satisfying $f_1 g_1 = f'_1 g'_1$, \ldots, $f_n g_n = f'_n g'_n$.  There are
evident projections $\Gamma_{n + 1} \go \Gamma_n$.  We will apply K\"onig's
Lemma.  

Each set $\Gamma_n$ is finite, by finiteness of $M$ and the fact that the
indexing in~\bref{eq:factorization-chunk} starts at $1$, not $0$.  I claim
that $\Gamma_n$ is also nonempty.  Indeed, \cstyle{S1} implies that
there exist $c_n$, $g_n$, and $g'_n$ making
\[
\begin{diagram}
	&		&a_n	&		&\rMod^{m_n}	&
	&a_{n-1}	&			&	\\
	&\ruTo<{g_n}	&	&\rdTo>{f_n}	&		&
\	&		&\rdTo>{f_{n-1}}	&	\\
c_n	&		&	&		&b_n		&
	&\rMod^{p_n}	&			&b_{n-1}\\
	&\rdTo<{g'_n}	&	&\ruTo>{f'_n}	&		&
	&		&\ruTo>{f'_{n-1}}	&	\\
	&		&a'_n	&		&\rMod_{m'_n}	&
	&a'_{n-1}	&			&	\\
\end{diagram}
\]
commute, and then nondegeneracy of $M$ (condition~\cstyle{ND1} at the end
of~\S\ref{sec:nondegen}) implies that the outside of this diagram can also
be filled in as
\[
\begin{diagram}
	&		&a_n	&		&\rMod^{m_n}	&
	&a_{n-1}	&			&	\\
	&\ruTo<{g_n}	&	&\ 		&		&
\ruTo<{g_{n-1}}&	&\rdTo>{f_{n-1}}	&	\\
c_n	&		&\rMod^{q_n}&		&c_{n-1}	&
	&		&			&b_{n-1}\\
	&\rdTo<{g'_n}	&	&		&		&
\rdTo<{g'_{n-1}}&	&\ruTo>{f'_{n-1}}	&	\\
	&		&a'_n	&		&\rMod_{m'_n}	&
	&a'_{n-1}.	&			&	\\
\end{diagram}
\]
Repeating this argument $(n - 2)$ times gives an element of $\Gamma_n$, as
required.

By K\"onig's Lemma, $\Lt{n} \Gamma_n$ is nonempty; that is,
diagram~\bref{eq:corner-in-I} with its rightmost column removed can be
completed to a commutative square in $\elt{\catI}$.  Using the diagram-filling
argument one more time shows that~\bref{eq:corner-in-I} can be, too.  
\done
\end{proof}

In the next few results we see that for general reasons, $\elt{\catI}$ being
componentwise cofiltered has two consequences: each category $\catI(a)$ is
also componentwise cofiltered, and $I: \scat{A} \go \Set$ is nondegenerate.

\begin{lemma}
\label{lemma:individualcategoriescomponentwisecofiltered}
Let $\cat{J}: \scat{B} \go \Cat$ be a functor on a small category
$\scat{B}$.  If $\elt{\cat{J}}$ is componentwise cofiltered then
$\cat{J}(a)$ is componentwise cofiltered for each $a \in \scat{B}$.
\end{lemma}

\begin{proof}
We have to prove that every diagram $\cdot \go \cdot \og \cdot$ in
$\cat{J}(a)$ can be completed to a commutative square, and that every parallel
pair $\cdot \parpairu \cdot$ can be completed to a fork.  Again I will just do
the first case; the second is similar.

Take a diagram
\[
\begin{diagdiag}
\omega	&		&	&		&\omega'	\\
	&\rdTo<\phi	&	&\ldTo>{\phi'}	&		\\
	&		&\chi	&		&		\\
\end{diagdiag}
\]
in $\cat{J}(a)$.  Then there is a commutative square
\[
\begin{diagdiag}
	&		&(b, \zeta)&		&		\\
	&\ldTo<{(g, \gamma)}&	&\rdTo>{(g', \gamma')}&		\\
(a, \omega)&		&	&		&(a, \omega')	\\
	&\rdTo<{(1, \phi)}&	&\ldTo>{(1, \phi')}&		\\
	&		&(a, \chi)&		&		\\
\end{diagdiag}
\]
in $\elt{\cat{J}}$.  Commutativity says that $g = g'$ and that the square 
\[
\begin{diagdiag}
	&		&(\cat{J}g)(\zeta)	&		&	\\
	&\ldTo<\gamma	&			&\rdTo>{\gamma'}&	\\
\omega	&		&			&		&\omega'\\
	&\rdTo<\phi	&			&\ldTo>{\phi'}	&	\\
	&		&\chi			&		&	\\
\end{diagdiag}
\]
in $\cat{J}(a)$ commutes, as required.
\done
\end{proof}

\begin{propn}
\label{propn:Iacomponentwisecofiltered}
$\catI(a)$ is componentwise cofiltered for each $a \in \scat{A}$.  
\done
\end{propn}

We will repeatedly use the following corollary.

\begin{cor}[Equality and spans]
\label{cor:complex-span}
Let $a \in \scat{A}$, and let $(a_\blb, m_\blb), (a'_\blb,
m'_\blb) \in \catI(a)$.  Then $[a_\blb, m_\blb] = [a'_\blb, m'_\blb]$ if and
only if there exists a span
\[
(a_\blb, m_\blb) \og \cdot \go (a'_\blb, m'_\blb)
\]
in $\catI(a)$.
\end{cor}

\begin{proof}
By Proposition~\ref{propn:Iacomponentwisecofiltered}, $\catI(a)$ has the
square-completion property; then use Lemma~\ref{lemma:connectednessbyspans}.
\done
\end{proof}

\begin{lemma}
\label{lemma:connectedcomponents}
Let $\cat{J}: \scat{B} \go \Cat$ be a functor on a small category
$\scat{B}$.  If $\elt{\cat{J}}$ is componentwise cofiltered then so is
$\elt{\Pi_0 \of \cat{J}}$.
\end{lemma}
\begin{proof}
Once again the proof splits into two similar cases.  For variety I will do
the second: that every diagram $ (a, \compt{\omega}) \parpair{f}{f'} (b,
\compt{\chi}) $ in $\elt{\Pi_0 \of \cat{J}}$ extends to a fork.

Since $\compt{(\cat{J}f)(\omega)} = \compt{\chi} = \compt{(\cat{J}f')
(\omega)}$, Lemmas~\ref{lemma:connectednessbyspans}
and~\ref{lemma:individualcategoriescomponentwisecofiltered} imply that
there exists a span
\[
\begin{diagdiag}
		&	&\xi	&	&		\\
		&\ldTo<\delta&	&\rdTo>{\delta'}&	\\
(\cat{J}f)(\omega)&	&	&	&(\cat{J}f')(\omega)\\
\end{diagdiag}
\]
in $\cat{J}(b)$.  We therefore have a finite connected diagram (solid arrows)
\[
\begin{diagdiag}
			&			&
(c, \zeta)		&
			&			\\
			&\ldGet<{(g, \gamma)}	&
			&
\rdGet>{(fg = f'g, \theta)}&			\\
(a, \omega)		&			&
			&
			&(b, \xi)		\\
			&\rdTo(4,4)>{\rdlabel{(f', 1)}}	&
			&
\ldTo(4,4)<{\ldlabel{(1, \delta)}}&		\\
\dTo<{(f, 1)}		&			&
			&
			&\dTo>{(1, \delta')}	\\
			&			&
			&
			&			\\
(b, (\cat{J}f)(\omega))	&			&
			&
			&(b, (\cat{J}f')(\omega))\\
\end{diagdiag}
\]
in $\elt{\cat{J}}$, so by hypothesis there exists a dotted commutative
diagram, giving a fork
\[
(c, \compt{\zeta})
\goby{g}
(a, \compt{\omega})
\parpair{f}{f'}
(b, \compt{\chi})
\]
in $\elt{\Pi_0 \of \cat{J}}$.  
\done
\end{proof}

\begin{propn}[Nondegeneracy]
\label{propn:setInondegenerate}
$I: \scat{A} \go \Set$ is nondegenerate.  
\end{propn}

\begin{proof}
By Proposition~\ref{propn:catInondegenerate} and
Lemma~\ref{lemma:connectedcomponents}, $\elt{I}$ is componentwise cofiltered.
By the remark after Corollary~\ref{cor:componentwisefilteredcategories}, this
is equivalent to $I$ being nondegenerate.
\done
\end{proof}

Hence $(I, \iota)$ is an $M$-coalgebra.  By Lambek's Lemma, a necessary
condition for it to be the universal solution is that $\iota$ is an
isomorphism.  We can prove this fact directly---and we need to, since it will
be used in the proof that $(I, \iota)$ is the universal solution.

\begin{cor}[Fixed point]
\label{cor:Iisafixedpoint}
$\iota: I \go M \otimes I$ is an isomorphism.
\end{cor}

\begin{proof}
It is enough to show that $\iota_a: I(a) \go (M \otimes I)(a)$ is bijective for
each $a \in \scat{A}$.  Certainly $\iota_a$ is surjective.  For
injectivity, suppose that
\[
\iota_a
\Bigl(
\bigcompt{
\cdots \gobymod{m_2} a_1 \gobymod{m_1} a
}
\Bigr)
=
\iota_a
\Bigl(
\bigcompt{
\cdots \gobymod{m'_2} a'_1 \gobymod{m'_1} a
}
\Bigr),
\]
that is,
\[
m_1
\otimes
\bigcompt{
\cdots \gobymod{m_2} a_1
}
=
m'_1
\otimes
\bigcompt{
\cdots \gobymod{m'_2} a'_1
}.
\]
By nondegeneracy of $I$ and Lemma~\ref{lemma:equalityinMtensorX},
there exist a commutative square
\[
\begin{diagdiag}
	&		&b	&		&	\\
	&\ldTo<f	&	&\rdTo>{f'}	&	\\
a_1	&		&	&		&a'_1	\\
	&\rdMod<{m_1}	&	&\ldMod>{m'_1}	&	\\
	&		&a	&		&	\\
\end{diagdiag}
\]
and an element $\compt{\cdots \gobymod{p_2} b_1 \gobymod{p_1} b} \in I(b)$
such that
\begin{eqnarray*}
\bigcompt{
\cdots \gobymod{p_2} b_1 \gobymod{f p_1} a_1
}	&
=	&
\bigcompt{
\cdots \gobymod{m_3} a_2 \gobymod{m_2} a_1
},	\\
\bigcompt{
\cdots \gobymod{p_2} b_1 \gobymod{f' p_1} a'_1
}	&
=	&
\bigcompt{
\cdots \gobymod{m'_3} a'_2 \gobymod{m'_2} a'_1
}.
\nonumber	
\end{eqnarray*}
Then 
\begin{eqnarray*}
\bigcompt{
\cdots \gobymod{m_3} a_2 \gobymod{m_2} a_1 \gobymod{m_1} a
}	&
=	&
\bigcompt{
\cdots \gobymod{p_2} b_1 \gobymod{f p_1} a_1 \gobymod{m_1} a
}	\\
	&
=	&
\bigcompt{
\cdots \gobymod{p_2} b_1 \gobymod{p_1} b \gobymod{m_1 f} a
},
\end{eqnarray*}
using the observation at~(\ref{eq:finitary}) (\S\ref{sec:construction}).
But $m_1 f = m'_1 f'$, so by symmetry of argument,
\[
\bigcompt{
\cdots \gobymod{m_2} a_1 \gobymod{m_1} a
}
=
\bigcompt{
\cdots \gobymod{m'_2} a'_1 \gobymod{m'_1} a
},
\]
as required.
\done
\end{proof}

\minihead{$I$ is the universal solution}

Consider resolutions in the coalgebra $(I, \iota)$.  Given $a \in \scat{A}$
and a complex $(a_\blb, m_\blb) \in \catI(a)$, there is a resolution of
$\compt{a_\blb, m_\blb} \in I(a)$ consisting of $(a_\blb, m_\blb)$ itself
together with, for each $n \in \nat$, the element $ \compt{ \cdots
\gobymod{m_{n+2}} a_{n+1} \gobymod{m_{n+1}} a_n } $
of $I(a_n)$.  We call this the \demph{canonical%
\label{p:tautological}
resolution} of the complex $(a_\blb, m_\blb)$.  

\begin{propn}[Double complex]
\label{propn:doublecomplex}
Let
\[
\begin{diagram}
\ddots	&		&\vdots	&		&
	&		&\vdots		\\
	&		&	&		&
	&		&\dMod>{m_3}	\\
\cdots	&\rMod^{m_2^3}	&a_2^2	&\rMod^{m_2^2}	&
a_2^1	&\rMod^{m_2^1}	&a_2^0		\\
	&		&	&		&
	&		&\dMod>{m_2}	\\
\cdots	&\rMod^{m_1^3}	&a_1^2	&\rMod^{m_1^2}	&
a_1^1	&\rMod^{m_1^1}	&a_1^0		\\
	&		&	&		&
	&		&\dMod>{m_1}	\\
\cdots	&\rMod^{m_0^3}	&a_0^2	&\rMod^{m_0^2}	&
a_0^1	&\rMod^{m_0^1}	&a_0^0		\\
\end{diagram}
\]
be a diagram satisfying
\[
\bigcompt{
\cdots 
\gobymod{m_n^3} a_n^2
\gobymod{m_n^2} a_n^1
\gobymod{m_n^1} a_n^0
}
=
\bigcompt{
\cdots 
\gobymod{m_{n+1}^2} a_{n+1}^1
\gobymod{m_{n+1}^1} a_{n+1}^0
\gobymod{m_{n+1}} a_n^0
}
\]
for all $n\in\nat$.  Then
\begin{equation}
\label{eq:dbl-cx}
\bigcompt{
\cdots \gobymod{m_0^3} a_0^2 \gobymod{m_0^2} a_0^1 \gobymod{m_0^1} a_0^0
}
=
\bigcompt{
\cdots \gobymod{m_3} a_2^0 \gobymod{m_2} a_1^0 \gobymod{m_1} a_0^0
}.
\end{equation}
\end{propn}

\begin{proof}
The left-hand side of~\bref{eq:dbl-cx} can be resolved canonically in $(I,
\iota)$ along
\[
\cdots \gobymod{m_0^2} a_0^1 \gobymod{m_0^1} a_0^0.
\]
It also has a resolution $(x_n)_{n\in\nat}$ in $(I, \iota)$ along
\[
\cdots \gobymod{m_2} a_1^0 \gobymod{m_1} a_0^0,
\]
where
\[
x_n 
=
\bigcompt{
\cdots \gobymod{m_n^2} a_n^1 \gobymod{m_n^1} a_n^0
}
\in 
I(a_n^0),
\]
since by hypothesis
\[
\iota_{a_n^0}(x_n)	
=		
\iota_{a_n^0}
\Bigl(
\bigcompt{
\cdots \gobymod{m_{n+1}^1} a_{n+1}^0 \gobymod{m_{n+1}} a_n^0
}	
\Bigr)
=		
m_{n+1} \otimes x_{n+1}.
\]
The result follows from nondegeneracy of $I$ and
Corollary~\ref{cor:resolving-complex}.
\done
\end{proof}

\begin{thm}[Universal solution in $\Set$]
\label{thm:universalsolutioninSet}
$(I, \iota)$ is the universal solution of $(\scat{A}, M)$ in $\Set$.
\end{thm}

\begin{proof}
Let $(X, \xi)$ be an $M$-coalgebra.  We show that there is a unique map $(X,
\xi) \go (I, \iota)$.

\paragraph*{Existence} Given $a \in \scat{A}$ and $x \in X(a)$, we may
choose a resolution $(a_\blb, m_\blb, x_\blb)$
of $x$ and put
\[
\ovln{\xi}_a(x) 
=
\compt{a_\blb, m_\blb}
\in 
I(a).
\]
This defines for each $a$ a function $\ovln{\xi}_a: X(a) \go I(a)$, which by
Corollary~\ref{cor:resolving-complex} is independent of choice of resolution.
The maps $(\ovln{\xi}_a)_{a \in \scat{A}}$ define a natural transformation
$\ovln{\xi}: X \go I$; that is, if $a \goby{f} a'$ is a map in $\scat{A}$ and
$x \in X(a)$ then $\ovln{\xi}_{a'} (fx) = f \ovln{\xi}_a (x)$.  For choose a
resolution $(a_\blb, m_\blb, x_\blb)$ of $x$: then
\[
( 
( \cdots \gobymod{m_2} a_1 \gobymod{f m_1} a' ),
( fx, x_1, x_2, \ldots )
)  
\]
is a resolution of $fx$, so
\[
\ovln{\xi}_{a'} (fx)	
=	
\bigcompt{ \cdots \gobymod{m_2} a_1 \gobymod{f m_1} a' }
=	
f \bigcompt{ \cdots \gobymod{m_2} a_1 \gobymod{m_1} a }
=	
f \ovln{\xi}_a (x).
\]
Moreover, $\ovln{\xi}$ is a map of coalgebras; that is, if $a \in \scat{A}$
and $x \in X(a)$ then
\[
(M \otimes \ovln{\xi})_a \xi_a (x) 
= 
\iota_a \ovln{\xi}_a (x).
\]
For choose a resolution $(a_\blb, m_\blb, x_\blb)$ of $x$: then
\[
( 
( \cdots \gobymod{m_3} a_2 \gobymod{m_2} a_1 ),
( x_1, x_2, x_3, \ldots )
)  
\]
is a resolution of $x_1$, so
\begin{eqnarray*}
(M \otimes \ovln{\xi})_a \xi_a (x)	&
=	&
(M \otimes \ovln{\xi})_a (m_1 \otimes x_1)			
=
m_1 \otimes \ovln{\xi}_{a_1} (x_1)				\\
	&
=	&
m_1 \otimes \bigcompt{ \cdots \gobymod{m_2} a_1 }
=
\iota_a \ovln{\xi}_a (x).
\end{eqnarray*}

\paragraph*{Uniqueness} Let $\twid{\xi}: (X, \xi) \go (I, \iota)$ be a map of
coalgebras, $a \in \scat{A}$, and $x \in X(a)$.  We show that $\twid{\xi}_a
(x) = \ovln{\xi}_a (x)$.

Choose a resolution $(a_\blb, m_\blb, x_\blb)$ of $x$, and for each $n \in
\nat$, write
\[
\twid{\xi}_{a_n} (x_n)
=
\bigcompt{
\cdots \gobymod{m_n^2} a_n^1 \gobymod{m_n^1} a_n^0 = a_n
}.
\]
For each $n\in\nat$, we have
\begin{eqnarray*}
(M \otimes \twid{\xi})_{a_n} \xi_{a_n} (x_n)	&
=	&
(M \otimes \twid{\xi})_{a_n} (m_{n+1} \otimes x_{n+1})	\\
	&
=	&
m_{n+1} \otimes \twid{\xi}_{a_{n+1}} (x_{n+1})		\\
	&
=	&
m_{n+1} \otimes 
\bigcompt{ 
\cdots 
\gobymod{m_{n+1}^2} a_{n+1}^1 \gobymod{m_{n+1}^1} a_{n+1}^0 = a_{n+1}
}	\\
	&
=	&
\iota_{a_n}
\Bigl(
\bigcompt{
\cdots 
\gobymod{m_{n+1}^1} a_{n+1}^0 = a_{n+1} \gobymod{m_{n+1}} a_n
}\Bigr).
\end{eqnarray*}
On the other hand, 
\[
(M \otimes \twid{\xi})_{a_n} \xi_{a_n} (x_n)
=
\iota_{a_n} \twid{\xi}_{a_n} (x_n)
\]
since $\twid{\xi}$ is a map of coalgebras.  Since $\iota_{a_n}$ is injective
(Corollary~\ref{cor:Iisafixedpoint}),
\begin{eqnarray*}
\bigcompt{
\cdots 
\gobymod{m_{n+1}^1} a_{n+1}^0 = a_{n+1} \gobymod{m_{n+1}} a_n
}	&
=	&
\twid{\xi}_{a_n} (x_n)	\\
	&
=	&
\bigcompt{
\cdots \gobymod{m_n^2} a_n^1 \gobymod{m_n^1} a_n^0 = a_n
}
\end{eqnarray*}
for each $n\in\nat$.  So Proposition~\ref{propn:doublecomplex} applies, and
\[
\bigcompt{
\cdots \gobymod{m_0^2} a_0^1 \gobymod{m_0^1} a_0^0 
}
=
\bigcompt{
\cdots \gobymod{m_2} a_1 \gobymod{m_1} a_0
}
\in
I(a),
\]
that is, $\twid{\xi}_a(x) = \ovln{\xi}_a (x)$, as required.
\done
\end{proof}

\minihead{The canonical representation of an element of the universal
solution} 

An element of the universal solution is an equivalence class of complexes.
One might not expect every element to have a \emph{canonical} complex
representing it, since, for example, not every real number has a
\emph{canonical} decimal expansion.  So it is perhaps a surprise that under
very mild conditions on $\scat{A}$, satisfied in every example of an
equational system in this paper, every element of the universal solution does
indeed have a canonical representing complex.

This result was suggested to me by Andr\'e Joyal, who has kindly allowed me to
include it here.  No later results depend on it.  

The main theorem is:

\begin{thm}[Canonical representation]
\label{thm:can-rep}
Suppose that $\scat{A}$ is Cauchy-complete.  Let $a \in \scat{A}$.  Then each
connected-component of $\catI(a)$ has an initial object.
\end{thm}

(Recall our standing assumption for this section that $(\scat{A}, M)$ is an
equational system satisfying \So.)

\begin{example*}{Interval}      \label{eg:can-interval}
In Example~\ref{eg:construction-Freyd} we considered the Freyd system
$(\scat{A}, M)$ and the various representations of real numbers in $[0, 1]$.
We saw that $1/2 \in [0, 1]$ is represented by infinitely many complexes
((\ref{eq:half-1})--(\ref{eq:half-5})); that is, the connected-component of
$\catI(1)$ corresponding to $1/2 \in I(1)$ has infinitely many objects.
Its initial object is the complex~(\ref{eq:half-1}), which can therefore be
regarded as the \emph{canonical} representation of $1/2$.
\end{example*}

We now prepare to prove Theorem~\ref{thm:can-rep}.  

\begin{lemma}   \label{lemma:pres-by-nd}
Nondegenerate functors preserve finite connected limits.
\end{lemma}
\begin{proof}
Let $\scat{B}$ be a small category and $X: \scat{B} \go \Set$ a
nondegenerate functor.  We have
\[
X 
\iso
\left(
\scat{B} \goby{\mathrm{Yoneda}} 
\pshf{\scat{B}} \goby{\dashbk \otimes X} 
\Set
\right)
\]
and the Yoneda embedding preserves limits.
\done
\end{proof}

\begin{lemma}[Connected limits of truncated complexes]
\label{lemma:conn-trunc}
Let $\scat{K}$ be a finite connected category.  If $\scat{A}$ has limits of
shape $\scat{K}$ then so does $\catI_n(a)$, for every $a \in \scat{A}$ and $n
\in \nat$.
\end{lemma}

\paragraph*{Remark} The same proof shows that under the same hypotheses,
$\catI(a)$ has limits of shape $\scat{K}$.  The projections $\catI(a) \go
\catI_n(a)$ and $\catI_n(a) \go \scat{A}$ all preserve those limits.

\begin{proof}
Suppose that $\scat{A}$ has limits of shape $\scat{K}$, let $a \in \scat{A}$,
and let $n \in \nat$.  Let $D: \scat{K} \go \catI_n(a)$ be a diagram, and
write its value at an object $k \in \scat{K}$ as
\[
D(k) 
=
(a^k_n \gobymod{m^k_n} \cdots \gobymod{m^k_1} a^k_0 = a).
\]
We construct a limit cone on $D$.

For each $r \in \{ 1, \ldots, n \}$, the diagram
\[
\begin{array}{cccc}
D_r:    &\scat{K}       &\go            &\scat{A},      \\
        &k              &\goesto        &a^k_r
\end{array}
\]
has a limit cone $(a_r \goby{p^k_r} a^k_r)_{k \in \scat{K}}$.  There is also a
trivial limit cone $(a_0 \goby{p^k_0} a^k_0)_{k \in \scat{K}}$ with $a_0 = a$
and $p^k_0 = 1_a$.  By Lemma~\ref{lemma:pres-by-nd}, the functor $M(a_r,
\dashbk)$ preserves limits of shape $\scat{K}$ for each $r$; hence
\[
\Bigl( 
M(a_r, a_{r - 1}) 
\goby{p^k_{r - 1} \cdot \dashbk} 
M(a_r, a^k_{r - 1})
\Bigr)_{k \in \scat{K}}
\]
is a limit cone.  It follows that, for each $r$, there is a unique sector
$m_r: a_r \gomod a_{r - 1}$ such that $p^k_{r - 1} m_r = m^k_r p^k_r$ for all
$k \in \scat{K}$.  This gives a cone
\[
\left(
\begin{diagram}
a_n     &\rMod^{m_n}    &a_{n-1}        &\rMod^{m_{n-1}}&
\       &\cdots         &         
\       &\rMod^{m_1}    &a_0 = a        \\
\dTo<{p^k_n}&           &\dTo>{p^k_{n-1}}&              &
        &               &
        &               &\dTo>{p^k_0 = 1_a}\\
a^k_n   &\rMod_{m^k_n}  &a^k_{n-1}      &\rMod_{m^k_{n-1}}&
\       &\cdots         &         
\       &\rMod_{m^k_1}  &a^k_0 = a      \\
\end{diagram}
\right)_{k \in \scat{K}}
\]
on $D$, and it straightforward to check, using Lemma~\ref{lemma:pres-by-nd}
again, that it is a limit cone.  
\done
\end{proof}

Let $\scat{H}$ be the two-element monoid consisting of the identity and an
idempotent.  A diagram of shape $\scat{H}$ in a category $\cat{C}$ is an
idempotent in $\cat{C}$, and a limit---or equally, a colimit---of such a
diagram is a splitting of the idempotent.  So Lemma~\ref{lemma:conn-trunc}
implies: 

\begin{cor}
\label{cor:CC-trunc}
Suppose that $\scat{A}$ is Cauchy-complete.  Then $\catI_n(a)$ is
Cauchy-complete for every $a \in \scat{A}$ and $n \in \nat$.  
\done
\end{cor}

We will use the fact that the filtered cocompletion of a small category
$\scat{B}$ is $\Flat(\scat{B}^\op, \Set)$, the full subcategory of
$\pshf{\scat{B}}$ consisting of the flat functors (\S6.3, 6.5 of~\cite{Bor1}).
In particular, $\Flat(\scat{B}^\op, \Set)$ has filtered colimits, and any
filtered colimit preserved by the Yoneda embedding $\mathbf{y}: \scat{B} \go
\Flat(\scat{B}^\op, \Set)$ is \demph{absolute}, that is, preserved by every
functor out of $\scat{B}$.

\begin{lemma}[Finite Cauchy-complete categories]
\label{lemma:finite-CC}
Let $\scat{B}$ be a finite category.  Then
\[
\scat{B} \textrm{ is Cauchy-complete}
\iff
\scat{B} \textrm{ has filtered colimits}
\iff
\scat{B} \textrm{ has cofiltered limits}.
\]
In that case, filtered colimits and cofiltered limits in $\scat{B}$ are
absolute. 
\end{lemma}

\begin{proof}
By duality, we need only consider filtered colimits.  Since the category
$\scat{H}$ is filtered, a category with filtered colimits is always
Cauchy-complete.  Conversely, suppose that $\scat{B}$ is Cauchy-complete.  By
finiteness and Proposition~\ref{propn:flat-on-fin}(\ref{part:ff-flat}), the
functor $\mathbf{y}: \scat{B} \go \Flat(\scat{B}^\op, \Set)$ is an
equivalence; so by the remarks above, $\scat{B}$ has filtered colimits and
they are absolute.  \done
\end{proof}

\begin{propn}[Cofiltered limits of truncated complexes]
\label{propn:cofilt-trunc}
Suppose that $\scat{A}$ is Cauchy-complete.  Then for each $n \in \nat$ and $a
\in \scat{A}$, the category $\catI_n(a)$ has filtered colimits and cofiltered
limits, and they are absolute.  
\end{propn}

\begin{proof}
Follows from Corollary~\ref{cor:CC-trunc}, Lemma~\ref{lemma:finite-CC}, and
finiteness of $\catI_n(a)$.
\done
\end{proof}

We now use these results about truncated complexes to deduce results about
ordinary, non-truncated, complexes.

\begin{lemma}
\label{lemma:lim-lim}
Let $\scat{K}$ be a category and let $a \in \scat{A}$.  Suppose that for all
$n \in \nat$, the category $\catI_n(a)$ has limits of shape $\scat{K}$ and the
projection functor $\pr_n: \catI_{n+1}(a) \go \catI_n(a)$ preserves them.
Then $\catI(a)$ has limits of shape $\scat{K}$.
\end{lemma}

\paragraph*{Remark}
This almost follows from the fact that $\catI(a) = \Lt{n} \catI_n(a)$.
However, this is a \emph{strict} (1-categorical) limit, whereas the functors
$\pr_n$ only preserve limits in the usual sense that a certain canonical map is
an \emph{isomorphism}.  One can, for instance, write down a
sequence~(\ref{eq:cat-seq}) of categories and functors in which each of
the categories has a terminal object and each of the functors preserves them,
but the limit does not have a terminal object.  Something extra is therefore
needed in order to build limits in $\catI(a)$.

\begin{proof}
First observe that each functor $\pr_n$ is an
\demph{isofibration}~\cite{JT,LackHTA}: given an object $\alpha \in
\catI_{n+1}(a)$ and an isomorphism $j: \pr_n(\alpha) \go \beta$ in
$\catI_n(a)$, there exists an isomorphism $i: \alpha \go \alpha'$ such that
$\pr_n(\alpha') = \beta$ and $\pr_n(i) = j$.

Now take a diagram $D: \scat{K} \go \catI(a)$.  Write $D_n: \scat{K} \go
\catI_n(a)$ for the composite of $D$ with the projection $\catI(a) \go
\catI_n(a)$.  We may choose a limit cone on $D_1$; then, using the
isofibration property, a limit cone on $D_2$ whose image in $\catI_1(a)$ is
the chosen cone on $D_1$; and so on.  This compatible sequence of cones
defines a cone on $D$ itself, which is a limit cone.  \done
\end{proof}

Proposition~\ref{propn:cofilt-trunc} and Lemma~\ref{lemma:lim-lim} together
imply: 
\begin{propn}[Cofiltered limits of complexes]
\label{propn:cofilt-complexes}
Suppose that $\scat{A}$ is Cauchy-complete.  Then for all $a \in \scat{A}$,
the category $\catI(a)$ has cofiltered limits.
\done
\end{propn}

We can now prove that every element of $I(a)$ has a canonical complex
representing it.

\begin{prooflike}{Proof of Theorem~\ref{thm:can-rep}}
Let $a \in \scat{A}$ and let $\scat{K}$ be a connected-component of
$\catI(a)$.  Condition \So\ implies that $\catI(a)$ is componentwise
cofiltered (Proposition~\ref{propn:Iacomponentwisecofiltered}), so $\scat{K}$
is cofiltered.  Then by Proposition~\ref{propn:cofilt-complexes}, the
inclusion $\scat{K} \rIncl \catI(a)$ has a limit.  But since $\scat{K}$ is a
connected-component, a limit cone on this inclusion functor amounts to a limit
cone on the identity functor $\scat{K} \go \scat{K}$, which by Lemma~X.1
of~\cite{CWM} amounts to an initial object of $\scat{K}$.
\done
\end{prooflike}


\section{Topological proofs}
\label{sec:Top-proofs}

Fix an equational system $(\scat{A}, M)$ satisfying the solvability condition
\So.  In this section we show that $(I, \iota)$, with the topology defined
in~\S\ref{sec:construction}, is an $M$-coalgebra in $\Top$, and indeed the
universal solution in $\Top$.  Along the way we prove that each space $I(a)$
is compact Hausdorff.

\minihead{$I(a)$ is Hausdorff}

We begin with the Hausdorff property.  Recall the sets $V_{m_1, \ldots, m_n}$
defined in~\S\ref{sec:construction}.  Define, for each $n\in\nat$ and
$a\in\scat{A}$, a binary relation $R_n^a$ on $I(a)$ by
\begin{align*}
R_n^a   &
=
\bigcup
\{
V_{p_1, \ldots, p_n} \times V_{p_1, \ldots, p_n}
\such
(b_n \gobymod{p_n} \cdots \gobymod{p_1} b_0)
\in 
\catI_n(a)
\}      \\
        &
\sub
I(a) \times I(a).
\end{align*}
Equivalently, $(t, t') \in R_n^a$ when there exists $(b_n \gobymod{p_n} \cdots
\gobymod{p_1} b_0) \in \catI_n(a)$ such that $t$ and $t'$ can both be written
in the form
\[
\bigcompt{
\cdots \gomod \cdot \gomod 
b_n \gobymod{p_n} \cdots \gobymod{p_1} b_0
}.
\]
As a subset of $I(a) \times I(a)$, $R_n^a$ is closed, by finiteness of
$\catI_n(a)$.  As a relation, $R_n^a$ is reflexive and symmetric, but not
in general transitive: for example, in the Freyd system,
$
R_1^1 
=
[0, \half]^2 \cup [\half, 1]^2
\sub
[0, 1]^2.
$

Given a set $S$, write $\Delta_S$ for the diagonal $\{ (s, s) \such s \in S\}
\sub S \times S$.

\begin{propn}[Relations determine equality]
\label{propn:relationsdetermineequality}
For each $a \in \scat{A}$, we have $\bigcap_{n\in\nat} R_n^a =
\Delta_{I(a)}$. 
\end{propn}

\begin{proof}
Certainly $\bigcap_{n\in\nat} R_n^a \supseteq \Delta_{I(a)}$.  Conversely, let
$(t, t') \in \bigcap_{n\in\nat} R_n^a$, writing
\begin{eqnarray*}
t	&=	&
\compt{
\cdots \gobymod{m_2} a_1 \gobymod{m_1} a_0 = a
},	\\
t'	&=	&
\compt{
\cdots \gobymod{m'_2} a'_1 \gobymod{m'_1} a'_0 = a
}.
\end{eqnarray*}
For each $n\in\nat$, we may choose $(b^n_n \gobymod{p^n_n} \cdots
\gobymod{p^1_n} b^0_n) \in \catI_n(a)$ such that $t, t' \in V_{p_n^1, \ldots,
p_n^n}$.  By Corollary~\ref{cor:complex-span}, there is for each $n\in\nat$ a
span in $\catI(a)$ of the form
\[
\begin{diagram}
\cdots	&
\rMod^{m_{n+2}}	&a_{n+1}	&
\rMod^{m_{n+1}}	&a_n		&
\rMod^{m_n}	&\ 		&
\cdots	&
\		&\rMod^{m_1}	&a_0 = a	\\
	&
		&\uTo		&
		&\uTo		&
		&		&
	&
		&		&\uTo>{1_a}	\\
\cdots	&
\rMod		&\cdot		&
\rMod		&\cdot		&
\rMod		&\ 		&
\cdots	&
\		&\rMod		&a		\\
	&
		&\dTo		&
		&\dTo		&
		&		&
	&
		&		&\dTo>{1_a}	\\
\cdots	&
\rMod		&\cdot		&
\rMod		&b^n_n		&
\rMod_{p^n_n}	&\ 		&
\cdots	&
\		&\rMod_{p^1_n}	&b^0_n = a.	\\
\end{diagram}
\]
Applying the projection functor $\catI(a) \go \catI_n(a)$, we have 
\[
d_{\catI_n(a)}
\bigl( a_n \gobymod{m_n} \cdots \gobymod{m_1} a_0, \ 
b^n_n \gobymod{p^n_n} \cdots \gobymod{p^1_n} b^0_n \bigr)
\leq 1.
\]
The same is true for $t'$, so by the triangle inequality,
\[
d_{\catI_n(a)}
\bigl( a_n \gobymod{m_n} \cdots \gobymod{m_1} a_0, \ 
a'_n \gobymod{m'_n} \cdots \gobymod{m'_1} a'_0 \bigr)
\leq 2
\]
for each $n \in \nat$.  So by Proposition~\ref{propn:equalityinIa}, $t = t'$.  
\done
\end{proof}

\begin{cor}[Hausdorff]
\label{cor:Iahausdorff}
For each $a \in \scat{A}$, the space $I(a)$ is Hausdorff.
\end{cor}

\begin{proof}
$\Delta_{I(a)}$ is closed in $I(a) \times I(a)$, being the intersection of the
closed subsets $R_n^a$.  
\done
\end{proof}

\begin{cor}[Singletons]
\label{cor:singletons}
Let $a \in \scat{A}$ and $(a_\blb, m_\blb) \in \catI(a)$.  Then $\bigcap_{r \in
\nat} V_{m_1, \ldots, m_r} = \{ \compt{a_\blb, m_\blb} \}$. 
\end{cor}

\begin{proof}
By Proposition~\ref{propn:relationsdetermineequality}, the left-hand side has
at most one element; but clearly $\compt{a_\blb, m_\blb}$ is an element.
\done
\end{proof}

\minihead{$(I, \iota)$ is a topological coalgebra}

By definition, an element of the \emph{set} $I(a)$ is an equivalence class of
elements of $\oba\catI(a)$, and the coalgebra structure on $I$ is induced by
the coalgebra structure on $\oba\catI$ via the quotient map $\pi: \oba\catI
\go I$.  The next phase of the proof is to show that, in a similar sense, $(I,
\iota)$ is a quotient of $(\oba\catI, \iota)$ as a coalgebra in $\Top$.

For this to make sense, we need to put a topology on $\oba\catI$.  For each $a
\in \scat{A}$, the set $\oba\catI(a)$ is the limit of the diagram of finite
sets
\[
\cdots \go \oba\catI_2(a) \go \oba\catI_1(a).
\]
Equipping each set $\oba\catI_n(a)$ with the discrete topology and taking the
limit in $\Top$ gives a topology on $\oba\catI(a)$ (the profinite topology).
We used the same construction in~\S\ref{sec:dsss}: writing $\oba\scat{A}$ for
the discrete category with the same objects as $\scat{A}$, there is an evident
discrete equational system $(\oba\scat{A}, M)$, and its universal solution is
$(\oba\catI, \iota)$.

In this way, $\oba\catI$ becomes a functor $\scat{A} \go \Top$.  Each space
$\oba\catI(a)$ is compact Hausdorff.  Hence, recalling
from~(\ref{eq:obcatI-nondegen}) that the $\Set$-valued functor $\oba\catI$ is
nondegenerate, the $\Top$-valued functor $\oba\catI$ is nondegenerate.
The maps $\iota_a$ are continuous, since in~(\ref{eq:iota-as-comp}) the first
map is a homeomorphism and the second is a topological quotient map.  So
$(\oba\catI, \iota)$ is a coalgebra in $\Top$.

We will show that for each $a$, the map $\pi_a: \oba\catI(a) \go I(a)$ exhibits
$I(a)$ as a topological quotient of $\oba\catI(a)$.  From that we will deduce
that $(I, \iota)$ too is a coalgebra in $\Top$.

\begin{lemma}[Membership of basic closed sets]
\label{lemma:basicclosedsets}
Let $a \in \scat{A}$, $n \in \nat$, $(a_\blb, m_\blb) \in \catI(a)$, and
$
(b_n \gobymod{p_n} \cdots \gobymod{p_1} b_0)
\in
\catI_n(a). 
$
Then 
\[
\compt{a_\blb, m_\blb} \in V_{p_1, \ldots, p_n}
\iff
\textrm{for all } r\in \nat,\ 
V_{m_1, \ldots, m_r} \cap V_{p_1, \ldots, p_n} \neq \emptyset.
\]
\end{lemma}

\begin{proof}
`$\Rightarrow$' is trivial.  For `$\Leftarrow$', we may choose for each
$r\in\nat$ complexes
\begin{eqnarray*}
\alpha_r	&=	&
(
\cdots
\gobymod{m_{r+2}^r} a_{r+1}^r \gobymod{m_{r+1}^r} a_r \gobymod{m_r}
\cdots
\gobymod{m_1} a_0 = a
),	\\
\beta_r	&=	&
(
\cdots
\gobymod{p_{n+2}^r} b_{n+1}^r \gobymod{p_{n+1}^r} b_n \gobymod{p_n}
\cdots
\gobymod{p_1} b_0 = a
) 
\end{eqnarray*}
such that $\compt{\alpha_r} = \compt{\beta_r}$.  By
Corollary~\ref{cor:complex-span}, there is for each $r\in\nat$ a span
\[
\alpha_r 
\og \cdot \go 
\beta_r
\]
in $\catI(a)$.  Applying K\"onig's Lemma~(\ref{lemma:Koenig}) to the limit
$\oba\catI(a) = \Lt{r} \oba\catI_r(a)$ and the elements $\pr_r(\beta_r) \in
\oba\catI_r(a)$ gives a complex $\beta \in \catI(a)$ with the following
property: 
\begin{condition}
for all $r \in \nat$, there exists $k \geq r$ such that $\pr_r(\beta) =
\pr_r(\beta_k)$.
\end{condition}
Taking $r = n$ gives $\pr_n(\beta) = (b_n \gobymod{p_n} \cdots \gobymod{p_1}
b_0)$.  Hence $\compt{\beta} \in V_{p_1, \ldots, p_n}$.  

I claim that $\compt{a_\blb, m_\blb} = \compt{\beta}$; the result will follow.
Indeed, let $r \in \nat$.  Choose $k \geq r$ such that $\pr_r(\beta) =
\pr_r(\beta_k)$.  We have $d_{\catI(a)} (\alpha_k, \beta_k) \leq 1$, so,
applying $\pr_r: \catI(a) \go \catI_r(a)$,
\[
d_{\catI_r(a)} (\pr_r(a_\blb, m_\blb), \pr_r(\beta))
=
d_{\catI_r(a)} (\pr_r(\alpha_k), \pr_r(\beta_k)) 
\leq 
1.
\]
So by Proposition~\ref{propn:equalityinIa}, $\compt{a_\blb, m_\blb} =
\compt{\beta}$, as required.  
\done
\end{proof}

\begin{propn}[Topological quotient]
\label{propn:Iaasaquotient}
For each $a\in\scat{A}$, the canonical surjection $\pi_a: \oba\catI(a)
\go I(a)$ is a topological quotient map.
\end{propn}

\begin{proof}
First, $\pi_a$ is continuous.  Let $n\in\nat$ and $(b_n \gobymod{p_n}
\cdots \gobymod{p_1} b_0) \in \catI_n(a)$; we must show that $\pi_a^{-1}
V_{p_1, \ldots, p_n}$ is a closed subset of $\oba\catI(a)$.  By
Lemma~\ref{lemma:basicclosedsets},  
\[
\pi_a^{-1} V_{p_1, \ldots, p_n} 
=
\bigcap_{r \in \nat} \pr_r^{-1} W_r
\]
where 
\[
W_r 
=
\{ 
(a_r \gobymod{m_r} \cdots \gobymod{m_1} a_0) \in \oba\catI_r(a)
\such
V_{m_1, \ldots, m_r} \cap V_{p_1, \ldots, p_n} \neq \emptyset
\}.
\]
But each space $\oba\catI_r(a)$ is discrete and each map $\pr_r$ is
continuous, so $\bigcap_{r\in\nat} \pr_r^{-1} W_r$ is closed, as required.

Since $\oba\catI(a)$ is compact and $I(a)$ is Hausdorff, $\pi_a$ is closed.
So $\pi_a$ is a continuous closed surjection, and therefore a quotient map.
\done
\end{proof}

\begin{cor}[Compactness]
\label{cor:Iacompact}
For each $a \in \scat{A}$, the space $I(a)$ is compact.  
\done
\end{cor}

\begin{cor}[Topological coalgebra]
\label{cor:topologicalcoalgebra}
$(I, \iota)$ is an $M$-coalgebra in $\Top$.  
\end{cor}

\begin{proof}
First we have to show that for each map $f: a \go a'$ in $\scat{A}$, the map
$If: I(a) \go I(a')$ is continuous and closed.  There is a commutative square
\[
\begin{diagram}
\oba\catI(a)	&\rQt^{\pi_a}	&I(a)		\\
\dTo<{\oba\catI f}&		&\dTo>{If}	\\
\oba\catI(a')	&\rQt_{\pi_{a'}}&I(a')		\\
\end{diagram}
\]
in which $\pi_a$ is a topological quotient map and $\oba\catI(f)$ and
$\pi_{a'}$ are continuous, so $If$ is also continuous.  But $I(a)$ is compact
and $I(a')$ Hausdorff, so $If$ is closed.  

We also have to show that for each $a \in \scat{A}$, the map $\iota_a: I(a)
\go (M \otimes I)(a)$ is continuous.  This is proved by a similar argument,
using the square
\[
\begin{diagram}
\oba\catI(a)		&\rQt^{\pi_a}			&I(a)		\\
\dTo<{\iota_a}		&				&\dTo>{\iota_a}	\\
(M\otimes \oba\catI)(a)	&\rTo_{(M \otimes \pi)_a}	&(M \otimes I)(a).\\
\end{diagram}
\]
\done
\end{proof}

\minihead{$I$ is the terminal $\Top$-coalgebra}

Our final task is to prove that for any $M$-coalgebra $(X, \xi)$ in $\Top$,
the unique map $\ovln{\xi}: (X, \xi) \go (I, \iota)$ of coalgebras in $\Set$
is continuous.  To do this we show that the inverse image of each basic closed
set $V_{p_1, \ldots, p_n}$ is closed, where $n \in \nat$ and $b_n
\gobymod{p_n} \cdots \gobymod{p_1} b_0 = a$.

Some care is needed in describing this inverse image.  Given an element $x \in
X(a)$, the complexes along which $x$ can be resolved all lie in the same
connected-component of $\catI(a)$, namely $\ovln{\xi}_a(x)$.  However, there
may be complexes in this component along which $x$ cannot be resolved.
So if we write 
\[
V^X_{p_1, \ldots, p_n} \sub X(a)
\]
for the set of elements of $X(a)$ that can be resolved along some complex
of the form
\[
\cdots \gomod \cdot \gomod b_n \gobymod{p_n} \cdots \gobymod{p_1} b_0 = a,
\]
then
\begin{equation}
\label{eq:inv-inc}
V^X_{p_1, \ldots, p_n}
\sub
\ovln{\xi}_a^{-1} V_{p_1, \ldots, p_n}
\end{equation}
but the inclusion may be strict.  The following example illustrates this.

\begin{example} \label{eg:Freyd-inverse-larger}
Let $(\scat{A}, M)$ be the Freyd system.  Choose an endpoint-preserving
continuous map $\xi_1: [0, 1] \go [0, 2]$ such that $\xi_1(2/3) = 2/3$; this
defines an $M$-coalgebra structure $\xi$ on $X = ( \{ \star \} \parpair{0}{1}
[0, 1] )$.  The element $2/3 \in X(1)$ has a unique resolution, which is along
the complex
\[
\cdots \gobymod{[0, \half]} 1 \gobymod{[0, \half]} 1.
\]
Hence
\[
\ovln{\xi}_1(2/3)
=
0
=
\bigcompt{\cdots \gobymod{\id} 0 \gobymod{\id} 0 \gobymod{0} 1}
\in
V_{p_1},
\]
where $p_1 = 0: 0 \gomod 1$.  So $2/3 \in \ovln{\xi}_1^{-1} V_{p_1}$, even
though $2/3$ cannot be resolved along any complex ending in $p_1$.
\end{example}

(The notation $V^X_{p_1, \ldots, p_n}$ is explained by the fact that
$V^I_{p_1, \ldots, p_n} = V_{p_1, \ldots, p_n}$.  This follows
from~(\ref{eq:inv-inc}) and the existence of canonical
resolutions~(\S\ref{sec:set}).)

\begin{lemma}[Inverse image of basic closed sets]
\label{lemma:inv-image}
Let $a \in \scat{A}$, $n \in \nat$, and $(b_n \gobymod{p_n} \cdots
\gobymod{p_1} b_0) \in \catI_n(a)$.  Let $(X, \xi)$ be an $M$-coalgebra in
$\Set$.  Then 
\begin{equation}
\label{eq:inverseimage}
\ovln{\xi}_a^{-1} V_{p_1, \ldots, p_n}
=
\bigcap_{r \in \nat} \bigcup
V^X_{m_1, \ldots, m_r}
\end{equation}
where the union is over all $(a_r \gobymod{m_r} \cdots \gobymod{m_1} a_0) \in
\catI_r(a)$ such that 
\begin{equation}
\label{eq:int-nonempty}
V_{m_1, \ldots, m_r} 
\cap
V_{p_1, \ldots, p_n}
\neq
\emptyset.
\end{equation}
\end{lemma}

\begin{proof}
Let $x \in \ovln{\xi}_a^{-1} V_{p_1, \ldots, p_n}$, and choose a complex
$(a_\blb, m_\blb)$ along which $x$ can be resolved.  Then for all $r$,
\[
\compt{a_\blb, m_\blb}
=
\ovln{\xi}_a (x)
\in
V_{m_1, \ldots, m_r} \cap V_{p_1, \ldots, p_n},
\]
and in particular~(\ref{eq:int-nonempty}) holds.  Also $x \in V^X_{m_1,
\ldots, m_r}$ by definition, so $x$ is in the right-hand side
of~(\ref{eq:inverseimage}).

Conversely, let $x$ be an element of the right-hand side
of~(\ref{eq:inverseimage}).  By K\"onig's
Lemma~(\ref{lemma:Koenig}), we may choose a complex $(a_\blb, m_\blb) \in
\catI(a)$ such that for all $r$, (\ref{eq:int-nonempty}) holds and $x \in
V^X_{m_1, \ldots, m_r}$.  By Lemma~\ref{lemma:basicclosedsets},
$\compt{a_\blb, m_\blb} \in V_{p_1, \ldots, p_n}$.  Now
using~(\ref{eq:inv-inc}) and Corollary~\ref{cor:singletons}, 
\[
x 
\in 
\bigcap_{r \in \nat} V^X_{m_1, \ldots, m_r}
\sub
\ovln{\xi}_a^{-1} \bigcap_{r \in \nat} V_{m_1, \ldots, m_r}
=
\ovln{\xi}_a^{-1} \{ \compt{a_\blb, m_\blb} \}
\sub
\ovln{\xi}_a^{-1} V_{p_1, \ldots, p_n},
\]
as required.
\done
\end{proof}

This describes the inverse images of the basic closed sets.  We now
prepare to show that they are closed.

\begin{lemma}
\label{lemma:resolutionsetsclosed}
Let $(X, \xi)$ be an $M$-coalgebra in $\Top$, $r\in\nat$, and $(a_r
\gobymod{m_r} \cdots \gobymod{m_1} a_0) \in \catI_r(a)$.  Then $V^X_{m_1,
\ldots, m_r}$ is a closed subset of $X(a)$.
\end{lemma}

\begin{proof}
When $r = 0$ this is trivial.  Suppose inductively that the result holds
for $r \in \nat$, and let $(a_{r+1} \gobymod{m_{r+1}} \cdots \gobymod{m_1}
a_0) \in \catI_{r+1}(a)$.  We have
\[
V^X_{m_1, \ldots, m_{r+1}}
= 
\xi_a^{-1}
\left(
m_1 \otimes V^X_{m_2, \ldots, m_{r+1}}
\right)
\]
where $m_1 \otimes S$ means the image of a subset $S \sub X(a_1)$ under the
map 
\[
m_1 \otimes \dashbk: 
X(a_1) 
\go 
(M \otimes X)(a).  
\]
But $V^X_{m_2, \ldots, m_{r+1}}$ is closed by inductive hypothesis, $m_1
\otimes \dashbk$ is closed by Corollary~\ref{cor:coprojections-closed}, and
$\xi_a$ is continuous, so $V^X_{m_1, \ldots, m_{r+1}}$ is closed in $X(a)$.
\done
\end{proof}

\begin{thm}[Universal solution in $\Top$]
\label{thm:universalsolutioninTop}
$(I, \iota)$ is the universal solution of $(\scat{A}, M)$ in $\Top$.  
\end{thm}

\begin{proof}
Let $(X, \xi)$ be an $M$-coalgebra in $\Top$.  It remains to show that for
each $a \in \scat{A}$, the map $\ovln{\xi}_a: X(a) \go I(a)$ is continuous,
and this follows from Lemmas~\ref{lemma:inv-image}
and~\ref{lemma:resolutionsetsclosed}.  
\done
\end{proof}


\section{Recognizing the universal solution}
\label{sec:recognition}

We have seen that an equational system possesses a universal solution if
and only if an explicit condition \So\ holds; if so, the universal
solution is unique and can be constructed explicitly.

Few examples have been given so far.  In principle one can take any equational
system $(\scat{A}, M)$ satisfying \So\ and find the universal solution $(I,
\iota)$ by going through the explicit construction.  In practice this is
cumbersome and it is much quicker to apply one of the Recognition Theorems
proved below, as follows.

We might sometimes observe that some familiar space has a recursive
decomposition, and we might ask whether it can be characterized as the
universal solution of some equational system (or rather, one of the spaces
$I(a)$ of which the universal solution is made up).  The Recognition Theorems
provide a way to confirm such guesses.  

For example, we might note that the standard topological simplices $\Delta^n$
admit barycentric subdivision, which exhibits each simplex as a
gluing-together of smaller simplices.  This barycentric subdivision can be
expressed as an isomorphism $\Delta^\blb \iso M \otimes \Delta^\blb$, where
$M$ is a certain module and $\Delta^\blb$ is the functor $n \goesto
\Delta^n$.  Using one of the Recognition Theorems, we can confirm that
$\Delta^\blb$ is in fact the universal solution of $M$
(Example~\ref{eg:rec-bary}), thus giving a new characterization of the spaces
$\Delta^n$.  

We prove two results.  The Precise Recognition Theorem gives necessary and
sufficient conditions for a fixed point of an equational system to be a
universal solution.  The Crude Recognition Theorem gives merely sufficient
conditions, but they are very quick to check and satisfied in many examples of
interest.  These two theorems will be applied in~\S\ref{sec:examples} to yield
examples of universal solutions, and in Appendix~\ref{app:admitting} to
determine exactly which spaces are recursively realizable.

We begin by listing some of the properties enjoyed by $(I, \iota)$, the
universal solution constructed in~\S\ref{sec:construction}.  These will form
the basis of the Precise Recognition Theorem.

From now up to and including Lemma~\ref{lemma:metriconIa}, fix an equational
system $(\scat{A}, M)$ satisfying the solvability condition \So.  

The first property of $(I, \iota)$ is that it is a fixed point of $M$, that
is, an $M$-coalgebra whose structure map is an isomorphism.  (Recall that
$M$-coalgebras, and in particular fixed points, are nondegenerate by
definition.)  A fixed point $(J, \gamma)$ is a coalgebra, but can also be
regarded as an algebra $(J, \psi)$ where $\psi = \gamma^{-1}$.  By definition,
an \demph{$M$-algebra} (in $\Top$) is a nondegenerate functor $J: \scat{A} \go
\Top$ together with a map $\psi: M \otimes J \go J$.  By the universal
property of $M \otimes J$ (Appendix~\ref{app:modules}), $\psi$ amounts to a
family 
\[
\Bigl(
J(b) \goby{\psi_m} J(a)
\Bigr)_{b \gobymod{m} a}
\]
of continuous maps $\psi_m$, indexed over all sectors $m: b \gomod a$,
satisfying a naturality axiom: $\psi_{fmg} = (Jf) \of \psi_m \of (Jg)$
whenever $m$ is a sector and $f$ and $g$ are arrows in $\scat{A}$ for which
this makes sense.  

For example, the fixed point $(I, \iota)$ has algebra structure $\phi =
\iota^{-1}$, where the components $\phi_m$ are as defined
in~\S\ref{sec:construction}. 

\begin{lemma}
\label{lemma:fixedpointcomponents}
Let $(J, \gamma = \psi^{-1})$ be a fixed point of $M$ in $\Top$.  Then for
each sector $b \gobymod{m} a$, the map $J(b) \goby{\psi_m} J(a)$ is
closed.
\end{lemma}

\begin{proof}
$\psi_m$ is the composite
\[
J(b)
\goby{m \otimes \dashbk}
(M \otimes J)(a)
\rTo^{\psi_a}_\diso
J(a)
\]
and $m \otimes \dashbk$ is closed by nondegeneracy of $J$ and
Corollary~\ref{cor:coprojections-closed}. 
\done
\end{proof}

Being a fixed point alone is not enough to imply being the universal solution:
for example, the constant functor $\emptyset$ is always a fixed point and not
usually the universal solution.  A functor $J: \scat{A} \go \Set$ is
\demph{occupied} if for all $a \in \scat{A}$, 
\[
\catI(a) \neq \emptyset \implies J(a) \neq \emptyset.  
\]
When $J$ has an $M$-coalgebra structure, being occupied means that the sets
$J(a)$ are `not empty unless they have to be': for if $\catI(a)$ is empty then
$J(a)$ must be empty, since any element of $J(a)$ would have a resolution
$(a_\blb, m_\blb, x_\blb)$ with $(a_\blb, m_\blb) \in \catI(a)$.  The second
property enjoyed by $I$ is that, trivially, it is occupied.

The third property of $I$ is that the spaces $I(a)$ are metrizable:
\begin{lemma}	\label{lemma:compactmetrizable}
A compact space is metrizable if and only if it is Hausdorff and
has a countable basis of open sets.
\end{lemma}
\begin{proof}
See~\cite[IX.2.9]{Bou2} (where `compact' means compact Hausdorff).
\done
\end{proof}

One naturally asks how a metric can be defined.  There are many possible
metrics and apparently no canonical choice among them, but the following
result tells us all we need to know.  Recall (from the beginning
of~\S\ref{sec:Top-proofs}) that for each $a \in \scat{A}$ and $n\in\nat$ we
have a closed binary relation $R_n^a$ on $I(a)$, with $R_0^a \supseteq R_1^a
\supseteq \cdots$.

\begin{lemma}[Metric on $I(a)$]
\label{lemma:metriconIa}
Let $a \in \scat{A}$ and let $d$ be a metric on $I(a)$ compatible with its
topology.  Then for all $\epsln > 0$, there exists $n\in\nat$ such that
\[
(t, t') \in R_n^a
\ 
\implies 
\ 
d(t, t') < \epsln.
\]
\end{lemma}
\begin{proof}
Let $\epsln > 0$.  Since $I(a)$ is compact, so too is $d^{-1} [\epsln,
\infty)$, the inverse image of $[\epsln, \infty)$ under the continuous map $d:
I(a) \times I(a) \go [0, \infty)$.

By Proposition~\ref{propn:relationsdetermineequality} we have
$\bigcap_{n\in\nat} R_n^a = \Delta_{I(a)}$, so $\bigcap_{n\in\nat} R_n^a \cap
d^{-1} [\epsln, \infty) = \emptyset$.  But each subset $R_n^a$ is closed, so
by compactness, there is some $n\in\nat$ for which $R_n^a \cap d^{-1} [\epsln,
\infty) = \emptyset$.  \done
\end{proof}

To state the main theorem, we need a little more notation.  

Given an equational system $(\scat{A}, M)$, a fixed point $(J, \gamma =
\psi^{-1})$, and a truncated complex $a_n \gobymod{m_n} \cdots \gobymod{m_1}
a_0$, write
\[
V^J_{m_1, \ldots, m_n} = \psi_{m_1} \cdots \psi_{m_n} J(a_n), 
\]
the image of the composite map
\[
J(a_n) \goby{\psi_{m_n}} \cdots \goby{\psi_{m_1}} J(a_0).
\]
Although we will not need to know it, this is the same as the set $V^J_{m_1,
\ldots, m_n}$ defined in~\S\ref{sec:Top-proofs} for an arbitrary coalgebra
$J$.

Write $\diam(S)$ for the diameter of a metric space $S$.

\begin{thm}[Precise Recognition Theorem]
\label{thm:preciserecognition}
Let $(\scat{A}, M)$ be an equational system.  The following are equivalent
conditions on a fixed point $(J, \gamma)$ of $M$ in $\Top$:
\begin{enumerate}
\item	\label{item:pr-univ}
$(J, \gamma)$ is a universal solution of $(\scat{A}, M)$ in $\Top$
\item	\label{item:pr-met}
$J$ is occupied, and for each $a \in \scat{A}$ the space $J(a)$ is compact
and can be metrized in such a way that
\[
\inf_{n \in \nat} \sup_{m_1, \ldots, m_n} 
\diam(V^J_{m_1, \ldots, m_n})
=
0
\]
where the supremum is over all 
\[
\Bigl(
a_n \gobymod{m_n} \cdots \gobymod{m_1} a_0 = a
\Bigr)
\in \catI_n(a)
\]
\item	\label{item:pr-notmet}
$J$ is occupied; for each $a \in \scat{A}$, the space $J(a)$ is compact;
and for every complex $(a_\blb, m_\blb)$, the set $\bigcap_{n \in \nat}
V^J_{m_1, \ldots, m_n}$ has at most one element.
\end{enumerate}
\end{thm}

The only part of the proof requiring substantial work is
(\ref{item:pr-notmet})$\implies$(\ref{item:pr-univ}).  We first prepare the
ground.

Let $(\scat{A}, M)$ be an equational system, $(X, \xi)$ an $M$-coalgebra in
$\Set$, and $(J, \gamma)$ a fixed point.  Write $\psi = \gamma^{-1}$, as
usual.  A natural transformation $\omega: X \go J$ is a map of coalgebras if
and only if for all $a \in \scat{A}$, the square
\begin{equation}        \label{eq:coalg-map}
\begin{diagram}
X(a)			&\rTo^{\xi_a}		&(M \otimes X)(a)\\
\dTo<{\omega_a}         &			&
\dTo>{(M \otimes \omega)_a}					\\
J(a)			&\rTo_{\gamma_a}	&(M \otimes J)(a)\\
\end{diagram}
\end{equation}
commutes.  Let $x \in X(a)$.  Writing 
\[
\xi_a(x) = \Bigl( b \gobymod{m} a \Bigr) \otimes y, 
\]
commutativity of the square at $x$ says that $\gamma_a \omega_a (x) = m
\otimes \omega_b (y)$, or equivalently, $\omega_a(x) = \psi_m \omega_b (y)$.
Hence $\omega_a(x)$ lies in the subset 
\[
\bigcap \psi_m J(b)
\]
of $J(a)$, where the intersection is over all $b \gobymod{m} a$ and $y \in
X(b)$ such that $x = m \otimes y$.  (Note that this subset is defined without
reference to $\omega$.)  The same reasoning can be applied to each such $y$,
further constraining where in $J(a)$ the element $\omega_a(x)$ can lie; and so
on, iteratively.  This suggests the following definition.

For each $a \in \scat{A}$ and $x \in X(a)$, define a sequence $(K_n(x))_{n
\in \nat}$ of subsets of $J(a)$ by
\begin{eqnarray*}
K_0(x)          &=      &J(a),    \\
K_{n + 1}(x)    &=      &
\bigcap_{\xi(x) = m \otimes y} \psi_m K_n(y)
\end{eqnarray*}
where the intersection is over all $b \in \scat{A}$, $m \in M(b, a)$ and $y
\in X(b)$ such that $\xi_a(x) = m \otimes y$.  We will show that $(K_n(x))_{n
\in \nat}$ is a decreasing sequence of closed subsets of $J(a)$, and,
moreover, that if $(J, \gamma)$ is the universal solution then $\bigcap_n
K_n(x)$ is the singleton set $\{ \ovln{\xi}_a(x) \}$, where $\ovln{\xi}$ is
the unique coalgebra map $X \go J$.  The sets $K_n(x)$ can therefore be
thought of as approximations to $\ovln{\xi}_a(x)$.

\begin{example}
Let $(\scat{A}, M)$ be the Freyd equational system~(\S\ref{sec:sss}) and let
$(J, \gamma)$ be its universal solution $(I, \iota)$.  Let $(X, \xi)$ be the
subcoalgebra of $(I, \iota)$ defined by taking $X(0) = \emptyset$ and 
\[
X(1) 
=
\{ x \in [0, 1] \such x \textrm{ is not a dyadic rational} \}.
\]
Then for each $x \in X(1)$ there is a unique pair $(m, y)$ such that $\xi(x) =
m \otimes y$, so the intersection in the definition of $K_{n+1}(x)$ is indexed
over a one-element set.  In fact, $K_n(x) \sub [0, 1]$ is the unique interval
of the form $[r/2^n, (r + 1)/2^n]$, with $r$ an integer, containing $x$.  
\end{example}

\begin{lemma}   \label{lemma:recognition-approx}
Let $(\scat{A}, M)$ be an equational system, let $(X, \xi)$ be an
$M$-coalgebra, and let $(J, \psi)$ be a fixed point of $M$.  Then for all $a
\in \scat{A}$ and $x \in X(a)$,
\begin{enumerate}
\item   \label{part:approx-nested}
$K_0(x) \supseteq K_1(x) \supseteq \cdots$
\item   \label{part:approx-closed}
$K_n(x)$ is closed in $J(a)$ for all $n \in \nat$
\item   \label{part:approx-nat}
$(Jf)(K_n(x)) \sub K_n(fx)$ for all maps $f: a \go a'$ in $\scat{A}$ and $n
\in \nat$
\item   \label{part:approx-nonempty}
if $J(a)$ is compact and $J$ is occupied then $\bigcap_{n \in \nat} K_n(x)
\neq \emptyset$. 
\end{enumerate}
\end{lemma}

\begin{proof}
Part~(\ref{part:approx-nested}) is a straightforward induction, and
part~(\ref{part:approx-closed}) follows from 
Lemma~\ref{lemma:fixedpointcomponents} by another induction.

Part~(\ref{part:approx-nat}) is also an induction.  For $n = 0$ it is trivial.
Suppose inductively that it holds for some $n\in\nat$.  Let $t \in
K_{n+1}(x)$, and let $b' \gobymod{m'} a'$ and $y' \in X(b)$ with $\xi (fx) =
m' \otimes y'$; we have to show that $ft \in \psi_{m'} K_n(y')$.

We may choose $b \gobymod{m} a$ and $y \in X(b)$ such that $\xi(x) = m \otimes
y$.  Then $m' \otimes y' = \xi(fx) = fm \otimes y$, so by
Lemma~\ref{lemma:equalityinMtensorX} there exist a commutative square
\[
\begin{diagdiag}
	&		&c	&		&	\\
	&\ldTo<g	&	&\rdTo>{g'}	&	\\
b	&		&	&		&b'	\\
	&\rdMod<{fm}	&	&\ldMod>{m'}	&	\\
	&		&a'	&		&	\\
\end{diagdiag}
\]
and $z \in X(c)$ such that $y = gz$ and $y' = g'z$.  Now 
\[
\xi(x) = m \otimes y = m \otimes gz = mg \otimes z, 
\]
so $t \in \psi_{mg} K_n(z)$ by definition of $K_n(z)$.  Hence
\begin{eqnarray*}
ft	&
\in	&
(Jf) \psi_{mg} K_n(z)	=
\psi_{fmg} K_n(z)	
=
\psi_{m'g'} K_n(z)
=
\psi_{m'} (Jg') K_n(z)	\\
	&
\sub	&
\psi_{m'} K_n(g'z)
=	
\psi_{m'} K_n(y')
\end{eqnarray*}
(the penultimate step by inductive hypothesis), as required.

Part~(\ref{part:approx-nonempty}) will follow from compactness and
parts~(\ref{part:approx-nested}) and~(\ref{part:approx-closed}) once we know
that each set $K_n(x)$ is nonempty.  We prove this by induction on $n$ over
all $a \in \scat{A}$ and $x \in X(a)$ simultaneously.

For $n = 0$, let $a \in \scat{A}$ and $x \in X(a)$.  There exists a
resolution of $x$, and in particular an element of $\catI(a)$.  Since $J$ is
occupied, $\emptyset \neq J(a) = K_0(x)$.

Now let $n\in\nat$, $a \in \scat{A}$, and $x \in X(a)$; we have to prove that
$K_{n+1}(x) \neq \emptyset$.  Since $K_{n + 1}(x)$ is an intersection of a
family of closed subsets of a compact space, it suffices to show that the
intersection of any finite sub-family is nonempty.  So, suppose that $r \in
\nat$ and $\xi(x) = m_1 \otimes y_1 = \cdots = m_r \otimes y_r$ where $b_i
\gobymod{m_i} a$ and $y_i \in X(b_i)$; we have to show that
\begin{equation}        \label{eq:nonempty-int}
\bigcap_{i=1}^r \psi_{m_i} K_n(y_i) 
\neq 
\emptyset.
\end{equation}
When $r = 0$ this says that $J(a) \neq \emptyset$, which we have just shown.
Suppose that $r \geq 1$.  By Lemma~\ref{lemma:equalityinMtensorX} and an easy
induction on $r$, there exist $c \gobymod{p} a$, an element $z \in X(c)$ and
maps $g_i: c \go b_i$ such that $m_i g_i = p$ and $g_i z = y_i$ for all $i \in
\{1, \ldots, r\}$.  Then $\xi(x) = p \otimes z$, and for each $i$,
\[
\psi_{m_i} K_n(y_i)
=
\psi_{m_i} K_n(g_i z)
\supseteq
\psi_{m_i} (Jg_i) K_n(z)
=
\psi_{m_i g_i} K_n(z)
=
\psi_p K_n(z),
\]
using~\bref{part:approx-nat}.  Hence $\bigcap_{i = 1}^r \psi_{m_i} K_n(y_i)
\supseteq \psi_p K_n(z)$.  But $K_n(z) \neq \emptyset$ by inductive
hypothesis, so $\bigcap_{i = 1}^n \psi_{m_i} K_n(y_i) \neq \emptyset$,
proving~\bref{eq:nonempty-int}. 
\done
\end{proof}

\begin{prooflike}{Proof of Theorem~\ref{thm:preciserecognition}}
\paragraph*{\bref{item:pr-univ}$\implies$\bref{item:pr-met}}
Assume~\bref{item:pr-univ}.  By
Theorem~\ref{thm:existenceofuniversalsolution}, condition~\So\ holds, so $(J,
\gamma)$ is the universal solution $(I, \iota)$ constructed
in~\S\ref{sec:construction}.  Certainly $I$ is occupied and each space $I(a)$
is compact.  Lemmas~\ref{lemma:compactmetrizable} and~\ref{lemma:metriconIa}
then give metrics with the property required.

\paragraph*{\bref{item:pr-met}$\implies$\bref{item:pr-notmet}}
Trivial.

\paragraph*{\bref{item:pr-notmet}$\implies$\bref{item:pr-univ}}
Assume~\bref{item:pr-notmet}.  First we show that $(J, \gamma)$ is the
universal solution in $\Set$.  So, let $(X, \xi)$ be an $M$-coalgebra in
$\Set$; we construct a coalgebra map $(X, \xi) \go (J, \gamma)$ and prove
that it is the unique such.

For each $a \in \scat{A}$ and $x \in X(a)$ we have a sequence $(K_n(x))_{n \in
\nat}$ of subsets of $J(a)$, defined above.  By
Lemma~\ref{lemma:recognition-approx}(\ref{part:approx-nonempty}), $\bigcap_{n
\in \nat} K_n(x)$ has at least one element.  On the other hand, choose a
resolution $(a_\blb, m_\blb, x_\blb)$ of $x$.  Then for each $n\in\nat$,
writing $\psi = \gamma^{-1}$, we have
\[
K_n(x) 
\sub 
\psi_{m_1} \cdots \psi_{m_n} J(a_n)
=
V^J_{m_1, \ldots, m_n}.
\]
So by~(\ref{item:pr-notmet}), $\bigcap_{n\in\nat} K_n(x)$ has at most one
element.  Hence we may define, for each $a \in \scat{A}$, a function
$\ovln{\xi}_a: X(a) \go J(a)$ by $\{ \ovln{\xi}_a(x) \} = \bigcap_{n \in \nat}
K_n(x)$.

The family $(\ovln{\xi}_a)_{a \in \scat{A}}$ is a natural transformation $X
\go J$.  Indeed, let $f: a \go a'$ be a map in $\scat{A}$.  Then for all $n
\in \nat$,
\[
f \ovln{\xi}_a (x) 
\in
(Jf) K_n(x)
\sub
K_n (fx)
\]
by Lemma~\ref{lemma:recognition-approx}\bref{part:approx-nat}, 
so $f \ovln{\xi}_a (x) = \ovln{\xi}_{a'} (fx)$, as required.  

I claim that $\ovln{\xi}$ is a map $(X, \xi) \go (J, \gamma)$ of coalgebras in
$\Set$.  Let $a \in \scat{A}$ and $x \in X(a)$, and write 
\[
\xi_a(x) 
=
\Bigl( b \gobymod{m} a \Bigr) \otimes y.
\]
Then by the observation at~\bref{eq:coalg-map}, we have to show that
$\psi_m(\ovln{\xi}_b(y) \in K_n(x)$ for all $n \in \nat$.  When $n = 0$ this
is certainly true.  Now let $n \geq 1$; we have to show that for all $m'$ and
$y'$ such that
\[
\xi_a(x) 
= 
\Bigl(
b' \gobymod{m'} a
\Bigr)
\otimes y',
\]
we have $\psi_m \ovln{\xi}_b (y) \in \psi_{m'} K_{n - 1} (y')$.  Since $m
\otimes y = m' \otimes y'$, there exist a commutative square
\[
\begin{diagdiag}
	&		&c	&		&	\\
	&\ldTo<g	&	&\rdTo>{g'}	&	\\
b	&		&	&		&b'	\\
	&\rdMod<m	&	&\ldMod>{m'}	&	\\
	&		&a	&		&	\\
\end{diagdiag}
\]
and $z \in X(c)$ such that $gz = y$ and $g'z = y'$.  Hence
\begin{eqnarray*}
\psi_m \ovln{\xi}_b (y)	&
=	&
\psi_m \ovln{\xi}_b (gz)
=
\psi_m g \ovln{\xi}_c (z)
=	
\psi_{mg} \ovln{\xi}_c (z)	
=	
\psi_{m'g'} \ovln{\xi}_c (z)	\\
	&
=	&
\psi_{m'} \ovln{\xi}_{b'} (y')
\in	
\psi_{m'} K_{n - 1}(y')
\end{eqnarray*}
(the last equality by symmetry), as required.

For uniqueness, let $\twid{\xi}: (X, \xi) \go (J, \gamma)$ be a map of
coalgebras in $\Set$.  We show by induction on $n$ that $\twid{\xi}_a
(x) \in K_n(x)$ for all $a \in \scat{A}$ and $x \in X(a)$; the result follows.
For $n=0$ this is trivial.  Let $n \geq 1$, $a \in \scat{A}$, and $x \in
X(a)$.  If $\xi_a(x) = (b \gobymod{m} a) \otimes y$ then, as observed
at~(\ref{eq:coalg-map}), $\twid{\xi}$ being a map of coalgebras implies that
$\twid{\xi}_a (x) = \psi_m \twid{\xi}_b (y)$; so by inductive hypothesis,
$\twid{\xi}_a (x) \in \psi_m K_{n - 1}(y)$.  Hence $\twid{\xi}_a (x) \in
K_n(x)$, as required.

We have now shown that $(J, \gamma)$ is the terminal coalgebra in $\Set$---or
properly, with notation as in Proposition~\ref{propn:coalgebrasinTopandSet},
that $U_*(J, \gamma)$ is the terminal coalgebra in $\Set$.  By
Theorem~\ref{thm:existenceofuniversalsolution}, condition \So\ holds, so
$U_*(J, \gamma)$ is the universal solution $U_*(I, \iota)$ constructed
in~\S\ref{sec:construction}.  Also $(I, \iota)$ is the universal solution in
$\Top$, so there is a unique map $(J, \gamma) \go (I,
\iota)$ of coalgebras in $\Top$.  Each component $J(a) \go I(a)$ is a
continuous bijection from a compact space to a Hausdorff space, and is
therefore a homeomorphism.  So $(J, \gamma)$ is isomorphic to $(I, \iota)$ as
a coalgebra in $\Top$; hence it is the universal solution in $\Top$.  
\done
\end{prooflike}

In many examples the universal solution is especially easy to recognize.

\begin{cor}[Crude Recognition Theorem]
\label{cor:cruderecognition}
Let $(\scat{A}, M)$ be an equational system with $\scat{A}$ finite.  Let $(J,
\gamma = \psi^{-1})$ be a fixed point of $M$ in $\Top$ such that for each $a
\in \scat{A}$, the space $J(a)$ is nonempty and compact.  Suppose further that
the spaces $J(a)$ can be metrized in such a way that for each sector $b
\gobymod{m} a$, the induced map $J(b) \goby{\psi_m} J(a)$ is a contraction.
Then $(J, \gamma)$ is the universal solution of $(\scat{A}, M)$.
\end{cor}

\begin{proof}
Since $\scat{A}$ and $M$ are finite, there are only finitely many sectors $m$,
so we may choose $\lambda < 1$ such that each map $\psi_m$ is a contraction
with constant $\lambda$.  Since $\scat{A}$ is finite and each space $J(a)$ is
compact, we may also choose $D \geq 0$ such that $\diam(J(a)) \leq D$ for all
$a \in \scat{A}$.

We verify condition~\bref{item:pr-met} of the Precise Recognition Theorem.
Certainly $J$ is occupied.  For the main part of the condition, we have
$\diam(V^J_{m_1, \ldots, m_n}) \leq \lambda^n D$, and $\inf_{n \in \nat}
\lambda^n D = 0$. 
\done
\end{proof}


\section{Examples}
\label{sec:examples}

We illustrate the power of the Recognition Theorems by using them to produce
examples of universal solutions.  We can easily derive Freyd's theorem on the
interval, and we give similar characterizations of circles, cubes, simplices
and various fractal spaces.

\minihead{Discrete examples}

Even in the relatively trivial case of discrete equational systems, the
Recognition Theorems can be useful.

\begin{example*}{Cantor set}	\label{eg:rec-terminal}
Write $\One$ for the terminal category (one object and only the identity
arrow).  An equational system $(\One, M)$ amounts to a finite set $M$, and an
$M$-coalgebra is a space $X$ equipped with a map into the $M$-fold coproduct
$M \times X$.  The universal solution is the power $M^\posint$ (regarding the
set $M$ as a discrete space) together with the isomorphism $\gamma =
\psi^{-1}: M^\posint \goiso M \times M^\posint$.  This can be shown directly,
or from the description of the universal solution in~\S\ref{sec:dsss}, or
from a Recognition Theorem as follows.

The space $M^\posint$ is compact.  It is nonempty if $M$ is, so the coalgebra
$(M^\posint, \gamma)$ is occupied.  For $m \in M$, the map $\psi_m: M^\posint
\go M^\posint$ is given by
\[
\psi_m (m_1, m_2, \ldots)
=
(m, m_1, m_2, \ldots),
\]
so condition~(\ref{item:pr-notmet}) of the Precise Recognition Theorem holds.
Hence $(M^\posint, \gamma)$ is the universal solution.  When $M$ has
cardinality $2$, the universal solution is the standard Cantor set
$2^\posint$. 

In fact, the homeomorphism type of $M^\posint$ is independent of $M$ for $|M|
\geq 2$.  This classical fact can be proved as follows.  Let $k \geq 2$.
Write $k = \{ 0, \ldots, k-1 \}$, write $\psi: 2 \times 2^\posint \goiso
2^\posint$ for the usual isomorphism, and let 
$\psi^{(k)}: k \times 2^\posint \goiso 2^\posint$ be the composite 
\[
k \times 2^\posint              \rTo^{\psi + \id}_\sim 
(k - 1) \times 2^\posint        \rTo^{\psi + \id}_\sim
\                               \cdots 
\                               \rTo^{\psi + \id}_\sim
2 \times 2^\posint              \rTo^\psi_\sim
2^\posint.
\]
Then for each $m \in k$, the map $\psi^{(k)}_m: 2^\posint \go 2^\posint$ is of
the form $\psi_{p_1} \cdots \psi_{p_r}$ for some $r \geq 1$ and $p_1, \ldots,
p_r \in 2$.  Using the metric on $2^\posint$ induced by its embedding into
$[0, 1]$ (defined in~\S\ref{sec:dsss}), $\psi_0$ and $\psi_1$ are contractions
with constant $1/3$; hence each map $\psi^{(k)}_m$ is also a contraction with
constant (at most) $1/3$.  By the Crude Recognition Theorem, $( 2^\posint,
(\psi^{(k)})^{-1} )$ is the universal solution of $(\One, k)$.  In particular,
$2^\posint \iso k^\posint$ for all $k \geq 2$.
\end{example*}

\begin{example*}{Universal convergent sequence}	\label{eg:rec-univ-seq}
There is a discrete equational system defined informally by
\begin{eqnarray*}
X_1	&\iso	&X_1	\\
X_2	&\iso	&X_1 + X_2	
\end{eqnarray*}
(as in the Introduction).  Its universal solution is $X_1 = \{ 0 \}$ and $X_2
= \nat \cup \{ \infty \}$, with $X_2$ topologized as the Alexandroff one-point
compactification of the discrete space $\nat$.  This can be shown by an easy
application of the Crude Recognition Theorem, metrizing $\nat \cup \{\infty\}$
by using the evident homeomorphism with the subspace $\{2^{-n} \such n \in
\nat\} \cup \{0\}$ of $\reals$.
\end{example*}

A discrete equational system may contain equations of the form $X_i = X_i$, or
loops such as $X_1 = X_2, X_2 = X_3, X_3 = X_1$, or infinite chains such as
$X_1 = X_2, X_2 = X_3, \ldots$.  In those cases the universal solution $(I,
\iota)$ will involve the one-point space, and perhaps other spaces containing
isolated points (as in the last example).  But if the one-point space is not
involved then $I$ is extremely simple:

\begin{propn}[Empty or Cantor]
\label{propn:emptyorcantor}
Let $(\scat{A}, M)$ be a discrete equational system with universal solution
$(I, \iota)$.  Suppose that $| I(a) | \neq 1$ for all $a \in \scat{A}$.  Then
each space $I(a)$ is either empty or the Cantor set.
\end{propn}

This is closely related to the classical fact that, up to homeomorphism, the
empty set and the Cantor set are the only totally disconnected compact
metrizable spaces with no isolated points~\cite{HY}.  In fact, our proposition
together with the discrete realizability
theorem~(\ref{thm:discretelyselfsimilarspaces}) leads to a new proof of this
fact~\cite{SS2}; see also Appendix~\ref{app:admitting}. 

\begin{proof}
Define $J: \scat{A} \go \Top$ by 
\[
J(a) =
\left\{
\begin{array}{ll}
\emptyset	&\textrm{if } I(a) = \emptyset	\\
2^\posint	&\textrm{if } I(a) \neq \emptyset.
\end{array}
\right.
\]
For each $a \in \scat{A}$, let $k(a) = \left| \sum_{b: I(b) \neq \emptyset}
M(b, a) \right| \in \nat$ and choose an isomorphism between the sets $\sum_{b:
I(b) \neq \emptyset} M(b, a)$ and $k(a)$.  (We continue to write $n$ for the
$n$-element set $\{0, \ldots, n - 1\}$.)  Then for all $a \in \scat{A}$,
\[
(M \otimes J)(a)
=
\sum_{b \in \scat{A}} M(b, a) \times J(b)
\iso
\sum_{b: I(b) \neq \emptyset} M(b, a) \times 2^\posint
\iso
k(a) \times 2^\posint.
\]
Define an isomorphism $\gamma_a: J(a) \goiso (M \otimes J)(a)$ for each $a \in
\scat{A}$ as follows.  If $I(a) = \emptyset$ then $J(a) = \emptyset = (M
\otimes J)(a)$, and we put $\gamma_a = 1_\emptyset$.  If $I(a) \neq \emptyset$
then $J(a) = 2^\posint$ and we put
\[
\gamma_a
=
\left(
2^\posint
\rTo^{ ( \psi^{ (k(a)) } )^{-1} }_\diso
k(a) \times 2^\posint
\iso
(M \otimes J)(a)
\right)
\]
where for $k \geq 2$, the homeomorphism 
\[
\psi^{(k)}: k \times 2^\posint \goiso 2^\posint
\]
is defined as in Example~\ref{eg:rec-terminal}, and $\psi^{(1)} = \id$.  We
show that this fixed point $(J, \gamma)$ satisfies
condition~(\ref{item:pr-notmet}) of the Precise Recognition Theorem.

Certainly $J$ is occupied, and each space $J(a)$ is compact and can be
equipped with the usual metric.  Now take a complex $(a_\blb, m_\blb)$.  Write
$\psi = \gamma^{-1}$.  For each $r \in \nat$, either $k(a_r) = 1$, in which
case $\psi_{m_{r + 1}}$ is an isometry, or $k(a_r) \geq 2$, in which case
$\psi_{m_{r + 1}}$ is a contraction with constant $1/3$.  The latter
case arises infinitely often: for if not, there is some $s \in \nat$ for which
$1 = k(a_s) = k(a_{s + 1}) = \cdots$, and then $|I(a_s)| = 1$, contrary to
hypothesis.  Moreover, $\diam(J(a_r)) \leq 1$ for each $r$.  So
\[
\diam(V^J_{m_1, \ldots, m_r})
=
\diam(\psi_{m_1} \ldots \psi_{m_r} J(a_r))
\to 0
\textrm{ as } r \to \infty
\]
and therefore condition~(\ref{item:pr-notmet}) of the Precise Recognition Theorem
is satisfied, as required. 
\done
\end{proof}

\begin{example*}{Walks} \label{eg:walks-newrule}
Consider again the example from~\S\ref{sec:dsss} of spaces of walks, but
suppose now that we change the rule at $0$ to read `if at position $0$, step
right'.  Thus, the first equation of the system changes from `$W_0 = W_0$' to
`$W_0 = W_1$'.  Each of the spaces $W_n$ making up the universal solution is
now infinite, and in particular $|W_n| > 1$ for all $n \in \nat$.  So by
Proposition~\ref{propn:emptyorcantor}, $W_n$ is homeomorphic to the Cantor set
for all $n \in \nat$.

Contrast the universal solution $(W_n)$ of the original set of rules; there,
$|W_0| = 1$, and since it is possible to walk to $0$ from any position $n$,
each of the spaces $W_n$ has at least one isolated point.
\end{example*}

\minihead{Non-discrete examples}

\begin{example*}{Interval}	\label{eg:rec-Freyd}
We finally prove the topological Freyd theorem~(\ref{thm:topologicalFreyd}).
So far we have verified that the $(\scat{A}, M)$ concerned is an equational
system, and exhibited an $M$-coalgebra $(J, \gamma)$ (previously written $(I,
\iota)$) with $J(0) = \{\star\}$ and $J(1) = [0, 1]$.  We apply the Crude
Recognition Theorem~(\ref{cor:cruderecognition}).  Both spaces $J(a)$ are
nonempty, compact, and can be metrized in the usual way.  Evidently $\gamma$
is invertible, so we have a fixed point $(J, \gamma = \psi^{-1})$.  For a
sector $m: b \gomod a$ in $(\scat{A}, M)$, the induced map $\psi_m: J(b) \go
J(a)$ is either constant or, in the case that $m$ is one of two sectors $1
\gomod 1$, it is one of the two maps
\[
\begin{array}{ccc}
[0, 1]	&\go		&[0, 1]		\\
t	&\goesto	&t/2		\\
t	&\goesto	&(t + 1)/2.
\end{array}
\]
All of the maps $\psi_m$ are therefore contractions.  Hence $(J, \gamma)$ is
the universal solution. 

Freyd's Theorem expresses $[0, 1]$ as two copies of itself glued end to end.
Two can be replaced by any larger number.  Thus, for each $k \geq 2$ there is
a corresponding equational system $(\scat{A}, M^{(k)})$, with $\scat{A}$ as
above and, for instance, $|M^{(k)} (1, 1)| = k$.  The multiplication map $k
\cdot \dashbk: [0, 1] \go [0, k]$ puts an $M^{(k)}$-coalgebra structure
$\gamma^{(k)}$ on the functor $J$, and the same argument shows that $(J,
\gamma^{(k)})$ is the universal solution.  So the interval, like the Cantor
set (Example~\ref{eg:rec-terminal}), is recursively realizable in infinitely
many ways.
\end{example*}

\begin{example*}{Circle}
The recursive description of the interval can easily be extended to give a
recursive description of the circle $S^1$.  The circle is the coequalizer of
the diagram 
\[
\{\star\} \parpair{0}{1} [0, 1],
\]
and the crucial observation is that all of these spaces and maps appear in the
universal solution of the Freyd system. 

Let $\scat{A}$ be the following category with $3$ objects and $2$ non-identity
arrows:
\[
\begin{diagram}[size=4em,tight]
0       &\pile{\rTo^\sigma \\ \rTo_\tau}        &1      &
\       &2.     \\
\end{diagram}
\]
The idea is to extend the Freyd system to an equational system on $\scat{A}$,
in such a way that for any $X \in \ndTop{\scat{A}}$, the space $(M \otimes
X)(2)$ is the coequalizer of $X\sigma, X\tau: X(0) \parpairu X(1)$.  So,
define a module $M: \scat{A} \gomod \scat{A}$ as follows.  The restriction of
$M$ to the full subcategory $\{ 0, 1 \}$ of $\scat{A}$ is the module of the
Freyd system.  For all $a \in \{0, 1, 2\}$, $M(2, a) =
\emptyset$.  Finally, $|M(0, 2)| = |M(1, 2)| = 1$.  

There is a nondegenerate functor $J: \scat{A} \go \Top$ given by
\[
\begin{diagram}[size=4em,tight]
\{\star\} &\pile{\rTo^0\\ \rTo_1}       &[0, 1]         &
\       &S^1.   \\
\end{diagram}
\]
The functor $M \otimes J$ can naturally be identified with
\[
\begin{diagram}[size=4em,tight]
\{\star\} &\pile{\rTo^0\\ \rTo_2}       &[0, 2]         &
\       &[0, 1]/(0 = 1)   \\
\end{diagram}
\]
(the rightmost object being $[0, 1]$ with its endpoints identified), so there
is an evident isomorphism $\gamma: J \goiso M \otimes J$.  

We show that $(J, \gamma)$ is the universal solution of $(\scat{A}, M)$ using
the Crude Recognition Theorem.  Each of the spaces $J(a)$ is compact and
nonempty.  Write $\psi = \gamma^{-1}$.  We have to check that the spaces
$J(a)$ can be metrized in such a way that for each sector $m: b \gomod a$, the
map $\psi_m: J(b) \go J(a)$ is a contraction.  For the sectors $m$ in the
Freyd system, we have already shown this in Example~\ref{eg:rec-Freyd}.  For
the sector $0 \gomod 2$, it is trivial.  The only remaining sector is $1
\gomod 2$, whose induced map is the quotient map $[0, 1] \go S^1$, and this is
a contraction if a suitably scaled-down metric on $S^1$ is chosen.  So the
Crude Recognition Theorem applies, as claimed.
\end{example*}

\minihead{Products}

Given recursive realizations of spaces $S$ and $S'$, there arises, in a
canonical way, a recursive realization of the product space $S \times S'$.
This follows from Proposition~\ref{propn:univ-soln-product} below.  We use the
fact that the category of equational systems has finite
products~(\S\ref{sec:sss}). 

\begin{lemma}
\label{lemma:productandtensor}
Let $\scat{B}$ and $\scat{B}'$ be small categories, let $Y: \scat{B}^\op \go
\Set$ and $Y': \scat{B}'^\op \go \Set$ be functors, and let $X: \scat{B} \go
\Top$ and $X': \scat{B}' \go \Top$ be functors taking values in compact
Hausdorff spaces.  Then 
\[
(Y \times Y') \otimes (X \times X')
\iso
(Y \otimes X) \times (Y' \otimes X')
\]
where on the left-hand side, `$\times$' is used in the sense of
Lemma~\ref{lemma:functors-on-products}. 
\end{lemma}

\begin{proof}
We use the fact that if $K$ is a compact Hausdorff space then $K \times
\dashbk: \Top \go \Top$ preserves colimits.  We also use a formula from
Appendix~\ref{app:modules}: 
\[
Y \otimes X 
=
\Colt{(b, y) \in \elt{Y}} X(b).
\]
Now
\begin{eqnarray*}
(Y \times Y') \otimes (X \times X')     &\iso   &
\Colt{((b, b'), (y, y')) \in \elt{Y \times Y'}} (X \times X')(b, b')    \\
        &\iso   &
\Colt{(b, y) \in \elt{Y}} \Colt{(b', y') \in \elt{Y'}} X(b) \times X'(b')\\
        &\iso   &
\Colt{(b, y) \in \elt{Y}} X(b) \times \Colt{(b', y') \in \elt{Y'}} X'(b')\\
        &\iso   &
(Y \otimes X) \times (Y' \otimes X').
\end{eqnarray*}
\done
\end{proof}

\begin{propn}[Universal solution of product]
\label{propn:univ-soln-product}
Let $(\scat{A}, M)$ and $(\scat{A}', M')$ be equational systems with universal
solutions $(I, \iota)$ and $(I', \iota')$, respectively, in $\Top$.  Then the
product $(\scat{A}, M) \times (\scat{A}', M') = (\scat{A} \times \scat{A}', M
\times M')$ in the category of equational systems has universal solution $(I
\times I', \iota \times \iota')$ in $\Top$.
\end{propn}

\begin{proof}
The functor $I \times I': \scat{A} \times \scat{A}' \go \Top$ is nondegenerate:
for
\[
U \of (I \times I')
\iso
(U \of I) \times (U \of I'):
\scat{A} \times \scat{A}'
\go 
\Set
\]
is nondegenerate by Lemma~\ref{lemma:functors-on-products}\bref{item:prod-nd},
and $I(a) \times I'(a')$ is compact Hausdorff for all $a \in \scat{A}$, $a'
\in \scat{A}'$.  By Lemma~\ref{lemma:productandtensor}, we have a natural
isomorphism 
\[
\iota \times \iota':
I \times I'
\goiso
(M \otimes I) \times (M' \otimes I')
\iso
(M \times M') \otimes (I \times I').
\]
Also, $I \times I'$ is occupied since $I$ and $I'$ are.  To finish the proof
it remains only to verify that $(I \times I', \iota \times \iota')$ satisfies
the main condition in~\bref{item:pr-notmet} of the Precise Recognition
Theorem, and this follows from the fact that it is satisfied by $(I, \iota)$
and $(I', \iota')$.
\done
\end{proof}

\begin{example*}{Cubes}	\label{eg:rec-cubes}
Let $(\scat{A}, M)$ be the Freyd system.  Then by
Proposition~\ref{propn:univ-soln-product}, $(\scat{A}^2, M^2)$ has a universal
solution $(I, \iota)$ satisfying $I(1, 1) = [0, 1]^2$.  Informally, the
self-similarity equations are
\[
\begin{array}{rclcrcl}
\zmark		&=	&\zmark	&
\diagspace	&
\begin{array}{c}
\setlength{\unitlength}{1em}
\begin{picture}(3,0.5)(-1.5,-0.25)
\put(-1.3,0){\line(1,0){2.6}}
\cell{-1.3}{0}{c}{\zmark}
\cell{1.3}{0}{c}{\zmark}
\end{picture}
\end{array}
&
=	&
\begin{array}{c}
\setlength{\unitlength}{1em}
\begin{picture}(5.6,0.5)(-2.8,-0.25)
\put(0,0){\line(1,0){2.6}}
\put(0,0){\line(-1,0){2.6}}
\cell{-2.6}{0}{c}{\zmark}
\cell{0}{0}{c}{\zmark}
\cell{2.6}{0}{c}{\zmark}
\end{picture}
\end{array}
\\
&&&&&&\\
\begin{array}{c}
\setlength{\unitlength}{1em}
\begin{picture}(0.5,3)(-0.25,-1.5)
\put(0,-1.3){\line(0,1){2.6}}
\cell{0}{-1.3}{c}{\zmark}
\cell{0}{1.3}{c}{\zmark}
\end{picture}
\end{array}
&
=	&
\begin{array}{c}
\setlength{\unitlength}{1em}
\begin{picture}(0.5,5.6)(-0.25,-2.8)
\put(0,0){\line(0,1){2.6}}
\put(0,0){\line(0,-1){2.6}}
\cell{0}{-2.6}{c}{\zmark}
\cell{0}{0}{c}{\zmark}
\cell{0}{2.6}{c}{\zmark}
\end{picture}
\end{array}
&
\diagspace	&
\begin{array}{c}
\setlength{\unitlength}{1em}
\begin{picture}(3,3)(-1.5,-1.5)
\put(-1.3,-1.3){\textcolor{grey}{\rule{2.6\unitlength}{2.6\unitlength}}}
\put(-1.3,-1.3){\line(1,0){2.6}}
\put(-1.3,1.3){\line(1,0){2.6}}
\put(-1.3,-1.3){\line(0,1){2.6}}
\put(1.3,-1.3){\line(0,1){2.6}}
\cell{-1.3}{-1.3}{c}{\zmark}
\cell{1.3}{-1.3}{c}{\zmark}
\cell{-1.3}{1.3}{c}{\zmark}
\cell{1.3}{1.3}{c}{\zmark}
\end{picture}
\end{array}
&
=	&
\begin{array}{c}
\setlength{\unitlength}{1em}
\begin{picture}(5.6,5.6)(-2.8,-2.8)
\put(-2.6,-2.6){\textcolor{grey}{\rule{5.2\unitlength}{5.2\unitlength}}}
\put(-2.6,-2.6){\line(0,1){5.2}}
\put(0,-2.6){\line(0,1){5.2}}
\put(2.6,-2.6){\line(0,1){5.2}}
\put(-2.6,-2.6){\line(1,0){5.2}}
\put(-2.6,0){\line(1,0){5.2}}
\put(-2.6,2.6){\line(1,0){5.2}}
\cell{-2.6}{-2.6}{c}{\zmark}
\cell{0}{-2.6}{c}{\zmark}
\cell{2.6}{-2.6}{c}{\zmark}
\cell{-2.6}{0}{c}{\zmark}
\cell{0}{0}{c}{\zmark}
\cell{2.6}{0}{c}{\zmark}
\cell{-2.6}{2.6}{c}{\zmark}
\cell{0}{2.6}{c}{\zmark}
\cell{2.6}{2.6}{c}{\zmark}
\end{picture}
\end{array}
.
\end{array}
\]
A similar statement holds in higher dimensions.
\end{example*}

\minihead{Further non-discrete examples}

An \demph{iterated function system} on $\reals^d$ is a family $\psi_0, \ldots,
\psi_n$ ($n\geq 0$) of contractions $\reals^d \go \reals^d$.  By a theorem of
Hutchinson~\cite{Hut}, there is a unique nonempty compact subset $S$ of
$\reals^d$ satisfying $S = \bigcup_{i = 0}^n \psi_i S$, the \demph{attractor}
of the system.  Various familiar self-similar spaces arise in this way.

\begin{example*}{Sierpi\'nski simplices}	\label{eg:rec-Sierpinski}
Let $n\in\nat$ and let $s_0, \ldots, s_n$ be affinely independent points of
$\reals^n$.  For each $i \in \{ 0, \ldots, n \}$, write $\psi_i: \reals^n \go
\reals^n$ for the scaling with scale factor $1/2$ and fixed point $s_i$.
The \demph{Sierpi\'nski simplex} with vertices $s_0, \ldots, s_n$ is the
attractor of the iterated function system $(\psi_0, \ldots, \psi_n)$.  When $n
= 1$, it is the closed interval with endpoints $s_0$ and $s_1$.  When $n = 2$,
it is the usual Sierpi\'nski triangle or gasket $S$, which satisfies an
isomorphism expressed informally as
\[
\begin{array}{c}
\includegraphics[width=13.6em]{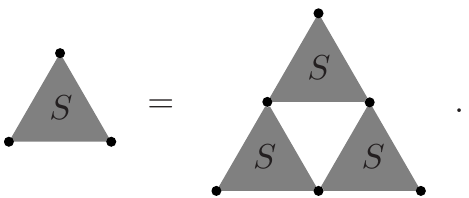}
\end{array}
\]

Now take any $n \in \nat$ and $s_0, \ldots, s_n$ as above, and write $S$ for
the resulting Sierpi\'nski simplex.  We construct an equational system
whose universal solution is $S$ (equipped with some extra structure).  

Let $\scat{A}$ be the category with objects $0$ and $1$ and
non-identity arrows $\sigma_0, \ldots, \sigma_n: 0 \go 1$.  Define an
equivalence relation $\sim$ on $\{ 0, \ldots, n \}^2$ by $(i, j) \sim (i',
j')$ if and only if $\{ i, j \} = \{ i', j' \}$, and write $[i, j]$ for the
equivalence class of $(i, j)$.  Define $M: \scat{A} \gomod \scat{A}$ by
\begin{center}
\begin{tabular}[c]{c|l}
        &
\begin{diagram}[width=4em,height=3em,tight]
M(\dashbk, 0)							&
\pile{\rTo^{\sigma_0 \cdot \dashbk}\\ 
	\avdots\\ 
	\rTo_{\sigma_n \cdot \dashbk}}				&
M(\dashbk, 1)							\\
\end{diagram}
        \\[4ex]
\hline
\raisebox{-4ex}{%
\begin{diagram}[width=3em,height=3em,tight]
M(0, \dashbk)   \\
\uTo<{\dashbk \cdot \sigma_0} 
\cdots 
\uTo>{\dashbk \cdot \sigma_n}   \\
M(1, \dashbk)   \\
\end{diagram}}
        &
\hspace*{1em}\raisebox{-4ex}{%
\begin{diagram}[width=4em,height=3em,tight]
\{ \id \}							&
\pile{\rTo^{[0, 0]}\\ \avdots\\ \rTo_{[n, n]}}			&
\{ 0, \ldots, n \}^2 / \sim					\\
\uTo \cdots \uTo						&
								&
\uTo<{[\dashbk, 0]} \cdots \uTo>{[\dashbk, n]}			\\
\emptyset							&
\pile{\rTo\\ \avdots\\ \rTo}					&
\{ 0, \ldots, n \}.						\\
\end{diagram}}
\end{tabular}
\end{center}
Then $(\scat{A}, M)$ is an equational system.  

Any space $X_1$ equipped with distinct basepoints $x_0, \ldots, x_n$
determines a nondegenerate functor $X: \scat{A} \go \Top$, with $X(0) =
\{\star\}$ and $X(1) = X_1$.  Then $M \otimes X$ is the functor determined by
the quotient space
\[
\frac{ 
\{ 0, \ldots, n \} \times X_1
}{
(i, x_j) = (j, x_i) \textrm{ for all } i, j
}
\]
with basepoints $(0, x_0), \ldots, (n, x_n)$.  

In particular, $(S, s_0, \ldots, s_n)$ determines a nondegenerate functor $J:
\scat{A} \go \Top$.  The function
\begin{equation}	\label{eq:Sierpinski-map}
\begin{array}{ccc}
\{ 0, \ldots, n \} \times S	&\go		&S		\\
(i, s)				&\goesto	&\psi_i(s)
\end{array}
\end{equation}
induces a map $(M \otimes J)(1) \go S = J(1)$, since $\psi_i(s_j) = \half(s_i
+ s_j) = \psi_j(s_i)$ for all $i, j$.  This map is surjective by definition of
$S$.  It is injective since each map $\psi_i$ is injective and 
\[
\psi_i S \cap \psi_j S
=
\{ \psi_i(s_j) \} 
=
\{ \psi_j(s_i) \}
\]
whenever $i \neq j$.  It also preserves basepoints.  So we have an isomorphism
$\psi: M \otimes J \go J$.  The spaces $J(0) = \{\star\}$ and $J(1) = S$ are
nonempty and compact, and the structure maps $\psi_i: J(1) \go J(1)$ are
contractions, so by the Crude Recognition Theorem, $(J, \psi^{-1})$ is the
universal solution of $(\scat{A}, M)$.

We have therefore realized the $n$-dimensional Sierpi\'nski simplex as the
solution of an equational system, in a way that formalizes the idea that it is
homeomorphic to a gluing of $(n + 1)$ half-sized copies of itself.
\end{example*}

\begin{example*}{Iterated function systems}	\label{eg:rec-ifs}
More generally, let $(\psi_0, \ldots, \psi_n)$ be an iterated function system
on $\reals^d$ (some $d\in\nat$).  Write $S$ for its attractor, and $s_i$ for
the fixed point of $\psi_i$.  Suppose that $\psi_0, \ldots, \psi_n$ are
injective, that $s_0, \ldots, s_n$ are distinct, and that if $\psi_i(s) =
\psi_j(t)$ with $i \neq j$ and $s, t \in S$ then $s, t \in \{ s_0, \ldots, s_n
\}$.  Then the space $S$ can be realized by a finite equational system, as
follows.

Define an equivalence relation $\sim$ on $\{ 0, \ldots, n \}^2$ by $(i, j)
\sim (i', j') \iff \psi_i(s_j) = \psi_{i'}(s_{j'})$, and write $[i, j]$ for
the equivalence class of $(i, j)$.  Proceeding exactly as in the previous
example, this equivalence relation gives rise to an equational system
$(\scat{A}, M)$; the space $S$ with basepoints $s_0, \ldots, s_n$ determines a
nondegenerate functor $J: \scat{A} \go \Top$; the maps $\psi_i$ determine an
isomorphism $M \otimes J \goiso J$; and by the Crude Recognition Theorem, this
is the universal solution of $(\scat{A}, M)$.

Even for iterated function systems not within the scope of this example, the
attractor may still have a straightforward description as a universal
solution: $[0, 1]^n$ in Example~\ref{eg:rec-cubes}, for instance.
\end{example*}

\begin{example*}{Barycentric subdivision}	\label{eg:rec-bary}
Barycentric subdivision expresses the $n$-simplex $\Delta^n$ as $(n + 1)!$
smaller copies of itself glued together along simplices of lower
dimension.  This self-similarity can be formalized as follows.

Let $\Dface$ be the category whose objects are the nonempty finite totally
ordered sets $\upr{n} = \{ 0, \ldots, n \}$ ($n \in \nat$) and whose maps are
the order-preserving injections.  For each $n, m \in \nat$, put
\[
M( \upr{n}, \upr{m} )
=
\{
\textrm{chains }
\emptyset \propersub Q(0) \propersub \cdots \propersub Q(n) \sub \upr{m}
\}
\]
where $\propersub$ means proper subset.  (This can be regarded as the set of
$n$-simplices occurring in the barycentric subdivision of $\Delta^m$.  It is
empty unless $n \leq m$.)  The idea can be seen in Figure~\ref{fig:bary}:
\begin{figure}
\centering
\includegraphics[width=9.3em]{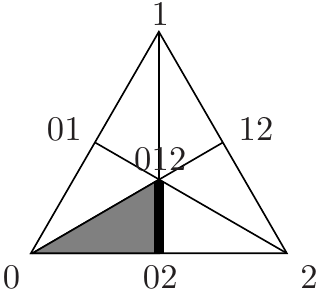}
\caption{Barycentric subdivision of $\Delta^2$}
\label{fig:bary}
\end{figure}
the $1$-simplex in bold and the shaded 2-simplex correspond respectively to
\begin{eqnarray*}
\left(
\emptyset \propersub \{ 0, 2 \} \propersub \{ 0, 1, 2 \}
\right)	&
\in	&
M(\upr{1}, \upr{2}),	\\
\left(
\emptyset \propersub \{ 0 \} \propersub \{ 0, 2 \} \propersub \{ 0, 1, 2 \}
\right)	&
\in	&
M(\upr{2}, \upr{2}).
\end{eqnarray*}
An element of $M(\upr{n}, \upr{m})$ can be regarded as an order-preserving
injection $\upr{n} \go \nepower \upr{m}$, where $\nepower$ denotes the set of
nonempty subsets ordered by inclusion.  By using direct images, $\nepower
\upr{m}$ is functorial in $\upr{m}$, so $M$ defines a module $\Dface \gomod
\Dface$.  It can be checked that $M$ is nondegenerate using the explicit
conditions~\cstyle{ND1} and~\cstyle{ND2} (\S\ref{sec:nondegen}).  And clearly
$M$ is finite, so $(\Dface, M)$ is an equational system.

We will show that the universal solution is given by the standard topological
simplex functor $\Delta^\blob: \Dface \go \Top$.  For each $n\in\nat$, fix an
affinely independent sequence $\vtr{e}^n_0, \ldots, \vtr{e}^n_n$ of points in
$\reals^n$, and let $\Delta^n$ be their convex hull.  Then for each map $f:
\upr{n} \go \upr{m}$ in $\Dface$ there is a unique affine map $\reals^n \go
\reals^m$ sending $\vtr{e}^n_j$ to $\vtr{e}^m_{f(j)}$ for each $j$, which
restricts to a map $\Delta f = f_*: \Delta^n \go \Delta^m$.  It is
straightforward to check that $U \of \Delta^\blob: \Dface \go \Set$ is
nondegenerate, again using conditions~\cstyle{ND1} and~\cstyle{ND2}.  (Roughly
speaking, this expresses the fact that the intersection of two faces of a
simplex, if not empty, is again a face.)  Moreover, each space $\Delta^n$ is
compact Hausdorff, so $\Delta^\blob$ is nondegenerate.

We construct an isomorphism $M \otimes \Delta^\blb \iso \Delta^\blb$.
(This expresses the fact that we really do have a subdivision.)  By the
universal property of tensor product (Appendix~\ref{app:modules}), a natural
transformation $\psi: M \otimes \Delta^\blb \go \Delta^\blb$ amounts to a
choice, for each sector $Q: \upr{n} \gomod \upr{m}$, of a map $\psi_Q:
\Delta^n \go \Delta^m$, satisfying the naturality condition $\psi_{fQg} = f_*
\of \psi_Q \of g_*$ for all $f$, $Q$ and $g$.  Indeed, given such a $Q$, there
is a unique affine map $\reals^n \go \reals^m$ such that 
\[
\vtr{e}^n_j
\goesto
\frac{1}{|Q(j)|}
\sum_{i \in Q(j)} \vtr{e}_i^m
\]
for all $j \in \upr{n}$, and this restricts to a map $\psi_Q: \Delta^n \go
\Delta^m$.  The naturality condition is easily verified.

This natural tranformation $\psi: M \otimes \Delta^\blb \go \Delta^\blb$ is
indeed an isomorphism.  To prove this, it suffices to show that for each $m
\in \nat$, the continuous map
\begin{equation}
\label{eq:bary-psi}
\psi:
M(\dashbk, \upr{m}) \otimes \Delta^\blb
\go 
\Delta^m
\end{equation}
is a homeomorphism.  Its domain is compact and its codomain Hausdorff, so in
fact it suffices to show that it is a bijection.  The inverse is constructed
as follows.  Let $\vtr{s} \in \Delta^m$; then $\vtr{s} = \sum_{i = 0}^m s_i
\vtr{e}^m_i$ with $s_i \geq 0$ and $\sum s_i = 1$.  There are unique $n \in
\nat$ and $s'_0 > \cdots > s'_n > s'_{n+1} = 0$ such that
\[
\{ s'_0, \ldots, s'_n, s'_{n+1} \}
=
\{ s_0, \ldots, s_m, 0 \},
\]
and we may define $q: \upr{m} \go \upr{n+1}$ by $s_i = s'_{q(i)}$.  For $j \in
\upr{n}$, put
\[
Q(j)		=	q^{-1} \{ 0, \ldots, j \},		
\diagspace
t_j		=	(s'_j - s'_{j + 1}) |Q(j)|,	
\diagspace
\vtr{t}		=	\sum_{j = 0}^n t_j \vtr{e}^n_j.
\]
A series of straightforward checks shows that $Q \in M(\upr{n}, \upr{m})$,
$\vtr{t} \in \Delta^n$, and the inverse to~(\ref{eq:bary-psi}) is given by $s
\goesto Q \otimes t$.

We now verify condition~(\ref{item:pr-met}) of the Precise Recognition
Theorem.  A standard calculation~\cite[2.21]{Hat} shows that for any
$Q: \upr{n} \gomod \upr{m}$,
\[
\diam (\psi_Q \Delta^n)
\leq
\frac{m}{m+1} \diam (\Delta^m)
\]
in the Euclidean metric.  More generally, if
\[
\upr{n_r} \gobymod{Q_r} \cdots \gobymod{Q_1} \upr{n_0}
\]
then the same method shows that
\begin{eqnarray*}
\diam (V^J_{Q_1, \ldots, Q_r})  &
=       &
\diam (\psi_{Q_1} \cdots \psi_{Q_r} \Delta^{n_r})	\\
        &
\leq	&
\left( \frac{n_{r-1}}{n_{r-1} + 1} \right)
\cdots
\left( \frac{n_0}{n_0 + 1} \right)
\diam (\Delta^{n_0})	\\
	&
\leq	&
\left( \frac{n_0}{n_0 + 1} \right)^r
\diam (\Delta^{n_0}).
\end{eqnarray*}
Condition~(\ref{item:pr-met}) follows.  

Hence the topological simplex functor $\Delta^\blb$ is the universal solution
to the equational system embodying the combinatorial process of barycentric
subdivision.
\end{example*}

Examples~\ref{eg:rec-terminal} (Cantor set) and~\ref{eg:rec-Freyd} (interval)
showed that the same space may have multiple different recursive realizations.
It may be thought that the different realizations in those examples were not
dramatically different.  But the previous example and the next illustrate a
greater contrast.

\begin{example*}{Edgewise subdivision}	\label{eg:rec-edgewise}
The topological simplex functor $\Delta^\blb: \Dface \go \Top$ can also be
characterized by edgewise subdivision.  This subdivision
(Figure~\ref{fig:edgewise} and~\cite{Freu})
\begin{figure}
\centering
\includegraphics[width=25.2em]{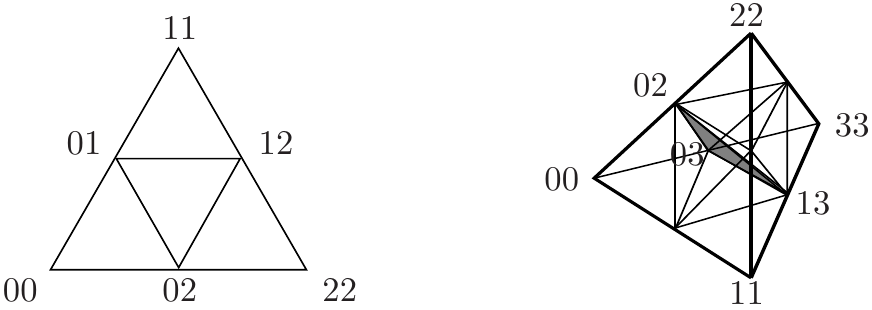}
\caption{Edgewise subdivisions of $\Delta^2$ and $\Delta^3$}
\label{fig:edgewise}
\end{figure}
expresses $\Delta^n$ as $2^n$ smaller copies of itself glued together.  It can
be viewed as an equational system $(\Dface, M)$ where 
\begin{eqnarray*}
M(\upr{n}, \upr{m})	&
=	&
\{
\textrm{order-preserving injections }
(p, q): \upr{n} \go \upr{m} \times \upr{m}	\\
	&	&
\ \,\textrm{such that }
p(n) \leq q(0)
\}
\end{eqnarray*}
and $\upr{m} \times \upr{m}$ is the product in the category of posets.
(Again, this set indexes the $n$-simplices occurring in the subdivision of
$\Delta^m$, and again, it is empty unless $n \leq m$.)  For instance, the
shaded 2-simplex inside the 3-simplex in Figure~\ref{fig:edgewise} is the
sector $\upr{2} \gomod \upr{3}$ given by the order-preserving
injection $\upr{2} \go \upr{3} \times \upr{3}$ with image $\{ (0, 2), (0, 3),
(1, 3) \}$.  Again it can be shown that $\Delta^\blob: \Dface \go \Top$, with
a canonical $M$-coalgebra structure, is the universal solution.
\end{example*}


\appendix

\section{Appendix: Modules}
\label{app:modules}

Here we state some basic features of the theory of modules over categories,
continuing the remarks at the end of the Introduction.

Much of this theory can be understood by analogy with the theory of modules
and bimodules in the ordinary sense of algebra.  It was already noted in the
Introduction that when $\scat{A}$ and $\scat{B}$ are monoids, seen as
one-object categories, a module $\scat{B} \gomod \scat{A}$ is a set with
compatible left $\scat{A}$- and right $\scat{B}$-actions.  If we work with
categories enriched in abelian groups, then a one-object category is exactly a
ring and a module $\scat{B} \gomod \scat{A}$ between rings $\scat{A}$ and
$\scat{B}$ is exactly an $(\scat{A}, \scat{B})$-bimodule.  In fact, the theory
of categorical modules can be developed in the generality of enriched
categories, and this general theory contains many parts of the theory of
algebraic (bi)modules.  For example, there are notions of tensor product and
flatness of categorical modules, generalizing the notions from algebra.

Indeed, given rings $A$, $B$ and $C$, an $(A, B)$-bimodule $M$, and a $(B,
C)$-bimodule $N$, there arises an $(A, C)$-bimodule $M \otimes_B N$.  There
is a similar tensor product of categorical modules:
\[
\scat{C} \gobymod{N} \scat{B} \gobymod{M} \scat{A}
\diagspace
\textrm{gives rise to}
\diagspace
\scat{C} \gobymod{M \otimes N} \scat{A}.
\]
Here $M \otimes N$ is defined by the coend formula
\[
(M \otimes N) (c, a)
=
\int^b M(b, a) \times N(c, b).
\]
Coends are explained in~\cite[Ch.~IX]{CWM}; concretely,
\[
(M \otimes N) (c, a)
=
\left(
\sum_{b \in \scat{B}} M(b, a) \times N(c, b)
\right)
/\sim
\]
where $\sim$ is the equivalence relation generated by $(mg, n) \sim (m, gn)$
for all $m \in M(b, a)$, $g \in \scat{B}(b', b)$ and $n \in N(c, b')$.  The
element of $(M \otimes N) (c, a)$ represented by $(m, n) \in M(b, a) \times
N(c, b)$ is written $m \otimes n$.  The tensor product of modules is
associative and unital up to coherent isomorphism.  (More precisely,
categories, modules, and their maps form a bicategory: \cite[7.8.2]{Bor1}.)

In the special case where $\scat{C}$ is the terminal category $\One$, the
tensor product construction gives for each module $M: \scat{B} \gomod
\scat{A}$ and functor $X: \scat{B} \go \Set$ a new functor $M \otimes X:
\scat{A} \go \Set$.  Concretely,
\[
(M \otimes X)(a)
=
\int^b M(b, a) \times X(b)
=
\left( 
\sum_{b \in \scat{B}} M(b, a) \times X(b) 
\right) 
/\sim
\]
where $a \in \scat{A}$ and $\sim$ is as above.  An equivalent formulation uses
the notion of category of elements (defined after Example~\ref{eg:Julia}):
\[
(M \otimes X)(a)
=
\Colt{(b, m) \in \elt{M(\dashbk, a)}} X(b),
\]
where the right-hand side is a colimit over objects $(b, m)$ of
$\elt{M(\dashbk, a)}$, the category of elements of $M(\dashbk, a):
\scat{B}^\op \go \Set$.  Both the coend and colimit formulations continue to
make sense when $X$ takes values not in $\Set$ but in some other category
$\cat{E}$ with small colimits; thus, 
\[
X: \scat{B} \go \cat{E}
\diagspace 
\textrm{and}
\diagspace
M: \scat{B}^\op \times \scat{A} \go \Set
\]
give rise to $M \otimes X: \scat{A} \go \cat{E}$.  This product $M \otimes X$
can be characterized by a universal property: for any functor $Z: \scat{A} \go
\cat{E}$, the natural transformations $\psi: M \otimes X \go Z$ are in natural
bijection with the families 
\[
\Bigl(
X(b) \goby{\psi_m} Z(a)
\Bigr)_{b \gobymod{m} a}
\]
of maps in $\cat{E}$ (indexed over all $b \in \scat{B}$, $a \in \scat{A}$ and
$m \in M(b, a)$) such that $\psi_{fmg} = (Zf) \of \psi_m \of (Xg)$ for all
\[
b' \goby{g} b \gobymod{m} a \goby{f} a'.
\]

In the even more special case $\scat{C} = \scat{A} = \One$, the tensor product
construction gives for each pair of functors $X: \scat{B} \go \Set$, $Y:
\scat{B}^\op \go \Set$ a set
\[
Y \otimes X 
=
\int^b Y(b) \times X(b)
=
\left( 
\sum_{b \in \scat{B}} Y(b) \times X(b) 
\right) 
/\!\sim\ \:
=
\Colt{(b, y) \in \elt{Y}} X(b).
\]
(Again, the construction also makes sense for $X: \scat{B} \go \cat{E}$ and
$Y: \scat{B}^\op \go \Set$, for suitable categories $\cat{E}$; then $Y \otimes
X \in \cat{E}$.)  In fact, the general construction can be written in terms of
this very special case: for modules $M$ and $N$ as above, and $a \in
\scat{A}$, $c \in \scat{C}$, we have functors 
\[
N(c, \dashbk): \scat{B} \go \Set,
\diagspace
M(\dashbk, a): \scat{B}^\op \go \Set,
\]
and then
\[
(M \otimes N)(c, a) = M(\dashbk, a) \otimes N(c, \dashbk).
\]

The notion of commutative diagram in a category $\scat{A}$ can be extended to
include elements of a module $M: \scat{A} \gomod \scat{A}$.  For instance, the
diagram
\[
\begin{diagram}
a_2	&\rMod^{m_2}	&a_1	&\rMod^{m_1}	&a_0	\\
\dTo<{f_2}&		&\dTo>{f_1}&		&\dTo>{f_0}\\
a'_2	&\rMod_{m'_2}	&a'_1	&\rMod_{m'_1}	&a'_0	\\
\end{diagram}
\]
is said to \demph{commute} if $m'_2 f_2 = f_1 m_2$ and $m'_1 f_1 = f_0
m_1$.  We never attempt to compose paths containing more than one crossed
arrow $\gomod$.


\section{Appendix: Solvability}
\label{app:solv}

Here we finish the proof of:
\begin{thm}[Existence of universal solution]
\label{thm:existenceofuniversalsolution}
Let $(\scat{A}, M)$ be an equational system.  The following are equivalent:
\begin{enumerate}
\item	\label{item:ex-S}
$(\scat{A}, M)$ satisfies the solvability condition \So\
of~\S\ref{sec:construction} 
\item	\label{item:ex-Top}
$(\scat{A}, M)$ has a universal solution in $\Top$
\item	\label{item:ex-Set}
$(\scat{A}, M)$ has a universal solution in $\Set$.
\end{enumerate}
In that case, the universal solution in $\Set$ is the underlying
coalgebra in $\Set$ of the universal solution in $\Top$.
\end{thm}

We proved \bref{item:ex-S}$\implies$\bref{item:ex-Top}
in~\S\ref{sec:Top-proofs}.  We proved
\bref{item:ex-Top}$\implies$\bref{item:ex-Set}, and the final sentence, as
Proposition~\ref{propn:coalgebrasinTopandSet}.  It remains to prove
\bref{item:ex-Set}$\implies$\bref{item:ex-S}.

Fix an equational system $(\scat{A}, M)$.  In this appendix, `$M$-coalgebra'
means `$M$-coalgebra in $\Set$'.  We constructed a functor $I: \scat{A} \go
\Set$ in~\S\ref{sec:construction}; it is defined regardless of whether \So\ 
holds.

\begin{lemma}
\label{lemma:S-nd}
The following conditions on $(\scat{A}, M)$ are equivalent:
\begin{enumerate}
\item \label{part:solv-S}
$(\scat{A}, M)$ satisfies \So
\item \label{part:solv-nd}
the functor $I: \scat{A} \go \Set$ is nondegenerate
\item \label{part:solv-transf}
there exist a nondegenerate functor $J: \scat{A} \go \Set$ and a natural
transformation $I \go J$.
\end{enumerate}
\end{lemma}

\begin{proof}
We proved (\ref{part:solv-S})$\implies$(\ref{part:solv-nd}) as
Proposition~\ref{propn:setInondegenerate}, and
(\ref{part:solv-nd})$\implies$(\ref{part:solv-transf}) is trivial.  For
(\ref{part:solv-transf})$\implies$(\ref{part:solv-S}), let $\gamma$ be a
natural transformation from $I$ to a nondegenerate functor $J: \scat{A} \go
\Set$.  By definition, $I = \Pi_0 \of \catI$, so $\gamma$ corresponds under
the adjunction
\[
\Cat \oppair{\Pi_0}{D} \Set
\]
to a natural transformation $\ovln{\gamma}: \catI \go D \of J$.  This in turn
corresponds to a functor $F: \elt{\catI} \go \elt{D \of J} \iso \elt{J}$
making the following triangle commute:
\[
\begin{diagram}
\elt{\catI}	&		&\rTo^F	&		&\elt{J}	\\
		&\rdTo<\pr	&	&\ldTo>\pr	&		\\
		&		&\scat{A}&		&		\\
\end{diagram}
\]
where $\pr$ denotes a projection.  Now, condition \So\ says that if $\scat{K}$
is either of the categories $\littlepullback$ or $\littleequalizer$, then for
any functor $G: \scat{K} \go \elt{\catI}$, the composite $\pr \of G$ admits a
cone.  But given such a $G$, nondegeneracy of $J$ implies that $F\of G$ admits
a cone, so $\pr \of F \of G = \pr \of G$ admits a cone, as required.  \done
\end{proof}

To prove (\ref{item:ex-Set})$\implies$(\ref{item:ex-Top}) of
Theorem~\ref{thm:existenceofuniversalsolution}, we will have to exploit the
existence of a terminal object in the category of $M$-coalgebras; hence we
will need a good supply of objects of that category.  

For each complex $(a_\blb, m_\blb)$, we construct a representable-type
coalgebra.  Its underlying functor is 
\[
H^{(a_\sblb, m_\sblb)}
=
\sum_{n\in\nat}
\scat{A} (a_n, \dashbk):
\scat{A} \go \Set.
\]
Any representable functor is flat, so $H^{(a_\sblb, m_\sblb)}$ is
nondegenerate by Theorem~\ref{thm:componentwiseflatness}.  Also
\[
(M \otimes H^{(a_\sblb, m_\sblb)})(b)
\iso
\sum_{n\in\nat} (M \otimes \scat{A}(a_n, \dashbk))(b)
\iso	
\sum_{n\in\nat} M(a_n, b),
\]
so an $M$-coalgebra structure on $H^{(a_\sblb, m_\sblb)}$ amounts to a
natural transformation
\[
\sum_{n\in\nat} \scat{A}(a_n, \dashbk)
\go
\sum_{n\in\nat} M(a_n, \dashbk).
\]
There is a unique such transformation sending $1_{a_n}$ to $m_{n+1} \in
M(a_{n+1}, a_n)$ for each $n\in\nat$; let $\theta^{(a_\sblb, m_\sblb)}$
be the corresponding coalgebra structure on $H^{(a_\sblb, m_\sblb)}$.

This defines an $M$-coalgebra $(H^{(a_\sblb, m_\sblb)}, \theta^{(a_\sblb,
m_\sblb)})$ for each object $(a_\blb, m_\blb)$ of $\elt{\catI}$.  Moreover, any map $f_\blb: (a_\blb, m_\blb) \go (a'_\blb, m'_\blb)$ in
$\elt{\catI}$ induces a map
\[
H^{(a'_\sblb, m'_\sblb)}
=
\sum_{n\in\nat} \scat{A}(a'_n, \dashbk)
\goby{\sum f_n^*}
\sum_{n\in\nat} \scat{A}(a_n, \dashbk)
=
H^{(a_\sblb, m_\sblb)}
\]
respecting the coalgebra structures.  So we have a functor
\[
(H^\blb, \theta^\blb):
\elt{\catI}^\op
\go
\Coalg{M}{\Set}.
\]

Having defined the representable-type coalgebras, we prove a Yoneda-type
lemma.  

Let $(X, \xi)$ be an $M$-coalgebra.  For each complex $(a_\blb, m_\blb)$,
write $(X, \xi)(a_\blb, m_\blb)$ for the set of resolutions along $(a_\blb,
m_\blb)$ in $(X, \xi)$, that is, sequences $(x_n \in X(a_n))_{n \in \nat}$
such that $\xi_{a_n}(x_n) = m_{n+1} \otimes x_{n+1}$ for all $n$.  This
defines a functor $(X, \xi): \elt{\catI} \go \Set$.

\begin{lemma}[`Yoneda']
\label{lemma:Yoneda}
There is a bijection
\[
\Coalg{M}{\Set}
\left(
(H^{(a_\sblb, m_\sblb)}, \theta^{(a_\sblb, m_\sblb)}),
(X, \xi)
\right)
\iso
(X, \xi)(a_\blb, m_\blb)
\]
natural in $(a_\blb, m_\blb) \in \elt{\catI}$ and $(X, \xi) \in
\Coalg{M}{\Set}$.  If $x \in X(a_0)$ then the maps $(H^{(a_\sblb, m_\sblb)},
\theta^{(a_\sblb, m_\sblb)}) \go (X, \xi)$ mapping $1_{a_0}$ to $x$ correspond
to the resolutions of $x$ along $(a_\blb, m_\blb)$.
\end{lemma}

\begin{proof}
By the standard Yoneda Lemma, a natural transformation $\alpha: H^{(a_\sblb,
m_\sblb)} \go X$ amounts to a sequence $(x_n)_{n\in\nat}$ with $x_n \in
X(a_n)$.  It is a map of coalgebras if and only if
\[
\begin{diagram}[height=6ex]
\sum_{n\in\nat} \scat{A}(a_n, \dashbk)	&
\rTo^\alpha		&
X	\\
\dTo<{\theta^{(a_\sblb, m_\sblb)}}	&
		&
\dTo>\xi\\
\sum_{n\in\nat} M(a_n, \dashbk)		&
\rTo_{M \otimes \alpha}	&
M \otimes X\\
\end{diagram}
\]
commutes, if and only if this diagram commutes when we take $1_{a_n}$ at the
top-left corner for every $n\in\nat$, if and only if $\xi(x_n) =
m_{n+1} \otimes x_{n+1}$ for all $n\in\nat$.  A coalgebra map $(H^{(a_\sblb,
m_\sblb)}, \theta^{(a_\sblb, m_\sblb)}) \go (X, \xi)$ therefore amounts to a
sequence $(x_n)_{n\in\nat}$ satisfying $\xi(x_n) = m_{n+1} \otimes x_{n+1}$
for all $n$, that is, a resolution along $(a_\blb, m_\blb)$ in $(X, \xi)$.
This establishes the bijection; naturality follows from the naturality in the
standard Yoneda Lemma.  \done
\end{proof}

We have met just one other canonical $M$-coalgebra: $(\oba\catI, \iota)$,
constructed in~\S\ref{sec:construction}.  (Recall that $M$-coalgebras are
nondegenerate by definition; $\oba\catI$ is nondegenerate whether or not \So\
holds.) 

\begin{propn}[Tautological map]
\label{propn:tautologicalmap}
For each complex $(a_\blb, m_\blb)$ there is a canonical map of
$M$-coalgebras  
\[
\kappa^{(a_\sblb, m_\sblb)}:
(H^{(a_\sblb, m_\sblb)}, \theta^{(a_\sblb, m_\sblb)})
\go
(\oba \catI, \iota),
\]
satisfying $\kappa^{(a_\sblb, m_\sblb)}(1_{a_0}) = (a_\blb, m_\blb)$.
\end{propn}

\begin{proof}
Every complex $(a_\blb, m_\blb)$, regarded as an element of $\oba\catI(a_0)$,
has a canonical resolution in $\oba\catI$.  By Lemma~\ref{lemma:Yoneda}, the
corresponding map $\kappa^{(a_\sblb, m_\sblb)}$ of coalgebras sends $1_{a_0}$
to $(a_\blb, m_\blb)$.  \done
\end{proof}

\begin{prooflike}{Proof of Theorem~\ref{thm:existenceofuniversalsolution}}
It remains to prove \bref{item:ex-Set}$\implies$\bref{item:ex-S}.

Suppose that $(\scat{A}, M)$ has a universal solution $(J, \gamma)$ in $\Set$.
Then there is a unique map $\beta: (\oba\catI, \iota) \go (J, \gamma)$ of
$M$-coalgebras.  I claim that the natural transformation $\beta$ can be
factorized as
\[
\begin{diagram}
\oba \catI 		&		&
\rTo^\beta		&
			&J,		\\
			&\rdQt<\pi	&
			&
\ruGet>{\ovln{\beta}}	&		\\
			&		&
I			&
			&		\\
\end{diagram}
\]
where $\pi$ is the usual projection~(\S\ref{sec:construction}).  Equivalently,
for each $a \in \scat{A}$ the function $\beta_a: \oba\catI(a) \go J(a)$ is
constant on connected-components of $\catI(a)$; equivalently, if
$f_\blb: (a_\blb, m_\blb) \go (b_\blb, p_\blb)$ in $\catI(a)$ then
$\beta_a(a_\blb, m_\blb) = \beta_a(b_\blb, p_\blb)$.  Indeed, given such an
$f_\blb$, there are coalgebra maps
\[
\begin{diagram}[width=6em]
(H^{(a_\sblb, m_\sblb)}, \theta^{(a_\sblb, m_\sblb)})	&
	&	&	&	\\
\uTo<{f_\sblb^*}	&
\rdTo(2,1)^{\kappa^{(a_\sblb, m_\sblb)}}		&
(\oba\catI, \iota)	&
\rTo^\beta	&
(J, \gamma),    \\
(H^{(b_\sblb, p_\sblb)}, \theta^{(b_\sblb, p_\sblb)})	&
\ruTo(2,1)_{\kappa^{(b_\sblb, p_\sblb)}}		&
	&	&	\\
\end{diagram}
\]
and $\beta \of \kappa^{(a_\sblb, m_\sblb)} \of f_\blb^* = \beta \of
\kappa^{(\beta_\sblb, p_\sblb)}$ by terminality of $(J, \gamma)$.  (The
triangle is not asserted to commute.)  But
\begin{eqnarray*}
\kappa^{(a_\sblb, m_\sblb)} f_\blb^* (1_a)	&
=	&
\kappa^{(a_\sblb, m_\sblb)} (1_a \of f_0)
=
\kappa^{(a_\sblb, m_\sblb)} (1_a)
=
(a_\blb, m_\blb),
\\
\kappa^{(b_\sblb, p_\sblb)} (1_a)	&
=	&
(b_\blb, p_\blb),
\end{eqnarray*}
so $\beta_a (a_\blb, m_\blb) = \beta_a (b_\blb, p_\blb)$, as required.  This
proves the claim.  It then follows from Lemma~\ref{lemma:S-nd} that
$(\scat{A}, M)$ satisfies \So.
\done
\end{prooflike}


\section{Appendix: Realizability}
\label{app:admitting}

Here we describe the class of topological spaces that can be characterized by
some equational system---those that are realizable, in the sense of
Definition~\ref{defn:realizable}.  We showed in~\S\ref{sec:Top-proofs} that
all such spaces are compact and metrizable, so the question is: which compact
metrizable spaces are realizable?  Perhaps surprisingly, the answer turns out
to be: all of them (Theorem~\ref{thm:selfsimilarspaces}).  

This theorem is less important than it might appear.  It characterizes those
spaces that admit at least one recursive decomposition, but the same space may
admit several such decompositions (Examples~\ref{eg:rec-terminal},
\ref{eg:rec-Freyd}, \ref{eg:rec-bary},~\ref{eg:rec-edgewise}).  Compare the
result that every nonempty set admits at least one group structure, which does
not play an important role in group theory.

It is crucial to this theorem that in the definition of equational system
$(\scat{A}, M)$, there may be infinitely many `equations' (objects of
$\scat{A}$)---even though each individual equation involves only finitely many
spaces.  In the proof, there is infinite regress: the given space $S$ is
decomposed into subspaces $S_i$; each $S_i$ is decomposed into subspaces
$S_{ij}$, and so on.  Our theorem is one of many stating that 
the topology of metric spaces can be probed effectively by countable
methods: compare, for instance, the fact that a metric space is compact if and
only if it is sequentially compact.

There is a similar theorem for discretely realizable spaces.  We already know
that every such space is compact, metrizable and totally disconnected
(Example~\ref{eg:construction-discrete});
Theorem~\ref{thm:discretelyselfsimilarspaces} states the converse.

The analogous questions for \emph{finite} equational systems are unanswered.
Since, up to homeomorphism, there are uncountably many compact metrizable
spaces but only countably many finitely realizable spaces, not every compact
metrizable space is finitely realizable.  It can also be shown that any 
space realizable by a finite discrete equational system has finite
Cantor--Bendixson rank.  

Here is our main theorem.

\begin{thm}[Realizability]
\label{thm:selfsimilarspaces}
A topological space is realizable if and only if it is compact and
metrizable.
\end{thm}

The idea of the proof is as follows.  Let $S$ be a compact metrizable space.
Cover $S$ by two closed subsets $V_1$ and $V'_1$.  Then $S = V_1 \cup V'_1$;
hence, $S$ is the pushout
\[
S = V_1 +_{V''_1} V'_1
\]
where $V''_1 = V_1 \cap V'_1$.  Next, cover $S$ by a different pair $V_2$,
$V'_2$ of closed subsets and write $V''_2 = V_2 \cap V'_2$: then
\begin{eqnarray*}
V_1	&=	&
(V_1 \cap V_2) +_{(V_1 \cap V''_2)} (V_1 \cap V'_2),	\\
V'_1	&=	&
(V'_1 \cap V_2) +_{(V'_1 \cap V''_2)} (V'_1 \cap V'_2),	\\
V''_1	&=	&
(V''_1 \cap V_2) +_{(V''_1 \cap V''_2)} (V''_1 \cap V'_2).	
\end{eqnarray*}
Continue in this way to obtain a countable equational system.  Compact
metrizability of $S$ means that the covers can be chosen to penetrate all
of its structure, and the universal solution $I$ is then made up of the
space $S$, the various covering subsets and their intersections, and the
inclusions between them. 

\bigskip

Given covers $\cov{W}$ and $\cov{V}$ of a space, $\cov{W}$ is said to 
\demph{refine} $\cov{V}$ if for all $W \in \cov{W}$ there exists $V \in
\cov{V}$ such that $W \sub V$.  

\begin{defn}
\label{defn:sep-seq}
Let $S$ be a topological space.  A \demph{separating sequence} for $S$ is a
sequence $(\cov{V}_n)_{n\in\nat}$ of finite closed covers of $S$ such that
\begin{enumerate}
\item \label{part:sep-seq-ref}
$\cov{V}_0 = \{S\}$, and for all $n\in\nat$, $\cov{V}_{n+1}$ refines
$\cov{V}_n$
\item \label{part:sep-seq-sep}
for all $s, t \in S$ with $s \neq t$, there exists $n \in \nat$ such that for
all $V \in \cov{V}_n$, $\{s, t\} \not\sub V$
\item \label{part:sep-seq-int-same}
for all $n \in \nat$, for all $V, V' \in \cov{V}_n$, we have $V \cap V' \in
\cov{V}_n$ 
\item \label{part:sep-seq-int-diff}
for all $n \in \nat$, for all $V \in \cov{V}_n$ and $W \in \cov{V}_{n+1}$, we
have $V \cap W \in \cov{V}_{n + 1}$.
\end{enumerate}
\end{defn}

The importance of condition~(\ref{part:sep-seq-int-diff}) is that any element
$V \in \cov{V}_n$ is covered exactly by elements of $\cov{V}_{n + 1}$: indeed,
\begin{equation}
\label{eq:covers}
V 
= 
V \cap \bigcup_{W \in \cov{V}_{n+1}} W 
=
\bigcup_{W \in \cov{V}_{n+1}} V \cap W 
=
\bigcup_{X \in \cov{V}_{n+1}: X \sub V} X.
\end{equation}

\begin{lemma}
\label{lemma:separatingsequence}
Every compact metrizable space admits a separating sequence.
\end{lemma}

\begin{proof}
Let $S$ be a compact metrizable space.  Then $S$ has a countable basis
$(U_n)_{n \geq 1}$ of open sets.  For each $n\geq 1$, let
\[
\cov{W}_n 
=
\{ \ovln{U_n}, \ \ 
S\without U_n, \ \ 
\ovln{U_n} \cap (S \without U_n) \}
\]
where $\ovln{U_n}$ is the closure of $U_n$.  Then $(\cov{W}_n)_{n \geq 1}$ is
a sequence of finite closed covers.  It satisfies
conditions~(\ref{part:sep-seq-sep}) and~(\ref{part:sep-seq-int-same}) of
Definition~\ref{defn:sep-seq}, with $\nat$ changed to $\posint$: 
condition~(\ref{part:sep-seq-int-same}) is obvious, and
for~(\ref{part:sep-seq-sep}), if $s \neq t$ then we may find $n \geq 1$ such
that $s \in U_n$ but $t \not\in \ovln{U_n}$, and then there is no $W \in
\cov{W}_n$ for which $s, t \in W$.  It does not necessarily
satisfy~(\ref{part:sep-seq-ref}) or~(\ref{part:sep-seq-int-diff}); but now
define, for each $n \in \nat$,
\[
\cov{V}_n
=
\{ 
W_1 \cap \cdots \cap W_n 
\such
W_1 \in \cov{W}_1, \ldots, W_n \in \cov{W}_n 
\}
\]
(understood as $\cov{V}_0 = \{ S \}$ when $n = 0$).  From the properties of
$(\cov{W}_n)_{n \geq 1}$ stated, it is easily shown that $(\cov{V}_n)_{n \in
\nat}$ is a separating sequence for $S$.
\done
\end{proof}

Fix a compact metrizable space $S$ with a separating sequence
$(\cov{V}_n)_{n\in\nat}$.

We define an equational system $(\scat{A}, M)$.  Recall that a poset can be
regarded as a category in which each hom-set has at most one element: there is
a map $a' \go a$ just when $a' \leq a$.  For each $n \geq 0$, let $\scat{A}_n$
be the set of nonempty elements of $\cov{V}_n$, ordered by inclusion.  Let
$\scat{A}$ be the coproduct $\sum_{n\in\nat} \scat{A}_n$, so that an object of
$\scat{A}$ is a pair $(n, V)$ with $n \in \nat$ and $\emptyset \neq V \in
\cov{V}_n$.  Define a module $M: \scat{A} \gomod \scat{A}$ by
\[
M((p, W), (n, V))
=
\left\{
\begin{array}{ll}
1               &\textrm{if } p = n + 1 \textrm{ and } W \sub V \\
\emptyset       &\textrm{otherwise}.
\end{array}
\right.
\]
Thus, $M$ is also `posetal': there is at most one sector from any object of
$\scat{A}$ to any other.

\begin{lemma}
\label{lemma:cmssss}
$(\scat{A}, M)$ is an equational system.
\end{lemma}

\begin{proof}
Finiteness of $M$ follows from finiteness of each cover $\cov{V}_n$.  For
nondegeneracy, we verify conditions~\cstyle{ND1} and~\cstyle{ND2}.
Condition~\cstyle{ND2} is trivial since $\scat{A}$ is a poset.
For~\cstyle{ND1}, take a square of solid arrows
\[
\begin{diagram}[width=3em,height=2em]
	&		&(n+1,W)&		&	\\
	&\ldMod(2,4)	&\dModget&\rdMod(2,4)   &	\\
	&		&(n, V\cap V')&		&	\\
	&\ldGet         &	&\rdGet         &	\\
(n, V)	&		&	&		&(n, V')\\
	&\rdTo          &	&\ldTo  	&	\\
	&		&(n, V'')&		&	\\
\end{diagram}
\]
in $(\scat{A}, M)$, so that $W \in \cov{V}_{n + 1}$, $V, V', V'' \in
\cov{V}_n$, and $W \sub V \cap V'$.  Since $W \neq \emptyset$, we have $V \cap
V' \neq \emptyset$, that is, $V \cap V' \in \scat{A}_n$.  Hence the diagram
can be filled in with the dotted arrows shown.  
\done
\end{proof}

Define a functor $J: \scat{A} \go \Top$ on objects by $J(n, V) = V$
(topologized as a subspace of $S$) and on morphisms by the evident inclusions.
We have, for each object $(n, V)$ of $\scat{A}$,
\begin{equation}
\label{eq:admit-colim}
(M \otimes J)(n, V) 
=
\Colt{W \in \scat{A}_{n + 1}: W \sub V} W,
\end{equation}
where the right-hand side is a colimit over a full subcategory of $\scat{A}_{n
+ 1}$.  Hence there is a (continuous) map
\begin{equation}
\label{eq:admit-psi}
\psi_{(n, V)}: 
(M \otimes J)(n, V) \go J(n, V) = V
\end{equation}
whose $W$-component is the inclusion $W \rIncl V$.  This defines a natural
transformation $\psi: M \otimes J \go J$.  

\begin{lemma}	
\label{lemma:fixedpoint}
$J: \scat{A} \go \Top$ is a nondegenerate functor, and $\psi: M \otimes J
\go J$ is an isomorphism.
\end{lemma}

\begin{proof}
For the first part, each space $J(n, V) = V$ is a closed subspace of the
compact Hausdorff space $S$, and therefore compact Hausdorff.  So it is enough
to prove that the underlying $\Set$-valued functor of $J$ is nondegenerate.
As in the proof of Lemma~\ref{lemma:cmssss}, condition~\cstyle{ND2} is
trivial, and condition~\cstyle{ND1} is an easy check.

For the second part, we have to show that for every object $(n, V)$ of
$\scat{A}$, the map $\psi_{(n, V)}$ of~(\ref{eq:admit-psi}) is a
homeomorphism.  Its domain is a finite colimit of compact spaces, hence
compact, and its codomain is Hausdorff, so it suffices to show that it is a
bijection.  

Surjectivity follows immediately from~(\ref{eq:covers}).

For injectivity, first note that an element of the
colimit~(\ref{eq:admit-colim}) is an equivalence class of pairs $(W, w)$ where
$w \in W \in \cov{V}_{n+1}$ and $W \sub V$.  The equivalence relation
$\sim$ is generated as follows: if $X, W \in \cov{V}_{n+1}$ with $x \in X \sub
W \sub V$ then $(X, x) \sim (W, x)$.  Writing $[\ ]$ for equivalence class,
we have $\psi_{(n, V)}([W, w]) = w$.  Now suppose that  
\[
\psi_{(n, V)}([W, w]) = \psi_{(n, V)}([W', w']).
\]
Then $w = w'$, so $w \in W \cap W' \in \cov{V}_{n+1}$ with $W \cap W' \sub
V$.  Hence $(W \cap W', w)$ is a pair of the relevant type, and
\[
(W, w) \sim (W \cap W', w = w') \sim (W', w'),
\]
as required.
\done
\end{proof}

\begin{propn}
\label{propn:topounivsoln}
$(J, \psi^{-1})$ is the universal solution of $(\scat{A}, M)$ in $\Top$.
\end{propn}

\begin{proof}
We verify condition~\bref{item:pr-notmet} of the Precise Recognition
Theorem~(\ref{thm:preciserecognition}).  Each space $J(n, V) = V$ is compact,
and nonempty by definition of $\scat{A}$, so it only remains to check the main
part of the condition.  

A complex in $(\scat{A}, M)$ is of the form
\[
\cdots \gobymod{m_2} (n + 1, V_{n + 1}) \gobymod{m_1} (n, V_n)
\]
where $V_r \in \scat{A}_r$ and $V_n \supseteq V_{n + 1} \supseteq \cdots$.  We
have
\[
\psi_{m_1} \of \cdots \of \psi_{m_r}
=
(V_{n + r} \rIncl \cdots \rIncl V_n)
=
(V_{n + r} \rIncl V_n),
\]
so $V^J_{m_1, \ldots, m_r} = V_{n + r}$.  Hence
\[
\bigcap_{r \in \nat} V^J_{m_1, \ldots, m_r} 
=
\bigcap_{i \geq n} V_i.
\]
Suppose that $s, t \in \bigcap_{i \geq n} V_i$.  By
condition~(\ref{part:sep-seq-ref}) of Definition~\ref{defn:sep-seq}, there
exist $V_{n-1} \in \cov{V}_{n-1}, \ldots, V_0 \in \cov{V}_0$ such that 
\[
V_n \sub V_{n-1} \sub \cdots \sub V_0,
\]
and then $s, t \in \bigcap_{i \in \nat} V_i$.  So by
condition~(\ref{part:sep-seq-sep}), $s = t$.  
\done
\end{proof}

\begin{prooflike}{Proof of Theorem~\ref{thm:selfsimilarspaces}}
Let $S$ be a compact metrizable space and construct $(\scat{A}, M)$ and $(J,
\psi)$ as above.  If $S$ is nonempty then $(0, S)$ is an object of $\scat{A}$,
and $J(0, S) = S$.  On the other hand, $\emptyset$ is the universal solution
of the equational system $(\One, \emptyset)$ (Example~\ref{eg:rec-terminal}).
\done
\end{prooflike}

\begin{thm}[Discrete realizability]
\label{thm:discretelyselfsimilarspaces}
The following conditions on a topological space $S$ are equivalent:
\begin{enumerate}
\item	\label{item:dss-dss}
$S$ is discretely realizable
\item	\label{item:dss-seq} 
$S$ is the limit of some sequence $(\cdots \go S_2 \go S_1)$ of finite
discrete spaces
\item	\label{item:dss-ctbl}
$S$ is the limit of some countable diagram of finite discrete spaces
\item	\label{item:dss-topo}
$S$ is compact, metrizable, and totally disconnected.
\end{enumerate}
\end{thm}

\begin{proof}
\paragraph*{\bref{item:dss-dss}$\implies$\bref{item:dss-seq}}
Let $(\scat{A}, M)$ be a discrete equational system.  The universal
solution is $\oba\catI$, and each space $\oba\catI(a)$ is the limit of the
sequence of the finite discrete spaces $\oba\catI_n(a)$
(Example~\ref{eg:construction-discrete}). 

\paragraph*{\bref{item:dss-seq}$\implies$\bref{item:dss-ctbl}}
Trivial.

\paragraph*{\bref{item:dss-ctbl}$\implies$\bref{item:dss-topo}}
Compact metrizable spaces are the same as compact Hausdorff spaces that are
second countable (have a countable basis of open sets).  The classes of
compact Hausdorff spaces and totally disconnected spaces are closed under all
limits, and the class of second countable spaces is closed under countable
limits.

\paragraph*{\bref{item:dss-topo}$\implies$\bref{item:dss-dss}}
For this we adapt the proof of Theorem~\ref{thm:selfsimilarspaces}.  We may
choose for $S$ a basis $(U_n)_{n\geq 1}$ of open sets that are also 
closed, by Theorem~II.4.2 of~\cite{Joh}.  The separating sequence
$(\cov{V}_n)_{n\in\nat}$ constructed in
Lemma~\ref{lemma:separatingsequence} then has the property that each cover
$\cov{V}_n$ is a partition of $S$.  The resulting category $\scat{A}$
is therefore discrete, and the result follows.
\done
\end{proof}

For example, the underlying topological space of the absolute Galois group
$\mathrm{Gal}(\ovln{\mathbb{Q}}/\mathbb{Q})$ is a countable limit of finite
discrete spaces, so discretely realizable.

A measure of the power of the realizability theorems is that some classical
results of topology~\cite{Wil,HY} can be deduced.
Proposition~\ref{propn:Iaasaquotient} implies that every realizable space is a
topological quotient of a discretely realizable space; thus, every compact
metrizable space is a quotient of a totally disconnected compact metrizable
space.  On the other hand, it can be shown directly that every nonempty
discretely realizable space is a retract of the Cantor set.  It follows that
every totally disconnected compact metrizable space is a subspace of the
Cantor set, and that every nonempty compact metrizable space is a quotient of
the Cantor set.  Finally, it follows from
Proposition~\ref{propn:emptyorcantor} that every totally disconnected compact
metrizable space without isolated points is either empty or homeomorphic to
the Cantor set.  We have thus deduced the classical results characterizing the
closed subspaces, quotients and homeomorphism type of the Cantor set.
Detailed proofs can be found in~\cite{SS2}.


\end{document}